\documentclass[12pt]{amsart}
\usepackage{graphicx, amsfonts, amsmath, amssymb, cleveref, esvect, xcolor, nicematrix} 

\usepackage{tikz-cd}
\usetikzlibrary{cd}
\usetikzlibrary{patterns}

\theoremstyle{theorem}
\newtheorem{theorem}{Theorem}[section]

\newtheorem*{maintheorem}{Main Theorem}

\newtheorem{lemma}[theorem]{Lemma}
\newtheorem{proposition}[theorem]{Proposition}

\newtheorem{remark}[theorem]{Remark}

\theoremstyle{definition}

\newtheorem{example}[theorem]{Example}

\newtheorem{definition}[theorem]{Definition}
\newtheorem{notation}[theorem]{Notation}
\newtheorem{question}[theorem]{Question}

\numberwithin{equation}{section}

\title[Every Weak Perron Number is an End-Periodic Stretch Factor]{Every Weak Perron Number is an End-Periodic Stretch Factor}

\author[P. Hillen, M. Loving, and C. Wu]{Paige Hillen, Marissa Loving, and Chenxi Wu}

\date{\today}

\begin{document}

\begin{abstract}
    Given any weak Perron number $\lambda$, we construct an end-periodic homeomorphism $f:\Sigma\rightarrow \Sigma$ with Handel-Miller stretch factor equal to $\lambda$ where $\Sigma$ is a connected infinite-type surface with finitely many ends all accumulated by genus. 
\end{abstract}

\maketitle

\section{Introduction}

Given an end-periodic homeomorphism $f: \Sigma \to \Sigma$ its stretch factor $\lambda(f)$ is the spectral radius of the incidence matrix of the Markov decomposition defined by the Handel--Miller laminations associated to $f$. This stretch factor coincides with the stretch factor of a spun pseudo-Anosov representative of $f$ considered in \cite{LandryMinskyTaylor2023}; see \cite{LovingWu2025} for a brief discussion of this fact. Landry--Minsky--Taylor show that a spun pseudo-Anosov minimizes the stretch factor over all homotopic end-periodic maps and that the stretch factor of a spun pseudo-Anosov captures the growth rate of the intersection number of a simple closed curve with its image under iteration by that spun pseudo-Anosov. See also their related work in \cite{LandryMinskyTaylor2022}. More recently, Buckminster showed that a spun pseudo-Anosov representative minimizes the stretch factor in its homotopy class \cite{Buckminster}. 

On the other hand, very little is known about the number theoretic properties of end-periodic stretch factors, which has been a fruitful direction of research for pseudo-Anosov stretch factors. For example, in his announcement of the classification theorem for elements of the mapping class group, Thurston remarked that a pseudo-Anosov homeomorphism on a genus $g$ surface $\Sigma_g$ has stretch factor with algebraic degree bounded by $6g-6$, the dimension of the Teichm\"uller space of $\Sigma_g$ \cite{Thurston88}. This was eventually proved by Strenner who not only proved that Thurston's asserted bound can be realized, but also completely characterized the set of possible algebraic degrees of pseudo-Anosov elements of $\textrm{Mod}(\Sigma_g)$ \cite{Strenner2017}. A related result of Fried proves that every pseudo-Anosov stretch factor $\lambda$ is bi-Perron \cite{Fried1985}. Fried  conjectured that every bi-Perron number arises as the stretch factor of a pseudo-Anosov. This remains an open question. In contrast, we fully characterize the set of end-periodic stretch factors.

\begin{maintheorem} \label{thm:main}
    A number $\lambda$ is the stretch factor of an end-periodic homeomorphism $f: \Sigma \to \Sigma$ if and only if $\lambda$ is a weak Perron number.
\end{maintheorem}

One direction of this characterization follows from work of Cantwell-Conlon--Fenley in which they show that the stretch factor of an end-periodic homeomorphism is necessarily weak Perron. In particular, they show that every ``pseudo-Anosov" end-periodic homeomorphism defines a \emph{core dynamical system} that is topologically conjugate to a two-ended Markov shift of finite type \cite[Theorem~9.2]{CC-book} and hence, has stretch factor a Perron number. Stretch factors that are weak Perron, but not Perron, are obtained when this extra irreducibility requirement is relaxed.

Thus, our contribution lies in showing that every weak Perron number can be realized as the stretch factor of some end-periodic homeomorphism. Our proof is entirely constructive. Given an irreducible, non-negative, square integer matrix $M$ with spectral radius equal to $\lambda$ we use $M$ to build a surface $\Sigma$ and an end-periodic homeomorphism $f: \Sigma \rightarrow \Sigma$ with $\lambda(f) = \lambda$. This is the same type of construction as the one used by Baik--Rafiqi--Wu to build examples of pseudo-Anosov homeomorphisms with a given stretch factor \cite{BaikRafiqiWu2016}. However, they imposed relatively restrictive conditions on their initial matrix in order to ensure that the surface they obtained had finite type.

Finally, we note that our main theorem parallels Thurston's characterization of entropy in dimension one \cite{Thurston2014, DDHKOP}. In particular, he proved that every weak Perron number arises as the stretch factor of some outer automorphism of a free group. Indeed, the constructions both here and in Baik--Rafiqi--Wu \cite{BaikRafiqiWu2016} can be traced back to ideas of Thurston. 

\subsection*{Structure of the Paper} We begin by reviewing some of the basic structure associated to and end-periodic homeomorphism, defining the Handel-Miller stretch factor, and recalling some Perron-Frobenius theory in \Cref{background}. 

In \Cref{integers}, we show that every integer can be realized as the stretch factor of an end-periodic homeomorphism. While this result follows from the Main Theorem, this special case serves as a nice warm up for the general construction and hints at the overall structure of the remainder of the paper. 

In \Cref{linear map}, we describe how to use an irreducible, square, nonnegative integer matrix $M$ to construct a set of rectangles $\Sigma_0$, called the \textit{rectangles corresponding to $M$},  along with two partitions of $\Sigma_0$. We then construct a piecewise linear homeomorphism $f_0$ from the interior of one partition to the interior of the other partition. This map $f_0$ stretches in the horizontal direction by the spectral radius of $M$, $\lambda$, and contracts by $\lambda^{-1}$ in the vertical direction. 

In \Cref{edge maps}, we define four maps on the edges of the partitions of $\Sigma_0$ and show these edge maps have finitely many periodic points. 

In \Cref{infinite strips}, we build $\Sigma_1$, called the \textit{extended rectangles corresponding to $M$}, by gluing infinite strips to $\Sigma_0$ around neighborhoods of the periodic points of the edge maps. We then define the \textit{extended piecemap corresponding to $M$}, denoted $f_1$, from the interior of one partition of $\Sigma_1$ to the interior of a second partition of $\Sigma_1$. We define four new edge maps of $f_1$, and use these new edge maps to generate an equivalence relation $\sim$ on the boundary of $\Sigma_1$. The \textit{infinite 2-complex corresponding to $M$} is $\Sigma_2 := \Sigma_1 / \sim$. The section concludes with a proof that a slight modification of $\Sigma_2$ results in a surface $\Sigma$. We then show in \Cref{infinite-type} that a straightforward modification of this construction guarantees that $\Sigma$ is of infinite type. In \Cref{sec:end-periodic map}, we show $f_1$ extends to an end-periodic homeomorphism $f:\Sigma \rightarrow \Sigma$ with stretch factor equal to $\lambda$.

In \Cref{weak perron} we show $\Sigma$ is connected when $M$ is primitive and that another compatible modification of the construction guarantees $\Sigma$ is connected in the case that $M$ is irreducible but not primitive. Note that when $\lambda$ is a weak Perron number and not a Perron number, then it is the spectral radius of an irreducible but not primitive matrix $M$.

Finally, in \Cref{discussion}, we discuss some further directions of research in this area and highlight several open questions.

\subsection*{Acknowledgments} Part of this research was performed while Loving was in residence at the Mathematical Sciences Research Institute (MSRI), now becoming the Simons Laufer Mathematical Sciences Institute (SLMath), which is supported by the National Science Foundation (Grant No. DMS-2424139). We also gratefully acknowledge NSF support via DMS-2231286 (Loving) and the NSF RTG: Geometry, Group Actions, and Dynamics at Wisconsin via DMS-2230900 (Hillen). Wu was partially supported by a Simons Collaboration Grant (No. 850685). We are grateful for valuable discussions with Hyungryul Baik, Ahmad Rafiqi and John Hubbard. Finally, we thank Chi Cheuk Tsang for his insightful comments on an earlier draft of this paper. 

\section{Background}\label{background}

\subsection{End-periodic homeomorphisms}

Let $\Sigma$ denote an orientable surface of infinite type with finitely many ends all of which are non-planar. Note that we drop the typical assumption that $\Sigma$ is connected as it will be convenient at some points to consider a surface with multiple connected components. An \emph{end-periodic homeomorphism} $f: \Sigma \rightarrow \Sigma$ is a homeomorphism satisfying the following. There exists $m>0$ such that for each end $E$ of $\Sigma$, there is a neighborhood $U_E$ of $E$ so that either 
    \begin{enumerate}
        \item[(i)] $f^m(U_E) \subsetneq U_E$ and the sets ${f^{nm}(U_E)}_{n > 0}$ form a neighborhood basis of $E$; or
        \item[(ii)] $f^{-m}(U_E) \subsetneq U_E$ and the sets ${f^{-nm}(U_E)}_{n > 0}$ form a neighborhood basis of $E$.
    \end{enumerate}
We call such a neighborhood $U_E$ a \emph{nesting neighborhood}.

Fix an end-periodic homeomorphism $f: \Sigma \rightarrow \Sigma$. Let $U_+$ be a union of nesting neighborhoods of the attracting ends and let $U_-$ be a union of nesting neighborhoods of the repelling ends. We call the set $U_+$ (resp. $U_-$) a \emph{positive} (resp. \emph{negative}) \emph{ladder} of $f$. In addition, we say the ladder $U_{\pm}$ is \emph{tight} if $f^{\pm 1}(U_{\pm}) \subset U_{\pm}$ is a proper inclusion. A disjoint pair of positive and negative tight ladders $U_+, U_-$ define a compact subsurface $Y = \Sigma - (U_+ \cup U_-)$, which is called a \emph{core} for $f$.

Define \[\mathcal U_+ = \bigcup_{n \geq 0} f^{-n}(U_+)\text{ and } \mathcal U_- = \bigcup_{n \geq 0} f^n(U_-).\] We call $\mathcal U_+$ the \emph{positive escaping set} and $\mathcal U_-$ the \emph{negative escaping set}. 

An essential multiloop in $\mathcal U_+$ (resp. $\mathcal U_-$) is a \emph{positive} (resp. \emph{negative}) \emph{juncture} of $f$ if it is the boundary of a tight positive (resp. negative) ladder. Given a positive and negative juncture $j_+$ and $j_-$ of $f$ we define \[J^+ = \bigcup_{k \in \mathbb Z} f^k(j_+) \text{ and } J^- = \bigcup_{k \in \mathbb Z} f^k(j_-).\]

Fix a hyperbolic metric $X$ on $\Sigma$ and let $\mathcal J^{\pm}$ be the union of tightened geodesics for each curve in $J^{\pm}$. Note that $\overline{\mathcal J^{\pm}}$ is a geodesic lamination on $\Sigma$ since it is the union of disjoint simple geodesics. The \emph{Handel--Miller laminations} associated to $f$ are the geodesic laminations defined by \[\Lambda^+ = \overline{\mathcal J^-} - \mathcal J^- \text{ and } \Lambda^- = \overline{\mathcal J^+} - \mathcal J^+.\] Note that, by definition, $\mathcal J^-$ is disjoint from $\Lambda^+$ and $\mathcal J^+$ is disjoint from $\Lambda^-$.

The intersection of the Handel--Miller laminations $\Lambda^{+} \cap \Lambda^-$ defines a Markov decomposition of the complement of the escaping points, and we denote by $\lambda(f)$ the corresponding spectral radius of the incidence matrix of this Markov decomposition, which is the exponential of the topological entropy on the action of $f$ on $\Lambda^+\cap\Lambda^-$. We call $\lambda(f)$ the \emph{Handel--Miller stretch factor} of $f$. See \cite{CC-book} for a more detailed discussion. We will often refer to $\lambda(f)$ as simply the \emph{stretch factor} of $f$ or the \emph{end-periodic stretch factor} of $f$. 


\subsection{Perron--Frobenius Theory}

We will say that an $n \times n$ matrix $M = (m_{ij})$ with non-negative integer entries is \emph{irreducible} if for each $(i, j)$ there is an integer $k > 0$ such that $m_{ij}$ is positive. $M$ is \emph{primitive} if there exists a $k > 0$ so that every entry of $M^k$ is positive. The \emph{graph associated to $M$} is the directed graph with vertex set $\{v_i\}_{i=1}^n$ and $m_{ij}$ many directed edges from $v_j$ to $v_i$. 
The Perron--Frobenius theorem guarantees that such a matrix has a unique real eigenvalue that is largest in absolute value. This is called the Perron--Frobenius eigenvalue. The following theorem of Lind gives the converse \cite{Lind1984}. 

\begin{theorem}[Lind]\label{thm:Lind}
    For any real algebraic integer $\lambda > 0$ that is strictly larger in absolute value than its Galois conjugates, there exists a non-negative integer matrix with some power that is strictly positive and has $\lambda$ as an eigenvalue.
\end{theorem}

We will call such a $\lambda$, as in \Cref{thm:Lind}, a \emph{Perron} number. If a number has some power that is Perron we call it \emph{weak Perron}.

\section{Warm Up: Integers}\label{integers}
As a warm-up for the general case, we first prove that every integer $d \geq 2$ is the stretch factor of some end-periodic map on an infinite type surface. Most of the notation we introduce here will be used again in the general case.

\begin{theorem}
    For any integer $d \geq 2$, there is an infinite-type surface $\Sigma$ and an end-periodic homeomorphism $f: \Sigma \rightarrow \Sigma$ with stretch factor equal to $d$. 
\end{theorem}
\begin{proof} First, let $$\Sigma_0:=[0,1] \times [0,1] \subseteq \mathbb{R}^2$$ denote the unit square in $\mathbb{R}^2$. Partition $\Sigma_0$ into $d$ vertical strips of equal width and denote the interior of the strips $V^{(1)}_{1,k}$ for $k \in \{1, \dots, d\}$. Now, partition $\Sigma_0$ into $d$ horizontal strips of equal height and denote the interior of the strips $H^{(1)}_{1,k}$ for $k \in \{1, \dots, d\}$. Although the superscript $(1)$ on the $V$ and $H$ may appear extraneous, it is necessary in the general case so we introduce it now for consistency. Set \[V := \bigcup_{k=1}^d V^{(1)}_{1,k} \text{ and } H := \bigcup_{k=1}^d H^{(1)}_{1,k}.\]
Let $f_0: V \rightarrow H$ be the piecewise linear homeomorphism which maps each $V^{(1)}_{1,k}$ onto $H^{(1)}_{1,k}$ and maintains top/bottom and left/right orientation. Observe that $f_0$ stretches each $V^{(1)}_{1,k}$ horizontally  by a factor of $d$ and shrinks $V^{(1)}_{1,k}$ vertically by a factor of $d^{-1}$.  

\begin{figure}[htbp] 
\begin{center}

\tikzset{every picture/.style={line width=0.75pt}} 

\begin{tikzpicture}[x=0.75pt,y=0.75pt,yscale=-1,xscale=1]

\draw  [fill={rgb, 255:red, 207; green, 132; blue, 221 }  ,fill opacity=0.44 ][dash pattern={on 0.84pt off 2.51pt}] (20.65,20) -- (160.58,20) -- (160.58,168.37) -- (20.65,168.37) -- cycle ;
\draw  [draw opacity=0][fill={rgb, 255:red, 197; green, 69; blue, 223 }  ,fill opacity=0.87 ][dash pattern={on 0.84pt off 2.51pt}] (20.65,20) -- (48.64,20) -- (48.64,168.37) -- (20.65,168.37) -- cycle ;
\draw    (184.37,94.19) .. controls (231.75,94.77) and (195.19,94.78) .. (238.99,94.2) ;
\draw [shift={(240.34,94.19)}, rotate = 179.26] [color={rgb, 255:red, 0; green, 0; blue, 0 }  ][line width=0.75]    (10.93,-3.29) .. controls (6.95,-1.4) and (3.31,-0.3) .. (0,0) .. controls (3.31,0.3) and (6.95,1.4) .. (10.93,3.29)   ;
\draw  [draw opacity=0][fill={rgb, 255:red, 197; green, 69; blue, 223 }  ,fill opacity=0.51 ][dash pattern={on 0.84pt off 2.51pt}] (20.65,20) -- (132.6,20) -- (132.6,168.37) -- (20.65,168.37) -- cycle ;
\draw  [dash pattern={on 0.84pt off 2.51pt}]  (48.64,20) -- (48.64,168.37) ;
\draw  [dash pattern={on 0.84pt off 2.51pt}]  (132.6,20) -- (132.6,168.37) ;
\draw  [draw opacity=0][fill={rgb, 255:red, 197; green, 69; blue, 223 }  ,fill opacity=0.5 ][dash pattern={on 0.84pt off 2.51pt}] (20.65,20) -- (76.62,20) -- (76.62,168.37) -- (20.65,168.37) -- cycle ;
\draw  [dash pattern={on 0.84pt off 2.51pt}]  (76.62,20) -- (76.62,168.37) ;
\draw  [fill={rgb, 255:red, 207; green, 132; blue, 221 }  ,fill opacity=0.44 ][dash pattern={on 0.84pt off 2.51pt}] (400,21.63) -- (399.74,170.01) -- (259.8,169.72) -- (260.06,21.33) -- cycle ;
\draw  [draw opacity=0][fill={rgb, 255:red, 197; green, 69; blue, 223 }  ,fill opacity=0.87 ][dash pattern={on 0.84pt off 2.51pt}] (400,21.63) -- (399.95,51.31) -- (260.02,51.02) -- (260.07,21.33) -- cycle ;
\draw  [draw opacity=0][fill={rgb, 255:red, 197; green, 69; blue, 223 }  ,fill opacity=0.51 ][dash pattern={on 0.84pt off 2.51pt}] (400,21.63) -- (399.79,140.33) -- (259.85,140.04) -- (260.06,21.33) -- cycle ;
\draw  [dash pattern={on 0.84pt off 2.51pt}]  (399.95,51.3) -- (260.02,51.02) ;
\draw  [dash pattern={on 0.84pt off 2.51pt}]  (399.79,140.33) -- (259.85,140.04) ;
\draw  [draw opacity=0][fill={rgb, 255:red, 197; green, 69; blue, 223 }  ,fill opacity=0.5 ][dash pattern={on 0.84pt off 2.51pt}] (400,21.63) -- (399.9,80.98) -- (259.96,80.69) -- (260.07,21.33) -- cycle ;
\draw  [dash pattern={on 0.84pt off 2.51pt}]  (399.89,80.98) -- (259.96,80.69) ;

\draw (202.16,57.4) node [anchor=north west][inner sep=0.75pt]    {$f_{0}$};
\draw (23.64,86.1) node [anchor=north west][inner sep=0.75pt]  [font=\footnotesize]  {$V_{1,1}^{( 1)}$};
\draw (132.79,86.1) node [anchor=north west][inner sep=0.75pt]  [font=\footnotesize]  {$V_{1,d}^{( 1)}$};
\draw (91.61,80.65) node [anchor=north west][inner sep=0.75pt]    {$\dotsc $};
\draw (319.1,26.75) node [anchor=north west][inner sep=0.75pt]  [font=\footnotesize]  {$H_{1,1}^{( 1)}$};
\draw (324.69,145.45) node [anchor=north west][inner sep=0.75pt]  [font=\footnotesize]  {$H_{1,d}^{( 1)}$};
\draw (51.63,86.1) node [anchor=north west][inner sep=0.75pt]  [font=\footnotesize]  {$V_{1,2}^{( 1)}$};
\draw (340.36,94.53) node [anchor=north west][inner sep=0.75pt]  [rotate=-90.11]  {$\dotsc $};
\draw (319.1,56.43) node [anchor=north west][inner sep=0.75pt]  [font=\footnotesize]  {$H_{1,2}^{( 1)}$};

\end{tikzpicture}
\end{center}
\caption{The piecewise linear homeomorphism $f_0:V \rightarrow H$. }
    \label{fig:integer piece map}
\end{figure}
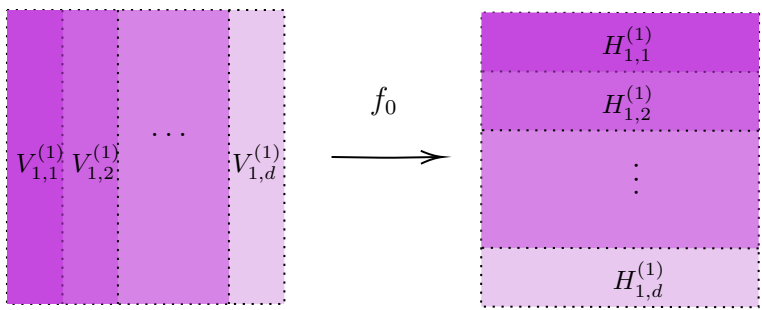

More concretely, for $(x,y) \in V^{(1)}_{1,k} := (\frac{k-1}{d},\frac{k}{d}) \times(0,1)$, we have 
$$f_0(x,y) := \Big{(}dx-k+1,1 - \frac{(k-1+y)}{d}\Big{)}.$$

\begin{remark} \normalfont At this point, we could begin to extend $f_0$ to a continuous map on the square by identifying intervals on the boundary in pairs as needed to ensure continuity. An expert in infinite-type translation surfaces might recognize this as a construction that would result in the Chamanara surface, a translation surface structure with underlying topological surface a Loch Ness monster \cite{Chamanara}. Instead of building this translation surface (whose underlying topological surface cannot admit an end-periodic homeomorphism) we instead blow up the accumulation points in the top left and bottom right corners of the square to infinite strips. This gives us a way of splitting the single end of the Loch Ness monster surface into a pair of ends, which will eventually become the attracting and repelling end of the end-periodic homeomorphism we construct as an extension of $f_0$.\end{remark}

Now, consider four infinite strips in the plane: \begin{align*} 
E_{L}(0,1)& := (-\infty, 0] \times [1-\frac{1}{d^2},1] \subseteq \mathbb{R}^2 \\
E_T(0,1) & := [0, \frac{1}{d^2}] \times [1, \infty) \subseteq \mathbb{R}^2 \\
E_R(1,0) & := [1, \infty) \times [0, \frac{1}{d^2}] \subseteq \mathbb{R}^2 \\
E_B(1,0) & := [1 - \frac{1}{d^2}, 1] \times [0, -\infty) \subseteq \mathbb{R}^2 
\end{align*}

Let $\Sigma_1$ equal $\Sigma_0$ together with the four infinite strips. Set \[(V_{1,1}^{(1)})' = \textrm{int}\left(\overline{V_{1,1}^{(1)}} \cup E_L(0,1) \cup E_T(0,1)\right),\] \[(V_{1,d}^{(1)})' = \textrm{int}\left(\overline{V_{1,d}^{(1)}} \cup E_R(1,0) \cup E_B(1,0)\right),\] and, for $2 \leq j \leq d-1$, set $(V_{1,j}^{(1)})'=V_{1,j}^{(1)}$. Finally, set \[V_1 = \bigcup_{j}(V_{1,j}^{(1)})'.\] Define each $(H_{1,j}^{(1)})'$ and $H_1$ similarly.  

Our goal is to define a homeomorphism $f_1:V_1 \rightarrow H_1$ which is equal to $f_0$ in $\Sigma_0 \subseteq \Sigma_1$, away from the points $(0,1)$ and $(1,0)$, and equal to a translation sufficiently deep into each of the infinite strips. To this end, we define a switch region around $(0,1)$ and $(1,0)$ where the behavior of $f_1$ switches from $f_0$ to translation. 

Let $W_1$ be a line segment joining $(0, 1-\frac{1}{d})$ to $(\frac{1}{d^2},1)$ which is contained in $[0,\frac{1}{d}] \times [1-\frac{1}{d},1]$. Define the \textit{switch region of $(0,1)$}, denoted $P_1$, be the region in the plane with boundary equal to
$$W_1 \cup (\{0\} \times [1-\frac{1}{d^2}, 2]) \cup ([0,\frac{1}{d^2}] \times \{2\}) \cup (\{\frac{1}{d^2}\} \times [1,2]).$$ 

Let $W_1' = \overline{f_0(W)}$. Let $P_1'$ be the region in the plane with boundary equal to
$$W_1' \cup (\{-1\} \times [1, 1-\frac{1}{d^2}]) \cup ([-1,\frac{1}{d}] \times \{1\}) \cup (\{-1\} \times [1- \frac{1}{d^2},1]).$$


Let $W_2$ be a line segment joining $(1-\frac{1}{d^2},0)$ to $(1,\frac{1}{d})$ which is contained in $[1-\frac{1}{d}, 1] \times [0,\frac{1}{d}]$. Define the \textit{switch region of $(1,0)$}, denoted $P_2$, be the region in the plane with boundary equal to
$$W_2 \cup (\{1\} \times [-1,\frac{1}{d}]) \cup ([1-\frac{1}{d^2},1] \times \{-1\}) \cup (\{1-\frac{1}{d^2}\} \times [0,-1]).$$ 

Let $W_2' = \overline{(f_0(W_2))}$. Let $P_2'$ be the region in the plane with boundary equal to
$$W_2' \cup (\{2\} \times [0, \frac{1}{d^2}]) \cup ([1-\frac{1}{d^2},2] \times \{0\}) \cup ( [1,2] \times \{1-\frac{1}{d^2}).$$


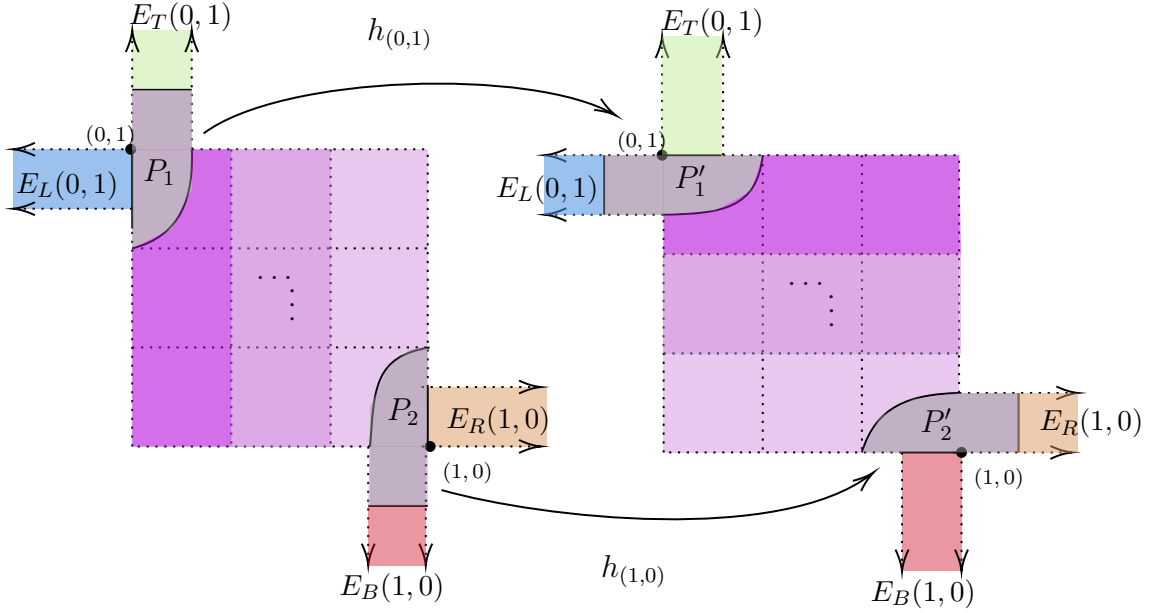
\begin{figure}[htbp] 
\begin{center}

\tikzset{every picture/.style={line width=0.75pt}} 

\begin{tikzpicture}[x=0.75pt,y=0.75pt,yscale=-1,xscale=1]

\draw  [draw opacity=0][fill={rgb, 255:red, 207; green, 132; blue, 221 }  ,fill opacity=0.44 ][line width=0.75]  (68,83) -- (217.02,83) -- (217.02,233) -- (68,233) -- cycle ;
\draw  [draw opacity=0][fill={rgb, 255:red, 184; green, 233; blue, 134 }  ,fill opacity=0.42 ][line width=0.75]  (68,23) -- (98,23) -- (98,53) -- (68,53) -- cycle ;
\draw  [draw opacity=0][fill={rgb, 255:red, 74; green, 144; blue, 226 }  ,fill opacity=0.55 ][line width=0.75]  (8,83) -- (68,83) -- (68,113) -- (8,113) -- cycle ;
\draw  [draw opacity=0][fill={rgb, 255:red, 207; green, 132; blue, 221 }  ,fill opacity=0.44 ][line width=0.75]  (188,233) -- (217.02,233) -- (217.02,263) -- (188,263) -- cycle ;
\draw  [draw opacity=0][fill={rgb, 255:red, 220; green, 149; blue, 91 }  ,fill opacity=0.49 ][line width=0.75]  (217.02,203) -- (277.02,203) -- (277.02,233) -- (217.02,233) -- cycle ;
\draw  [dash pattern={on 0.84pt off 2.51pt}]  (68,83) -- (68,25) ;
\draw [shift={(68,23)}, rotate = 90] [color={rgb, 255:red, 0; green, 0; blue, 0 }  ][line width=0.75]    (10.93,-3.29) .. controls (6.95,-1.4) and (3.31,-0.3) .. (0,0) .. controls (3.31,0.3) and (6.95,1.4) .. (10.93,3.29)   ;
\draw  [dash pattern={on 0.84pt off 2.51pt}]  (98,83) -- (98,25) ;
\draw [shift={(98,23)}, rotate = 90] [color={rgb, 255:red, 0; green, 0; blue, 0 }  ][line width=0.75]    (10.93,-3.29) .. controls (6.95,-1.4) and (3.31,-0.3) .. (0,0) .. controls (3.31,0.3) and (6.95,1.4) .. (10.93,3.29)   ;
\draw  [dash pattern={on 0.84pt off 2.51pt}]  (187.02,233) -- (187.02,291) ;
\draw [shift={(187.02,293)}, rotate = 270] [color={rgb, 255:red, 0; green, 0; blue, 0 }  ][line width=0.75]    (10.93,-3.29) .. controls (6.95,-1.4) and (3.31,-0.3) .. (0,0) .. controls (3.31,0.3) and (6.95,1.4) .. (10.93,3.29)   ;
\draw  [dash pattern={on 0.84pt off 2.51pt}]  (216.04,233) -- (216.04,291) ;
\draw [shift={(216.04,293)}, rotate = 270] [color={rgb, 255:red, 0; green, 0; blue, 0 }  ][line width=0.75]    (10.93,-3.29) .. controls (6.95,-1.4) and (3.31,-0.3) .. (0,0) .. controls (3.31,0.3) and (6.95,1.4) .. (10.93,3.29)   ;
\draw  [dash pattern={on 0.84pt off 2.51pt}]  (214,233) -- (272,233) ;
\draw [shift={(274,233)}, rotate = 180] [color={rgb, 255:red, 0; green, 0; blue, 0 }  ][line width=0.75]    (10.93,-3.29) .. controls (6.95,-1.4) and (3.31,-0.3) .. (0,0) .. controls (3.31,0.3) and (6.95,1.4) .. (10.93,3.29)   ;
\draw  [dash pattern={on 0.84pt off 2.51pt}]  (214,203) -- (272,203) ;
\draw [shift={(274,203)}, rotate = 180] [color={rgb, 255:red, 0; green, 0; blue, 0 }  ][line width=0.75]    (10.93,-3.29) .. controls (6.95,-1.4) and (3.31,-0.3) .. (0,0) .. controls (3.31,0.3) and (6.95,1.4) .. (10.93,3.29)   ;
\draw  [dash pattern={on 0.84pt off 2.51pt}]  (68,113) -- (10,113) ;
\draw [shift={(8,113)}, rotate = 360] [color={rgb, 255:red, 0; green, 0; blue, 0 }  ][line width=0.75]    (10.93,-3.29) .. controls (6.95,-1.4) and (3.31,-0.3) .. (0,0) .. controls (3.31,0.3) and (6.95,1.4) .. (10.93,3.29)   ;
\draw  [dash pattern={on 0.84pt off 2.51pt}]  (68,83) -- (10,83) ;
\draw [shift={(8,83)}, rotate = 360] [color={rgb, 255:red, 0; green, 0; blue, 0 }  ][line width=0.75]    (10.93,-3.29) .. controls (6.95,-1.4) and (3.31,-0.3) .. (0,0) .. controls (3.31,0.3) and (6.95,1.4) .. (10.93,3.29)   ;
\draw  [dash pattern={on 0.84pt off 2.51pt}]  (98,83) -- (217.02,83) ;
\draw  [dash pattern={on 0.84pt off 2.51pt}]  (217.02,203) -- (217.02,83) ;
\draw  [dash pattern={on 0.84pt off 2.51pt}]  (68,233) -- (68,123) ;
\draw  [dash pattern={on 0.84pt off 2.51pt}]  (188,233) -- (68,233) ;
\draw  [draw opacity=0][fill={rgb, 255:red, 0; green, 0; blue, 0 }  ,fill opacity=1 ] (64.49,83) .. controls (64.49,81.48) and (65.72,80.24) .. (67.24,80.24) .. controls (68.77,80.24) and (70,81.48) .. (70,83) .. controls (70,84.52) and (68.77,85.76) .. (67.24,85.76) .. controls (65.72,85.76) and (64.49,84.52) .. (64.49,83) -- cycle ;
\draw  [draw opacity=0][fill={rgb, 255:red, 0; green, 0; blue, 0 }  ,fill opacity=1 ] (215.49,233) .. controls (215.49,231.48) and (216.72,230.24) .. (218.24,230.24) .. controls (219.77,230.24) and (221,231.48) .. (221,233) .. controls (221,234.52) and (219.77,235.76) .. (218.24,235.76) .. controls (216.72,235.76) and (215.49,234.52) .. (215.49,233) -- cycle ;
\draw    (68,53) -- (98,53) ;
\draw    (68,83) -- (68,123) ;
\draw    (98,83) .. controls (99,106.63) and (93,126.63) .. (68,133) ;
\draw    (187.02,263) -- (217.02,263) ;
\draw    (217.02,203) -- (217.02,233) ;
\draw    (218,183) .. controls (190,187.63) and (190,203.63) .. (188,233) ;
\draw  [dash pattern={on 0.84pt off 2.51pt}]  (118,83) -- (118,233) ;
\draw  [dash pattern={on 0.84pt off 2.51pt}]  (168,83) -- (168,233) ;
\draw  [dash pattern={on 0.84pt off 2.51pt}]  (68,133) -- (218,133) ;
\draw  [dash pattern={on 0.84pt off 2.51pt}]  (68,183) -- (218,183) ;
\draw  [draw opacity=0][fill={rgb, 255:red, 155; green, 155; blue, 155 }  ,fill opacity=0.5 ] (188.02,213) -- (197.02,193) -- (208,184) -- (217.02,183) -- (216.04,263) -- (187.02,263) -- cycle ;
\draw  [draw opacity=0][fill={rgb, 255:red, 207; green, 132; blue, 221 }  ,fill opacity=0.44 ][line width=0.75]  (336,86) -- (485.02,86) -- (485.02,236) -- (336,236) -- cycle ;
\draw  [draw opacity=0][fill={rgb, 255:red, 207; green, 132; blue, 221 }  ,fill opacity=0.44 ][line width=0.75]  (306,86) -- (336,86) -- (336,116) -- (306,116) -- cycle ;
\draw  [draw opacity=0][fill={rgb, 255:red, 208; green, 2; blue, 27 }  ,fill opacity=0.41 ][line width=0.75]  (456.22,236) -- (486.24,236) -- (486.24,296) -- (456.22,296) -- cycle ;
\draw  [draw opacity=0][fill={rgb, 255:red, 207; green, 132; blue, 221 }  ,fill opacity=0.44 ][line width=0.75]  (485.02,206) -- (515.02,206) -- (515.02,236) -- (485.02,236) -- cycle ;
\draw  [dash pattern={on 0.84pt off 2.51pt}]  (335.24,86) -- (335.24,28) ;
\draw [shift={(335.24,26)}, rotate = 90] [color={rgb, 255:red, 0; green, 0; blue, 0 }  ][line width=0.75]    (10.93,-3.29) .. controls (6.95,-1.4) and (3.31,-0.3) .. (0,0) .. controls (3.31,0.3) and (6.95,1.4) .. (10.93,3.29)   ;
\draw  [dash pattern={on 0.84pt off 2.51pt}]  (366,86) -- (366,28) ;
\draw [shift={(366,26)}, rotate = 90] [color={rgb, 255:red, 0; green, 0; blue, 0 }  ][line width=0.75]    (10.93,-3.29) .. controls (6.95,-1.4) and (3.31,-0.3) .. (0,0) .. controls (3.31,0.3) and (6.95,1.4) .. (10.93,3.29)   ;
\draw  [dash pattern={on 0.84pt off 2.51pt}]  (456.24,236) -- (456.24,294) ;
\draw [shift={(456.24,296)}, rotate = 270] [color={rgb, 255:red, 0; green, 0; blue, 0 }  ][line width=0.75]    (10.93,-3.29) .. controls (6.95,-1.4) and (3.31,-0.3) .. (0,0) .. controls (3.31,0.3) and (6.95,1.4) .. (10.93,3.29)   ;
\draw  [dash pattern={on 0.84pt off 2.51pt}]  (486.24,236) -- (486.24,294) ;
\draw [shift={(486.24,296)}, rotate = 270] [color={rgb, 255:red, 0; green, 0; blue, 0 }  ][line width=0.75]    (10.93,-3.29) .. controls (6.95,-1.4) and (3.31,-0.3) .. (0,0) .. controls (3.31,0.3) and (6.95,1.4) .. (10.93,3.29)   ;
\draw  [dash pattern={on 0.84pt off 2.51pt}]  (482,236) -- (540,236) ;
\draw [shift={(542,236)}, rotate = 180] [color={rgb, 255:red, 0; green, 0; blue, 0 }  ][line width=0.75]    (10.93,-3.29) .. controls (6.95,-1.4) and (3.31,-0.3) .. (0,0) .. controls (3.31,0.3) and (6.95,1.4) .. (10.93,3.29)   ;
\draw  [dash pattern={on 0.84pt off 2.51pt}]  (485.02,206) -- (543.02,206) ;
\draw [shift={(545.02,206)}, rotate = 180] [color={rgb, 255:red, 0; green, 0; blue, 0 }  ][line width=0.75]    (10.93,-3.29) .. controls (6.95,-1.4) and (3.31,-0.3) .. (0,0) .. controls (3.31,0.3) and (6.95,1.4) .. (10.93,3.29)   ;
\draw  [dash pattern={on 0.84pt off 2.51pt}]  (336,116) -- (278,116) ;
\draw [shift={(276,116)}, rotate = 360] [color={rgb, 255:red, 0; green, 0; blue, 0 }  ][line width=0.75]    (10.93,-3.29) .. controls (6.95,-1.4) and (3.31,-0.3) .. (0,0) .. controls (3.31,0.3) and (6.95,1.4) .. (10.93,3.29)   ;
\draw  [dash pattern={on 0.84pt off 2.51pt}]  (336,86) -- (278,86) ;
\draw [shift={(276,86)}, rotate = 360] [color={rgb, 255:red, 0; green, 0; blue, 0 }  ][line width=0.75]    (10.93,-3.29) .. controls (6.95,-1.4) and (3.31,-0.3) .. (0,0) .. controls (3.31,0.3) and (6.95,1.4) .. (10.93,3.29)   ;
\draw  [dash pattern={on 0.84pt off 2.51pt}]  (366,86) -- (485.02,86) ;
\draw  [dash pattern={on 0.84pt off 2.51pt}]  (485.02,206) -- (485.02,86) ;
\draw  [dash pattern={on 0.84pt off 2.51pt}]  (336,236) -- (336,116) ;
\draw  [dash pattern={on 0.84pt off 2.51pt}]  (456,236) -- (336,236) ;
\draw  [draw opacity=0][fill={rgb, 255:red, 0; green, 0; blue, 0 }  ,fill opacity=1 ] (332.49,86) .. controls (332.49,84.48) and (333.72,83.24) .. (335.24,83.24) .. controls (336.77,83.24) and (338,84.48) .. (338,86) .. controls (338,87.52) and (336.77,88.76) .. (335.24,88.76) .. controls (333.72,88.76) and (332.49,87.52) .. (332.49,86) -- cycle ;
\draw  [draw opacity=0][fill={rgb, 255:red, 0; green, 0; blue, 0 }  ,fill opacity=1 ] (483.49,236) .. controls (483.49,234.48) and (484.72,233.24) .. (486.24,233.24) .. controls (487.77,233.24) and (489,234.48) .. (489,236) .. controls (489,237.52) and (487.77,238.76) .. (486.24,238.76) .. controls (484.72,238.76) and (483.49,237.52) .. (483.49,236) -- cycle ;
\draw    (335.24,86) -- (365.24,86) ;
\draw    (306,86) -- (306,116) ;
\draw    (386,86) .. controls (383,112.63) and (366,115.63) .. (336,116) ;
\draw    (456.24,236) -- (486.24,236) ;
\draw    (515.02,206) -- (515.02,236) ;
\draw    (485.02,206) .. controls (460,206.63) and (443,212.63) .. (436,236) ;
\draw  [dash pattern={on 0.84pt off 2.51pt}]  (386,86) -- (386,236) ;
\draw  [dash pattern={on 0.84pt off 2.51pt}]  (436,86) -- (436,236) ;
\draw  [dash pattern={on 0.84pt off 2.51pt}]  (336,136) -- (486,136) ;
\draw  [dash pattern={on 0.84pt off 2.51pt}]  (336,186) -- (486,186) ;
\draw  [draw opacity=0][fill={rgb, 255:red, 155; green, 155; blue, 155 }  ,fill opacity=0.5 ] (306,86) -- (386,86) -- (380,104) -- (362,114) -- (336,116) -- (306,116) -- cycle ;
\draw  [draw opacity=0][fill={rgb, 255:red, 155; green, 155; blue, 155 }  ,fill opacity=0.5 ] (445,218) -- (462,209) -- (485.02,206) -- (515.02,206) -- (515.02,236) -- (436,236) -- cycle ;
\draw    (104,75) .. controls (143.6,45.3) and (269.45,41.7) .. (310.77,65.28) ;
\draw [shift={(312,66)}, rotate = 211.35] [color={rgb, 255:red, 0; green, 0; blue, 0 }  ][line width=0.75]    (10.93,-3.29) .. controls (6.95,-1.4) and (3.31,-0.3) .. (0,0) .. controls (3.31,0.3) and (6.95,1.4) .. (10.93,3.29)   ;
\draw    (224,255) .. controls (275.48,269.49) and (394.59,282.37) .. (440.63,247.08) ;
\draw [shift={(442,246)}, rotate = 140.85] [color={rgb, 255:red, 0; green, 0; blue, 0 }  ][line width=0.75]    (10.93,-3.29) .. controls (6.95,-1.4) and (3.31,-0.3) .. (0,0) .. controls (3.31,0.3) and (6.95,1.4) .. (10.93,3.29)   ;
\draw  [draw opacity=0][fill={rgb, 255:red, 189; green, 16; blue, 224 }  ,fill opacity=0.48 ] (68,133) -- (88,123) -- (98,103) -- (98,83) -- (118,83) -- (118,233) -- (68,233) -- cycle ;
\draw  [draw opacity=0][fill={rgb, 255:red, 189; green, 16; blue, 224 }  ,fill opacity=0.48 ] (336,116) -- (362,114) -- (380,104) -- (386,86) -- (485.02,86) -- (486,136) -- (336,136) -- cycle ;
\draw  [draw opacity=0][fill={rgb, 255:red, 207; green, 132; blue, 221 }  ,fill opacity=0.44 ][line width=0.75]  (118,83) -- (168,83) -- (168,233) -- (118,233) -- cycle ;
\draw  [draw opacity=0][fill={rgb, 255:red, 207; green, 132; blue, 221 }  ,fill opacity=0.44 ][line width=0.75]  (336,136) -- (486,136) -- (486,186) -- (336,186) -- cycle ;
\draw  [draw opacity=0][fill={rgb, 255:red, 207; green, 132; blue, 221 }  ,fill opacity=0.44 ][line width=0.75]  (68,53) -- (98,53) -- (98,83) -- (68,83) -- cycle ;
\draw  [draw opacity=0][fill={rgb, 255:red, 155; green, 155; blue, 155 }  ,fill opacity=0.5 ] (68,53) -- (98,53) -- (98,83) -- (98,103) -- (88,123) -- (68,133) -- cycle ;
\draw  [draw opacity=0][fill={rgb, 255:red, 184; green, 233; blue, 134 }  ,fill opacity=0.42 ][line width=0.75]  (335.24,26) -- (365.24,26) -- (365.24,86) -- (335.24,86) -- cycle ;
\draw  [draw opacity=0][fill={rgb, 255:red, 74; green, 144; blue, 226 }  ,fill opacity=0.55 ][line width=0.75]  (276,86) -- (306,86) -- (306,116) -- (276,116) -- cycle ;
\draw  [draw opacity=0][fill={rgb, 255:red, 220; green, 149; blue, 91 }  ,fill opacity=0.49 ][line width=0.75]  (515.02,206) -- (545.02,206) -- (545.02,236) -- (515.02,236) -- cycle ;
\draw  [draw opacity=0][fill={rgb, 255:red, 208; green, 2; blue, 27 }  ,fill opacity=0.41 ][line width=0.75]  (187.02,263) -- (216.04,263) -- (216.04,293) -- (187.02,293) -- cycle ;

\draw (65,7.4) node [anchor=north west][inner sep=0.75pt]  [font=\small]  {$E_{T}( 0,1)$};
\draw (8,92.4) node [anchor=north west][inner sep=0.75pt]  [font=\small]  {$E_{L}( 0,1)$};
\draw (225,211.4) node [anchor=north west][inner sep=0.75pt]  [font=\small]  {$E_{R}( 1,0)$};
\draw (171,295.4) node [anchor=north west][inner sep=0.75pt]  [font=\small]  {$E_{B}( 1,0)$};
\draw (43,69.4) node [anchor=north west][inner sep=0.75pt]  [font=\tiny]  {$( 0,1)$};
\draw (223,239.4) node [anchor=north west][inner sep=0.75pt]  [font=\tiny]  {$( 1,0)$};
\draw (129,145.4) node [anchor=north west][inner sep=0.75pt]    {$\dotsc $};
\draw (146.48,172.05) node [anchor=north west][inner sep=0.75pt]  [rotate=-269.83]  {$\dotsc $};
\draw (333,10.4) node [anchor=north west][inner sep=0.75pt]  [font=\small]  {$E_{T}( 0,1)$};
\draw (250,94.4) node [anchor=north west][inner sep=0.75pt]  [font=\small]  {$E_{L}( 0,1)$};
\draw (524,212.4) node [anchor=north west][inner sep=0.75pt]  [font=\small]  {$E_{R}( 1,0)$};
\draw (439,298.4) node [anchor=north west][inner sep=0.75pt]  [font=\small]  {$E_{B}( 1,0)$};
\draw (311,72.4) node [anchor=north west][inner sep=0.75pt]  [font=\tiny]  {$( 0,1)$};
\draw (491,242.4) node [anchor=north west][inner sep=0.75pt]  [font=\tiny]  {$( 1,0)$};
\draw (397,148.4) node [anchor=north west][inner sep=0.75pt]    {$\dotsc $};
\draw (418.4,174.99) node [anchor=north west][inner sep=0.75pt]  [rotate=-269.83]  {$\dotsc $};
\draw (185,15.4) node [anchor=north west][inner sep=0.75pt]    {$h_{( 0,1)}$};
\draw (303,286.4) node [anchor=north west][inner sep=0.75pt]    {$h_{( 1,0)}$};
\draw (72,86.4) node [anchor=north west][inner sep=0.75pt]    {$P_{1}$};
\draw (340,89.4) node [anchor=north west][inner sep=0.75pt]    {$P_{1} '$};
\draw (195,207.4) node [anchor=north west][inner sep=0.75pt]    {$P_{2}$};
\draw (464,212.4) node [anchor=north west][inner sep=0.75pt]    {$P_{2} '$};

\end{tikzpicture}
\end{center}
\caption{The homeomorphism $f_1:V_1 \rightarrow H_1$ with colors indicating the image of each region. }
    \label{fig:integer extended piece map}
\end{figure}

Define a homeomorphism $f_1:V_1 \rightarrow H_1$ by
$$f_1(x,y) = \begin{cases}
    f_0(x,y) & \ \text{if } (x,y) \in V_0 - (P_1 \cup P_2) \\
    h_{(0,1)}(x,y) & \ \text{if } (x,y) \in P_1 \\
    h_{(1,0)}(x,y) & \ \text{if } (x,y) \in P_2 \\
    (x-1,y) & \ \text{if } (x,y) \in E_L(0,1) \\
    (x,y-1) & \ \text{if } (x,y) \in E_T(0,1)  - P_1\\
    (x+1,y) & \ \text{if } (x,y) \in E_R(1,0) \\
    (x, y+1) & \ \text{if } (x,y) \in E_B(1,0) -P_2
\end{cases}$$

\noindent where $h_{(0,1)}:P_1 \rightarrow P_1' \text{ and }h_{(1,0)}:P_2 \rightarrow P_2'$ are homeomorphisms called the \textit{switch maps} chosen so that $f_1$ is a well-defined homeomorphism. Note that $f_1$ acts exactly as we had hoped: it restricts to $f_0$ on the interior of the square minus the switch regions, acts by a translation in the infinite strips minus the switch regions, and transitions between $f_0$ and a translation in the switch regions. 

Let $L_1$ denote the union of the left edges of the vertical strips in $V_1$. By left edge of $(V_{1,1}^{(1)})'$, we mean \[(\{0\} \times [0, 1-\frac{1}{d^2}]) \cup (-\infty,0] \times \{1-\frac{1}{d^2}\}).\] Similarly, define $R_1$ to be the union of the right edges of the vertical strips, define $T_1$ to be the union of the top edges of the horizontal strips, and define $B_1$ to be the union of the bottom edges of the horizontal strips. Next we define four edge maps, coming from the different ways of extending $f_1$ and $f_1^{-1}$ to the boundaries of the partitions.

Given a point $x \in L_1$, there is a sequence $x_i \in V_1$ converging to $x$ ``from the right", i.e. from the interior of the vertical strip whose left edge $x$ lives in. In the case that $x$ is in the bottom boundary component of $E_L(0,1)$, that is $x \in ((-\infty,0] \times \{1-\frac{1}{d^2}\})$, then by ``from the right" we mean any sequence in $V_1$ converging to $x$. Define the \emph{left edge map of $f_1$} as \[f_{1,L}:L_1\rightarrow L_1\] \[f_{1,L}(x) := \lim_{i \rightarrow \infty} f_0(x_i),\] where the limit is taken ``from the right" as described above. (See \Cref{from the right}.) 

Given a point $x \in R_1$, there is a sequence $x_i \in V_1$ converging to $x$ ``from the left", i.e. from the interior of the vertical strip whose right edge $x$ lives in. In the case that $x$ is in the top boundary component of $E_R(1, 0)$ then by ``from the left" we mean any sequence in $V_1$ converging to $x$. Define the \emph{right edge map of $f_1$} as
\[f_{1,R}:R_1\rightarrow R_1\] \[f_{1,R}(x) := \lim_{i \rightarrow \infty} f_0(x_i),\] where the limit is taken ``from the left" as described above.

Similarly, define the top edge map of $f_1^{-1}$, $$f_{1,T}^{-1}:T_1 \rightarrow T_1$$ and the bottom edge map of $f_1^{-1}$, $$f_{1,B}^{-1}:B_1 \rightarrow B_1$$ using the analogous notions of convergence ``from below" and ``from above", respectively. 


Now, we build a 2-complex 
$$\Sigma_2 := \Sigma_1 /\sim,$$
where $\sim$ is an equivalence relation on $\Sigma_1$ generated by $z \sim w$ if
\begin{enumerate}
    \item[(i)] there exists $x \in L_1 \cap R_1$ and $N \in \mathbb{Z}_{\geq 1}$ such that $f_{1,L}^N(x) =z$ and $f_{1,R}^N(x) = w$, or
    \item[(ii)]there exists $x \in T_1 \cap B_1$ and $N \in \mathbb{Z}_{\geq 1}$ such that $f_{1,T}^{-N}(x) =z$ and $f_{1,B}^{-N}(x) = w$.
\end{enumerate}



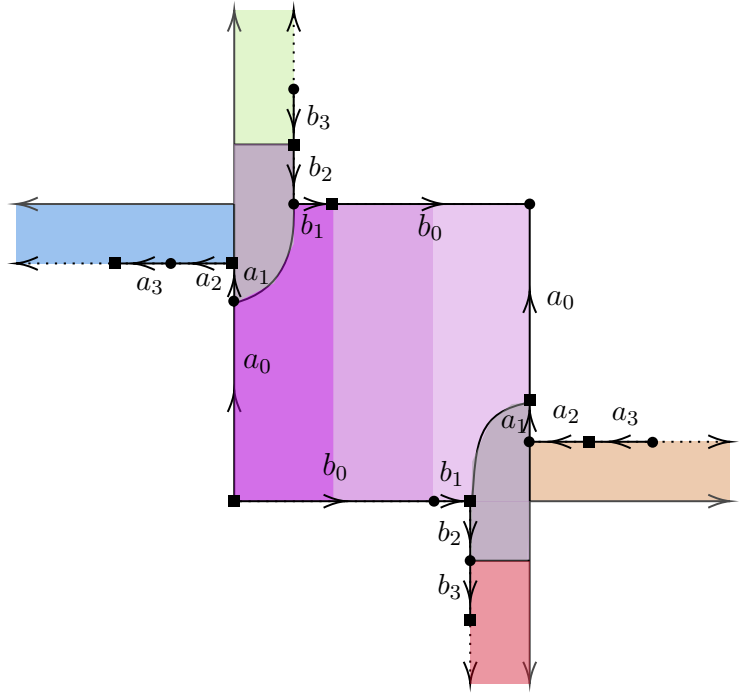
\begin{figure}[htbp]
\begin{center}

\tikzset{every picture/.style={line width=0.75pt}} 

\begin{tikzpicture}[x=0.75pt,y=0.75pt,yscale=-1,xscale=1]

\draw  [draw opacity=0][fill={rgb, 255:red, 207; green, 132; blue, 221 }  ,fill opacity=0.44 ][line width=0.75]  (130,108) -- (279.02,108) -- (279.02,258) -- (130,258) -- cycle ;
\draw  [draw opacity=0][fill={rgb, 255:red, 184; green, 233; blue, 134 }  ,fill opacity=0.42 ][line width=0.75]  (130,10) -- (160,10) -- (160,78) -- (130,78) -- cycle ;
\draw  [draw opacity=0][fill={rgb, 255:red, 74; green, 144; blue, 226 }  ,fill opacity=0.55 ][line width=0.75]  (20,108) -- (130,108) -- (130,138) -- (20,138) -- cycle ;
\draw  [draw opacity=0][fill={rgb, 255:red, 207; green, 132; blue, 221 }  ,fill opacity=0.44 ][line width=0.75]  (250,258) -- (279.02,258) -- (279.02,288) -- (250,288) -- cycle ;
\draw  [draw opacity=0][fill={rgb, 255:red, 220; green, 149; blue, 91 }  ,fill opacity=0.49 ][line width=0.75]  (279.02,228) -- (380,228) -- (380,258) -- (279.02,258) -- cycle ;
\draw [color={rgb, 255:red, 74; green, 74; blue, 74 }  ,draw opacity=1 ]   (130,108) -- (130,12) ;
\draw [shift={(130,10)}, rotate = 90] [color={rgb, 255:red, 74; green, 74; blue, 74 }  ,draw opacity=1 ][line width=0.75]    (10.93,-3.29) .. controls (6.95,-1.4) and (3.31,-0.3) .. (0,0) .. controls (3.31,0.3) and (6.95,1.4) .. (10.93,3.29)   ;
\draw  [dash pattern={on 0.84pt off 2.51pt}]  (160,108) -- (160,12) ;
\draw [shift={(160,10)}, rotate = 90] [color={rgb, 255:red, 0; green, 0; blue, 0 }  ][line width=0.75]    (10.93,-3.29) .. controls (6.95,-1.4) and (3.31,-0.3) .. (0,0) .. controls (3.31,0.3) and (6.95,1.4) .. (10.93,3.29)   ;
\draw  [dash pattern={on 0.84pt off 2.51pt}]  (249.02,258) -- (249.02,348) ;
\draw [shift={(249.02,350)}, rotate = 270] [color={rgb, 255:red, 0; green, 0; blue, 0 }  ][line width=0.75]    (10.93,-3.29) .. controls (6.95,-1.4) and (3.31,-0.3) .. (0,0) .. controls (3.31,0.3) and (6.95,1.4) .. (10.93,3.29)   ;
\draw [color={rgb, 255:red, 74; green, 74; blue, 74 }  ,draw opacity=1 ]   (279.02,258) -- (279.02,348) ;
\draw [shift={(279.02,350)}, rotate = 270] [color={rgb, 255:red, 74; green, 74; blue, 74 }  ,draw opacity=1 ][line width=0.75]    (10.93,-3.29) .. controls (6.95,-1.4) and (3.31,-0.3) .. (0,0) .. controls (3.31,0.3) and (6.95,1.4) .. (10.93,3.29)   ;
\draw [color={rgb, 255:red, 74; green, 74; blue, 74 }  ,draw opacity=1 ]   (279.02,258) -- (378,258) ;
\draw [shift={(380,258)}, rotate = 180] [color={rgb, 255:red, 74; green, 74; blue, 74 }  ,draw opacity=1 ][line width=0.75]    (10.93,-3.29) .. controls (6.95,-1.4) and (3.31,-0.3) .. (0,0) .. controls (3.31,0.3) and (6.95,1.4) .. (10.93,3.29)   ;
\draw  [dash pattern={on 0.84pt off 2.51pt}]  (279,228) -- (378,228) ;
\draw [shift={(380,228)}, rotate = 180] [color={rgb, 255:red, 0; green, 0; blue, 0 }  ][line width=0.75]    (10.93,-3.29) .. controls (6.95,-1.4) and (3.31,-0.3) .. (0,0) .. controls (3.31,0.3) and (6.95,1.4) .. (10.93,3.29)   ;
\draw  [dash pattern={on 0.84pt off 2.51pt}]  (130,138) -- (22,138) ;
\draw [shift={(20,138)}, rotate = 360] [color={rgb, 255:red, 0; green, 0; blue, 0 }  ][line width=0.75]    (10.93,-3.29) .. controls (6.95,-1.4) and (3.31,-0.3) .. (0,0) .. controls (3.31,0.3) and (6.95,1.4) .. (10.93,3.29)   ;
\draw [color={rgb, 255:red, 74; green, 74; blue, 74 }  ,draw opacity=1 ]   (130,108) -- (22,108) ;
\draw [shift={(20,108)}, rotate = 360] [color={rgb, 255:red, 74; green, 74; blue, 74 }  ,draw opacity=1 ][line width=0.75]    (10.93,-3.29) .. controls (6.95,-1.4) and (3.31,-0.3) .. (0,0) .. controls (3.31,0.3) and (6.95,1.4) .. (10.93,3.29)   ;
\draw  [dash pattern={on 0.84pt off 2.51pt}]  (160,108) -- (279.02,108) ;
\draw    (130,258) -- (130,158) ;
\draw [shift={(130,202)}, rotate = 90] [color={rgb, 255:red, 0; green, 0; blue, 0 }  ][line width=0.75]    (10.93,-3.29) .. controls (6.95,-1.4) and (3.31,-0.3) .. (0,0) .. controls (3.31,0.3) and (6.95,1.4) .. (10.93,3.29)   ;
\draw  [dash pattern={on 0.84pt off 2.51pt}]  (250,258) -- (130,258) ;
\draw    (130,78) -- (160,78) ;
\draw    (130,108) -- (130,140) ;
\draw    (160,108) .. controls (161,131.63) and (155,151.63) .. (130,158) ;
\draw    (249.02,288) -- (279.02,288) ;
\draw    (279.02,228) -- (279.02,258) ;
\draw    (280,208) .. controls (252,212.63) and (252,228.63) .. (250,258) ;
\draw  [draw opacity=0][fill={rgb, 255:red, 155; green, 155; blue, 155 }  ,fill opacity=0.5 ] (250,238) -- (259,218) -- (269.98,209) -- (279,208) -- (278.02,288) -- (249,288) -- cycle ;
\draw  [draw opacity=0][fill={rgb, 255:red, 189; green, 16; blue, 224 }  ,fill opacity=0.48 ] (130,158) -- (150,148) -- (160,128) -- (160,108) -- (180,108) -- (180,258) -- (130,258) -- cycle ;
\draw  [draw opacity=0][fill={rgb, 255:red, 207; green, 132; blue, 221 }  ,fill opacity=0.44 ][line width=0.75]  (180,108) -- (230,108) -- (230,258) -- (180,258) -- cycle ;
\draw  [draw opacity=0][fill={rgb, 255:red, 207; green, 132; blue, 221 }  ,fill opacity=0.44 ][line width=0.75]  (130,78) -- (160,78) -- (160,108) -- (130,108) -- cycle ;
\draw  [draw opacity=0][fill={rgb, 255:red, 155; green, 155; blue, 155 }  ,fill opacity=0.5 ] (130,78) -- (160,78) -- (160,108) -- (160,128) -- (150,148) -- (130,158) -- cycle ;
\draw  [draw opacity=0][fill={rgb, 255:red, 208; green, 2; blue, 27 }  ,fill opacity=0.41 ][line width=0.75]  (249.02,288) -- (279.02,288) -- (279.02,350) -- (249.02,350) -- cycle ;
\draw    (279.02,208) -- (279.02,108) ;
\draw [shift={(279.02,152)}, rotate = 90] [color={rgb, 255:red, 0; green, 0; blue, 0 }  ][line width=0.75]    (10.93,-3.29) .. controls (6.95,-1.4) and (3.31,-0.3) .. (0,0) .. controls (3.31,0.3) and (6.95,1.4) .. (10.93,3.29)   ;
\draw    (130,258) -- (230,258) ;
\draw [shift={(186,258)}, rotate = 180] [color={rgb, 255:red, 0; green, 0; blue, 0 }  ][line width=0.75]    (10.93,-3.29) .. controls (6.95,-1.4) and (3.31,-0.3) .. (0,0) .. controls (3.31,0.3) and (6.95,1.4) .. (10.93,3.29)   ;
\draw    (180,108) -- (279.02,108) ;
\draw [shift={(235.51,108)}, rotate = 180] [color={rgb, 255:red, 0; green, 0; blue, 0 }  ][line width=0.75]    (10.93,-3.29) .. controls (6.95,-1.4) and (3.31,-0.3) .. (0,0) .. controls (3.31,0.3) and (6.95,1.4) .. (10.93,3.29)   ;
\draw    (130,160) -- (130,140) ;
\draw [shift={(130,144)}, rotate = 90] [color={rgb, 255:red, 0; green, 0; blue, 0 }  ][line width=0.75]    (10.93,-3.29) .. controls (6.95,-1.4) and (3.31,-0.3) .. (0,0) .. controls (3.31,0.3) and (6.95,1.4) .. (10.93,3.29)   ;
\draw    (279,228) -- (279,208) ;
\draw [shift={(279,212)}, rotate = 90] [color={rgb, 255:red, 0; green, 0; blue, 0 }  ][line width=0.75]    (10.93,-3.29) .. controls (6.95,-1.4) and (3.31,-0.3) .. (0,0) .. controls (3.31,0.3) and (6.95,1.4) .. (10.93,3.29)   ;
\draw    (160,108) -- (180,108) ;
\draw [shift={(176,108)}, rotate = 180] [color={rgb, 255:red, 0; green, 0; blue, 0 }  ][line width=0.75]    (10.93,-3.29) .. controls (6.95,-1.4) and (3.31,-0.3) .. (0,0) .. controls (3.31,0.3) and (6.95,1.4) .. (10.93,3.29)   ;
\draw    (229.02,258) -- (249.02,258) ;
\draw [shift={(245.02,258)}, rotate = 180] [color={rgb, 255:red, 0; green, 0; blue, 0 }  ][line width=0.75]    (10.93,-3.29) .. controls (6.95,-1.4) and (3.31,-0.3) .. (0,0) .. controls (3.31,0.3) and (6.95,1.4) .. (10.93,3.29)   ;
\draw    (160,78) -- (160,108) ;
\draw [shift={(160,99)}, rotate = 270] [color={rgb, 255:red, 0; green, 0; blue, 0 }  ][line width=0.75]    (10.93,-3.29) .. controls (6.95,-1.4) and (3.31,-0.3) .. (0,0) .. controls (3.31,0.3) and (6.95,1.4) .. (10.93,3.29)   ;
\draw    (310,228) -- (280,228) ;
\draw [shift={(289,228)}, rotate = 360] [color={rgb, 255:red, 0; green, 0; blue, 0 }  ][line width=0.75]    (10.93,-3.29) .. controls (6.95,-1.4) and (3.31,-0.3) .. (0,0) .. controls (3.31,0.3) and (6.95,1.4) .. (10.93,3.29)   ;
\draw    (340,228) -- (310,228) ;
\draw [shift={(319,228)}, rotate = 360] [color={rgb, 255:red, 0; green, 0; blue, 0 }  ][line width=0.75]    (10.93,-3.29) .. controls (6.95,-1.4) and (3.31,-0.3) .. (0,0) .. controls (3.31,0.3) and (6.95,1.4) .. (10.93,3.29)   ;
\draw    (100,138) -- (70,138) ;
\draw [shift={(79,138)}, rotate = 360] [color={rgb, 255:red, 0; green, 0; blue, 0 }  ][line width=0.75]    (10.93,-3.29) .. controls (6.95,-1.4) and (3.31,-0.3) .. (0,0) .. controls (3.31,0.3) and (6.95,1.4) .. (10.93,3.29)   ;
\draw    (130,138) -- (100,138) ;
\draw [shift={(109,138)}, rotate = 360] [color={rgb, 255:red, 0; green, 0; blue, 0 }  ][line width=0.75]    (10.93,-3.29) .. controls (6.95,-1.4) and (3.31,-0.3) .. (0,0) .. controls (3.31,0.3) and (6.95,1.4) .. (10.93,3.29)   ;
\draw    (160,50) -- (160,80) ;
\draw [shift={(160,71)}, rotate = 270] [color={rgb, 255:red, 0; green, 0; blue, 0 }  ][line width=0.75]    (10.93,-3.29) .. controls (6.95,-1.4) and (3.31,-0.3) .. (0,0) .. controls (3.31,0.3) and (6.95,1.4) .. (10.93,3.29)   ;
\draw    (249.02,258) -- (249.02,288) ;
\draw [shift={(249.02,279)}, rotate = 270] [color={rgb, 255:red, 0; green, 0; blue, 0 }  ][line width=0.75]    (10.93,-3.29) .. controls (6.95,-1.4) and (3.31,-0.3) .. (0,0) .. controls (3.31,0.3) and (6.95,1.4) .. (10.93,3.29)   ;
\draw    (249,288) -- (249,318) ;
\draw [shift={(249,309)}, rotate = 270] [color={rgb, 255:red, 0; green, 0; blue, 0 }  ][line width=0.75]    (10.93,-3.29) .. controls (6.95,-1.4) and (3.31,-0.3) .. (0,0) .. controls (3.31,0.3) and (6.95,1.4) .. (10.93,3.29)   ;
\draw  [draw opacity=0][fill={rgb, 255:red, 0; green, 0; blue, 0 }  ,fill opacity=1 ] (338.24,228.24) .. controls (338.24,226.72) and (339.48,225.49) .. (341,225.49) .. controls (342.52,225.49) and (343.76,226.72) .. (343.76,228.24) .. controls (343.76,229.77) and (342.52,231) .. (341,231) .. controls (339.48,231) and (338.24,229.77) .. (338.24,228.24) -- cycle ;
\draw  [draw opacity=0][fill={rgb, 255:red, 0; green, 0; blue, 0 }  ,fill opacity=1 ] (276,228) .. controls (276,226.48) and (277.23,225.24) .. (278.76,225.24) .. controls (280.28,225.24) and (281.51,226.48) .. (281.51,228) .. controls (281.51,229.52) and (280.28,230.76) .. (278.76,230.76) .. controls (277.23,230.76) and (276,229.52) .. (276,228) -- cycle ;
\draw  [draw opacity=0][fill={rgb, 255:red, 0; green, 0; blue, 0 }  ,fill opacity=1 ] (276.26,108) .. controls (276.26,106.48) and (277.5,105.24) .. (279.02,105.24) .. controls (280.54,105.24) and (281.78,106.48) .. (281.78,108) .. controls (281.78,109.52) and (280.54,110.76) .. (279.02,110.76) .. controls (277.5,110.76) and (276.26,109.52) .. (276.26,108) -- cycle ;
\draw  [draw opacity=0][fill={rgb, 255:red, 0; green, 0; blue, 0 }  ,fill opacity=1 ] (157.24,108) .. controls (157.24,106.48) and (158.48,105.24) .. (160,105.24) .. controls (161.52,105.24) and (162.76,106.48) .. (162.76,108) .. controls (162.76,109.52) and (161.52,110.76) .. (160,110.76) .. controls (158.48,110.76) and (157.24,109.52) .. (157.24,108) -- cycle ;
\draw  [draw opacity=0][fill={rgb, 255:red, 0; green, 0; blue, 0 }  ,fill opacity=1 ] (157.24,50) .. controls (157.24,48.48) and (158.48,47.24) .. (160,47.24) .. controls (161.52,47.24) and (162.76,48.48) .. (162.76,50) .. controls (162.76,51.52) and (161.52,52.76) .. (160,52.76) .. controls (158.48,52.76) and (157.24,51.52) .. (157.24,50) -- cycle ;
\draw  [draw opacity=0][fill={rgb, 255:red, 0; green, 0; blue, 0 }  ,fill opacity=1 ] (95.24,138) .. controls (95.24,136.48) and (96.48,135.24) .. (98,135.24) .. controls (99.52,135.24) and (100.76,136.48) .. (100.76,138) .. controls (100.76,139.52) and (99.52,140.76) .. (98,140.76) .. controls (96.48,140.76) and (95.24,139.52) .. (95.24,138) -- cycle ;
\draw  [draw opacity=0][fill={rgb, 255:red, 0; green, 0; blue, 0 }  ,fill opacity=1 ] (127,157) .. controls (127,155.48) and (128.23,154.24) .. (129.76,154.24) .. controls (131.28,154.24) and (132.51,155.48) .. (132.51,157) .. controls (132.51,158.52) and (131.28,159.76) .. (129.76,159.76) .. controls (128.23,159.76) and (127,158.52) .. (127,157) -- cycle ;
\draw  [draw opacity=0][fill={rgb, 255:red, 0; green, 0; blue, 0 }  ,fill opacity=1 ] (228,258) .. controls (228,256.48) and (229.23,255.24) .. (230.76,255.24) .. controls (232.28,255.24) and (233.51,256.48) .. (233.51,258) .. controls (233.51,259.52) and (232.28,260.76) .. (230.76,260.76) .. controls (229.23,260.76) and (228,259.52) .. (228,258) -- cycle ;
\draw  [draw opacity=0][fill={rgb, 255:red, 0; green, 0; blue, 0 }  ,fill opacity=1 ] (246.24,288) .. controls (246.24,286.48) and (247.48,285.24) .. (249,285.24) .. controls (250.52,285.24) and (251.76,286.48) .. (251.76,288) .. controls (251.76,289.52) and (250.52,290.76) .. (249,290.76) .. controls (247.48,290.76) and (246.24,289.52) .. (246.24,288) -- cycle ;
\draw  [draw opacity=0][fill={rgb, 255:red, 0; green, 0; blue, 0 }  ,fill opacity=1 ] (127.12,255.12) -- (132.88,255.12) -- (132.88,260.88) -- (127.12,260.88) -- cycle ;
\draw  [draw opacity=0][fill={rgb, 255:red, 0; green, 0; blue, 0 }  ,fill opacity=1 ] (176.24,105.24) -- (182,105.24) -- (182,111) -- (176.24,111) -- cycle ;
\draw  [draw opacity=0][fill={rgb, 255:red, 0; green, 0; blue, 0 }  ,fill opacity=1 ] (246.12,255.12) -- (251.88,255.12) -- (251.88,260.88) -- (246.12,260.88) -- cycle ;
\draw  [draw opacity=0][fill={rgb, 255:red, 0; green, 0; blue, 0 }  ,fill opacity=1 ] (157.12,75.12) -- (162.88,75.12) -- (162.88,80.88) -- (157.12,80.88) -- cycle ;
\draw  [draw opacity=0][fill={rgb, 255:red, 0; green, 0; blue, 0 }  ,fill opacity=1 ] (246.12,315.12) -- (251.88,315.12) -- (251.88,320.88) -- (246.12,320.88) -- cycle ;
\draw  [draw opacity=0][fill={rgb, 255:red, 0; green, 0; blue, 0 }  ,fill opacity=1 ] (276.24,204.24) -- (282,204.24) -- (282,210) -- (276.24,210) -- cycle ;
\draw  [draw opacity=0][fill={rgb, 255:red, 0; green, 0; blue, 0 }  ,fill opacity=1 ] (126.24,135) -- (132,135) -- (132,140.76) -- (126.24,140.76) -- cycle ;
\draw  [draw opacity=0][fill={rgb, 255:red, 0; green, 0; blue, 0 }  ,fill opacity=1 ] (306.24,225.24) -- (312,225.24) -- (312,231) -- (306.24,231) -- cycle ;
\draw  [draw opacity=0][fill={rgb, 255:red, 0; green, 0; blue, 0 }  ,fill opacity=1 ] (67.12,135.12) -- (72.88,135.12) -- (72.88,140.88) -- (67.12,140.88) -- cycle ;

\draw (133,182.4) node [anchor=north west][inner sep=0.75pt]    {$a_{0}$};
\draw (286,150.4) node [anchor=north west][inner sep=0.75pt]    {$a_{0}$};
\draw (173,232.4) node [anchor=north west][inner sep=0.75pt]    {$b_{0}$};
\draw (221,111.4) node [anchor=north west][inner sep=0.75pt]    {$b_{0}$};
\draw (133,138.4) node [anchor=north west][inner sep=0.75pt]  [font=\small]  {$a_{1}$};
\draw (263,213.4) node [anchor=north west][inner sep=0.75pt]  [font=\small]  {$a_{1}$};
\draw (162,111.4) node [anchor=north west][inner sep=0.75pt]  [font=\small]  {$b_{1}$};
\draw (232,236.4) node [anchor=north west][inner sep=0.75pt]  [font=\small]  {$b_{1}$};
\draw (166,82.4) node [anchor=north west][inner sep=0.75pt]  [font=\small]  {$b_{2}$};
\draw (289,208.4) node [anchor=north west][inner sep=0.75pt]  [font=\small]  {$a_{2}$};
\draw (319,209.4) node [anchor=north west][inner sep=0.75pt]  [font=\small]  {$a_{3}$};
\draw (109,141.4) node [anchor=north west][inner sep=0.75pt]  [font=\small]  {$a_{2}$};
\draw (79,143.4) node [anchor=north west][inner sep=0.75pt]  [font=\small]  {$a_{3}$};
\draw (165,56.4) node [anchor=north west][inner sep=0.75pt]  [font=\small]  {$b_{3}$};
\draw (231,266.4) node [anchor=north west][inner sep=0.75pt]  [font=\small]  {$b_{2}$};
\draw (231,293.4) node [anchor=north west][inner sep=0.75pt]  [font=\small]  {$b_{3}$};

\end{tikzpicture}
\end{center}
\caption{The 2-complex $\Sigma_2$ corresponding to the integer $d=3$. }
    \label{fig:integer infinite 2-complex}
\end{figure}

A 2-complex is a topological surface with boundary if the link of every point is homeomorphic to $S^1$ or $[0,1]$. Note that $\Sigma_2$ has vertices with infinitely many vertices in their equivalence class. 
In this case, such a vertex necessarily has link $\mathbb R$. Indeed, this is the case for $\text{link}([(0,0)]_{\sim})$ and $\text{link}([(1,1)]_\sim)$. Observe that in \Cref{fig:integer infinite 2-complex}, $[(0,0)]_{\sim}$ is the collection of black square vertices and $[(1,1)]_{\sim}$ is the collection of black circular vertices. 

Let $\Sigma$ be the surface built from $\Sigma_2$ by first replacing $[(0,0)]_{\sim}$ and $([(1,1)]_\sim$ with their links and then doubling across the boundary. Observe that $\Sigma$ is a closed connected surface which is homeomorphic to the ladder surface, and $f_1$ uniquely extends to a homeomorphism $f:\Sigma \rightarrow \Sigma.$

The induced regions in $\Sigma$ coming from the blue and orange extended rectangles in \Cref{fig:integer infinite 2-complex}, together with their doubles, define a nesting neighborhood of the attracting end of $\Sigma$. The induced regions in $\Sigma$ coming from the green and red extended rectangles, together with their doubles define a nesting neighborhood of the repelling end of $\Sigma$. Therefore $f$ is end-periodic by construction. Moreover, the rectangular decomposition of the subsurface of $\Sigma$ coming from the purple region in the figure, together with its double, define a Markov decomposition for a core of $f$, with incidence matrix $\begin{bmatrix}
    d & 0 \\ 0 & d
\end{bmatrix}$. Therefore, the stretch factor of $f$ is $d$. \end{proof}

\section{A Piecewise Linear Map}\label{linear map}

For the remaining sections of the paper, fix an $n \times n$ irreducible matrix $M=(m_{ij})$ with non-negative integer entries whose spectral radius is equal to $\lambda$. Let $\vec{\eta} = (\eta_1, \eta_2, \dots, \eta_n)$ (resp. $\vec{\omega}=(\omega_1, \omega_2, \dots,\omega_n)$) be the right (resp. left) Perron-Frobenius eigenvector for $M$. 

\begin{example}\label[example]{example: run ex}
    (Running Example) Let 
    $$M = \begin{bmatrix}
        0 & 0 & 1 & 0 \\
        1 & 0 & 0 & 1 \\
        0 & 0 & 0 & 1\\
        1 & 2 & 0 & 0
    \end{bmatrix}.$$
    The characteristic polynomial of $M$ is $x^4-2x^2-x-2$. The spectral radius of $M$, $\lambda$, is approximately $1.785$. Observe that since det$(M) = -2$, $\lambda$ is not an algebraic unit, and hence not the stretch factor of any pseudo-Anosov map on a finite-type surface. Approximately, $\vec{\eta}= (0.31, \ 0.74, \ 0.56, \ 1)$ and $\vec{\omega}=(1.19, \ 1.12, \ 0.67, \ 1)$. 
\end{example}

Let $\Sigma_0$ be the disjoint union of $n$ rectangles denoted $Q_1$, \dots, $Q_n$ with heights given by $\vec{\eta}$ and widths given by $\vec{\omega}$. Partition each $Q_k$ into $\sum_{i=1}^{n}m_{ik}$ many vertical strips. Denote the interiors of the vertical strips $V^{(k)}_{i,j}$ for $j \in \{1, \dots, m_{ik}\}$, ordered from left to right:
$$V^{(k)}_{1,1}, \dots, V^{(k)}_{1,m_{1k}},V^{(k)}_{2,1}, \dots, V^{(k)}_{2,m_{2k}}, \dots, V^{(k)}_{n,1}, \dots, V^{(k)}_{n,m_{nk}}.$$
 Observe that since $\vec{\omega}$ is a left eigenvector of $M$ with eigenvalue $\lambda$, 
$$\sum_{i=1}^n m_{ik}( \lambda^{-1} \omega_i) = \omega_k.$$
Thus we can choose the vertical strips so that each $V^{(k)}_{i,j}$ has width $\lambda^{-1} \omega_i $.

Partition each $Q_k$ into $\sum_{i=1}^{n}m_{ki}$ many horizontal strips. Denote the interiors of the horizontal strips $H^{(k)}_{i,j}$ for $j \in \{1, \dots, m_{ki}\}$, ordered from top to bottom:
$$H^{(k)}_{1,1}, \dots, H^{(k)}_{1,m_{1k}},H^{(k)}_{2,1}, \dots, H^{(k)}_{2,m_{2k}}, \dots, H^{(k)}_{n,1}, \dots, H^{(k)}_{n,m_{nk}}.$$
Since $\vec{\eta}$ is a right eigenvector of $M$ with eigenvalue $\lambda$, 
$$\sum_{i=1}^n m_{ki}(\lambda^{-1} \eta_i) = \eta_k.$$
Hence we can choose horizontal strips so that each $H^{(k)}_{i,j}$ has height $\lambda^{-1} \eta_i$.\\

\begin{definition} Consider the foliation of $\Sigma_0$ by vertical lines segments and the foliation of $\Sigma_0$ by horizontal line segments. These foliations can be straightened to laminations of $\Sigma_0$ which we will call the \emph{vertical and horizontal laminations of $\Sigma_0$} and denote by $\mathcal V_0$ and $\mathcal H_0$, respectively.
\end{definition}

For each rectangle $Q_i$, choose a bijection $$\sigma_i: \{H^{(i)}_{k,j} \ | \ 1 \leq k \leq n, 1 \leq j \leq m_{ik}\} \circlearrowleft$$
and a bijection
$$\tau_i: \{V^{(i)}_{k,j} \ | \ 1 \leq k \leq n, 1 \leq j \leq m_{ki}\} \circlearrowleft.$$
Resize the strips so that $\tau_i^{-1}(V_{k,j}^{(i)})$ has width $\lambda^{-1}\omega_k$ and $\sigma_i(H^{(i)}_{k,j})$ has height $\lambda^{-1}\eta_k$. \\

Let $$V = \bigcup_{k,i,j} V^{(k)}_{i,j}$$ be the union  of the interiors of the vertical strips and let $$H = \bigcup_{k,i,j} H^{(k)}_{i,j}$$ be the union of the interiors of the horizontal strips. \\

Let $\sigma :H \rightarrow H$ be the piecewise linear homeomorphism mapping $H^{(i)}_{j,k}$ onto $\sigma_i(H^{(i)}_{j,k})$ preserving left/right/top/bottom. Similarly, let $\tau:V \rightarrow V$
be the piecewise linear homeomorphism mapping $V^{(i)}_{j,k}$ onto $\tau_i(V^{(i)}_{j,k})$ preserving left/right/top/bottom.\\

\begin{definition}
    The \textit{piece map corresponding to $M$, $\{\sigma_i\}_{i=1}^n$, and $\{\tau_i\}_{i=1}^n$} is the piecewise linear homeomorphism 
    $$f_0: V \rightarrow H$$ 
    mapping $\tau_k^{-1}(V^{(k)}_{i,j})$ bijectively onto $\sigma_i(H^{(i)}_{k,j})$, maintaining left/right and top/bottom orientation.  Note that $V^{(k)}_{i,j}$ is a vertical strip in $\Sigma_0$ exactly when $1 \leq i,k \leq n$ and $1 \leq j \leq m_{ik}$. Similarly $H^{(i)}_{k,j}$ is a horizontal strip under the exact same conditions on $i,j,k$. Thus $f_0$ is well-defined. 
    
    Equivalently, we can define $f_0:V \rightarrow H$ by $$f_0 =\sigma \circ w \circ \tau$$
     where $w:V \rightarrow H$ is the piecewise linear map homeomorphism mapping $V^{(i)}_{j,k}$ onto $H^{(j)}_{i,k}$ preserving  left/right/top/bottom orientation.
\end{definition} 

By our choice of widths and heights of the strips, each vertical strip is scaled under $f_0$ horizontally by a factor of $\lambda$ and vertically by a factor $\lambda^{-1}$. Moreover, $f_0$ preserves the laminations $\mathcal{V}_0$ and $\mathcal{H}_0$.

     For each rectangle $Q_i$, let $V^{(i)}_{\text{min}}$ be the vertical strip appearing on the leftmost side of $Q_i$. In other words, $\text{min}$ is equal to $(p,1)$ where $p$ is the minimal index in the $i^{th}$ column of $M$ which has a non-zero entry. Similarly, let $H^{(i)}_{\text{min}}$ be the horizontal strip appearing on the topmost side of $Q_i$.

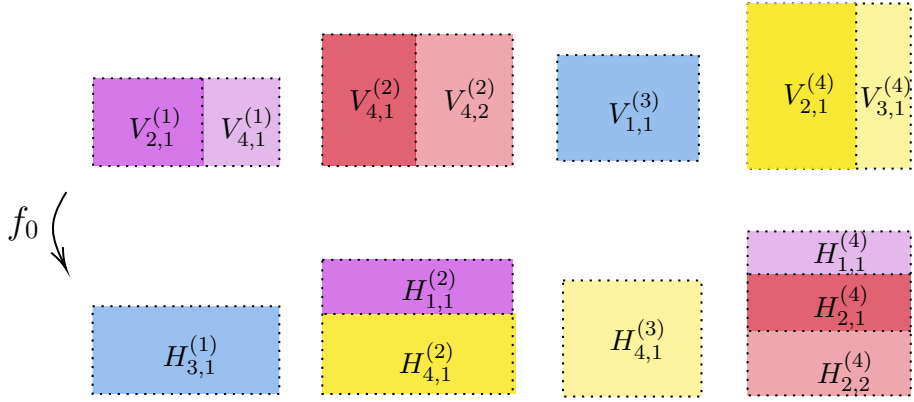
\begin{figure}[htbp]
\begin{center}

\tikzset{every picture/.style={line width=0.75pt}} 

\begin{tikzpicture}[x=0.75pt,y=0.75pt,yscale=-1,xscale=1]

\draw  [draw opacity=0][fill={rgb, 255:red, 207; green, 132; blue, 221 }  ,fill opacity=0.57 ][dash pattern={on 0.84pt off 2.51pt}] (107.2,58.77) -- (146.1,58.77) -- (146.1,103.06) -- (107.2,103.06) -- cycle ;
\draw  [fill={rgb, 255:red, 74; green, 144; blue, 226 }  ,fill opacity=0.58 ][dash pattern={on 0.84pt off 2.51pt}] (286.28,47.05) -- (357.7,47.05) -- (357.7,100.45) -- (286.28,100.45) -- cycle ;
\draw  [fill={rgb, 255:red, 248; green, 231; blue, 28 }  ,fill opacity=0.42 ][dash pattern={on 0.84pt off 2.51pt}] (382.27,21) -- (464.55,21) -- (464.55,104.36) -- (382.27,104.36) -- cycle ;
\draw  [draw opacity=0][fill={rgb, 255:red, 197; green, 69; blue, 223 }  ,fill opacity=0.72 ][dash pattern={on 0.84pt off 2.51pt}] (52.2,58.77) -- (107.74,58.77) -- (107.74,103.06) -- (52.2,103.06) -- cycle ;
\draw  [draw opacity=0][fill={rgb, 255:red, 207; green, 132; blue, 221 }  ,fill opacity=0.57 ][dash pattern={on 0.84pt off 2.51pt}] (382.27,136.04) -- (464.55,136.04) -- (464.55,157.66) -- (382.27,157.66) -- cycle ;
\draw  [fill={rgb, 255:red, 74; green, 144; blue, 226 }  ,fill opacity=0.57 ][dash pattern={on 0.84pt off 2.51pt}] (51.93,173.81) -- (146.1,173.81) -- (146.1,218.1) -- (51.93,218.1) -- cycle ;
\draw  [draw opacity=0][fill={rgb, 255:red, 248; green, 231; blue, 28 }  ,fill opacity=0.83 ][dash pattern={on 0.84pt off 2.51pt}] (167.77,177.72) -- (265.12,177.72) -- (265.12,218.1) -- (167.77,218.1) -- cycle ;
\draw  [draw opacity=0][fill={rgb, 255:red, 197; green, 69; blue, 223 }  ,fill opacity=0.72 ][dash pattern={on 0.84pt off 2.51pt}] (167.77,150.37) -- (263.8,150.37) -- (263.8,177.72) -- (167.77,177.72) -- cycle ;
\draw  [fill={rgb, 255:red, 248; green, 231; blue, 28 }  ,fill opacity=0.42 ][dash pattern={on 0.84pt off 2.51pt}] (289.2,160.79) -- (359.02,160.79) -- (359.02,219.4) -- (289.2,219.4) -- cycle ;
\draw    (38.71,116.18) .. controls (28.75,132.01) and (30.87,138.63) .. (37.92,153.59) ;
\draw [shift={(38.71,155.25)}, rotate = 244.6] [color={rgb, 255:red, 0; green, 0; blue, 0 }  ][line width=0.75]    (10.93,-3.29) .. controls (6.95,-1.4) and (3.31,-0.3) .. (0,0) .. controls (3.31,0.3) and (6.95,1.4) .. (10.93,3.29)   ;
\draw  [draw opacity=0][fill={rgb, 255:red, 248; green, 231; blue, 28 }  ,fill opacity=0.81 ][dash pattern={on 0.84pt off 2.51pt}] (382.27,21) -- (437.05,21) -- (437.05,104.36) -- (382.27,104.36) -- cycle ;
\draw  [dash pattern={on 0.84pt off 2.51pt}]  (437.05,21) -- (437.05,104.36) ;
\draw  [draw opacity=0][fill={rgb, 255:red, 208; green, 2; blue, 27 }  ,fill opacity=0.62 ][dash pattern={on 0.84pt off 2.51pt}] (382.27,157.66) -- (464.55,157.66) -- (464.55,186.32) -- (382.27,186.32) -- cycle ;
\draw  [draw opacity=0][fill={rgb, 255:red, 208; green, 2; blue, 27 }  ,fill opacity=0.35 ][dash pattern={on 0.84pt off 2.51pt}] (382.27,186.32) -- (464.55,186.32) -- (464.55,218.81) -- (382.27,218.81) -- cycle ;
\draw  [dash pattern={on 0.84pt off 2.51pt}] (382.27,136.04) -- (464.55,136.04) -- (464.55,218.81) -- (382.27,218.81) -- cycle ;
\draw  [dash pattern={on 0.84pt off 2.51pt}]  (464.55,157.66) -- (382.27,157.66) ;
\draw  [dash pattern={on 0.84pt off 2.51pt}]  (464.55,186.32) -- (382.27,186.32) ;
\draw  [dash pattern={on 0.84pt off 2.51pt}] (167.77,150.37) -- (263.8,150.37) -- (263.8,218.1) -- (167.77,218.1) -- cycle ;
\draw  [dash pattern={on 0.84pt off 2.51pt}]  (263.8,177.72) -- (167.77,177.72) ;
\draw  [draw opacity=0][fill={rgb, 255:red, 208; green, 2; blue, 27 }  ,fill opacity=0.35 ][dash pattern={on 0.84pt off 2.51pt}] (214.87,36.63) -- (263.8,36.63) -- (263.8,103.06) -- (214.87,103.06) -- cycle ;
\draw  [draw opacity=0][fill={rgb, 255:red, 208; green, 2; blue, 27 }  ,fill opacity=0.62 ][dash pattern={on 0.84pt off 2.51pt}] (167.77,36.63) -- (214.87,36.63) -- (214.87,103.06) -- (167.77,103.06) -- cycle ;
\draw  [dash pattern={on 0.84pt off 2.51pt}]  (214.87,36.63) -- (214.87,103.06) ;
\draw  [dash pattern={on 0.84pt off 2.51pt}] (52.2,58.77) -- (146.1,58.77) -- (146.1,103.06) -- (52.2,103.06) -- cycle ;
\draw  [dash pattern={on 0.84pt off 2.51pt}]  (107.2,58.77) -- (107.2,103.06) ;
\draw  [dash pattern={on 0.84pt off 2.51pt}] (167.77,36.63) -- (263.8,36.63) -- (263.8,103.06) -- (167.77,103.06) -- cycle ;

\draw (7.23,121.9) node [anchor=north west][inner sep=0.75pt]  [font=\large]  {$f_{0}$};
\draw (68.71,71.71) node [anchor=north west][inner sep=0.75pt]  [font=\small]  {$V_{2,1}^{( 1)}$};
\draw (114.99,71.71) node [anchor=north west][inner sep=0.75pt]  [font=\small]  {$V_{4,1}^{( 1)}$};
\draw (398.65,54.78) node [anchor=north west][inner sep=0.75pt]  [font=\small]  {$V_{2,1}^{( 4)}$};
\draw (204.92,190.26) node [anchor=north west][inner sep=0.75pt]  [font=\small]  {$H_{4,1}^{( 2)}$};
\draw (310.37,63.9) node [anchor=north west][inner sep=0.75pt]  [font=\small]  {$V_{1,1}^{( 3)}$};
\draw (206.25,153.49) node [anchor=north west][inner sep=0.75pt]  [font=\small]  {$H_{1,1}^{( 2)}$};
\draw (85.87,186.75) node [anchor=north west][inner sep=0.75pt]  [font=\small]  {$H_{3,1}^{( 1)}$};
\draw (311.37,178.94) node [anchor=north west][inner sep=0.75pt]  [font=\small]  {$H_{4,1}^{( 3)}$};
\draw (413.88,134.56) node [anchor=north west][inner sep=0.75pt]  [font=\small]  {$H_{1,1}^{( 4)}$};
\draw (437.33,56.08) node [anchor=north west][inner sep=0.75pt]  [font=\small]  {$V_{3,1}^{( 4)}$};
\draw (413.88,161.91) node [anchor=north west][inner sep=0.75pt]  [font=\small]  {$H_{2,1}^{( 4)}$};
\draw (415.86,195.22) node [anchor=north west][inner sep=0.75pt]  [font=\small]  {$H_{2,2}^{( 4)}$};
\draw (179.8,57.38) node [anchor=north west][inner sep=0.75pt]  [font=\small]  {$V_{4,1}^{( 2)}$};
\draw (227.41,57.38) node [anchor=north west][inner sep=0.75pt]  [font=\small]  {$V_{4,2}^{( 2)}$};

\end{tikzpicture}
\end{center}
\caption{\textbf{Running Example:} The piece map $f_0:V \rightarrow H$ corresponding to $M$ in Example \ref{example: run ex} and each $\sigma_k$ and $\tau_k$ equal to the identity map. Colors indicate the image of each region. }
    \label{fig:running ex f0}
\end{figure}

Observe that there is no continuous extension of $f_0$ to points in $\Sigma_0$ which lie both on the right edge of some vertical strip and on the left edge of a different vertical strip. Similarly, there is no continuous extension of $f_0^{-1}$ to points which lie both on the top edge of some horizontal strip and the bottom edge of a different horizontal strip. 

To manage the two different natural extensions of $f_0$ to these vertical edges in the interior of $\Sigma_0$, we will eventually identify certain segments of the left edges of the rectangles with certain segments of the right edges of the rectangles. Similarly, to manage the two different natural extensions of $f_0^{-1}$ to the horizontal edges in the interior of $\Sigma_0$, we will identify certain segments of the top edges of the rectangles with certain segments of the bottom edges of the rectangles.

If we made these identifications on $\Sigma_0$ now, there would be points on the boundary of $\Sigma_0$ where these edge identifications would accumulate. For this reason, we will wait to make these identifications until we have glued infinite strips onto $\Sigma_0$, which we will do in Section \ref{infinite strips}. 


\section{The Edge Maps}\label{edge maps}

Let \[L := \bigcup_{k,i,j} \text{Left edge}(V^{(k)}_{i,j}) \text{ and } R := \bigcup_{k,i,j} \text{Right edge}(V^{(k)}_{i,j})\]
be the \textit{left and right edges of the vertical strips}. Let
\[T := \bigcup_{k,i,j} \text{Top edge}(H^{(k)}_{i,j}) \text{ and }B := \bigcup_{k,i,j} \text{Bottom edge}(H^{(k)}_{i,j})\]
be the \textit{top and bottom edges of the horizontal strips}.

\begin{definition} (Edge maps of $f_0$)
   \begin{enumerate}
       \item The \textit{left edge map of $f_0$} $$f_{L}:L \rightarrow L$$ is given by
       $f_{L}(x) := \lim_{i \rightarrow \infty}f_0(x_i)$, where $x = \lim_{i \rightarrow \infty}x_i$ for $x_i \in V$ converging to $x$ from the right. By the definition of $f_0$, $f_L$ maps Left edge$(\tau_k^{-1}(V_{i,j}^{(k)}))$ bijectively onto Left edge$(\sigma_i(H_{k,j}^{(i)}))$.
       \item The \textit{right edge map of $f$} is $$f_{R}:R \rightarrow R$$ is given by
       $f_{R}(x) := \lim_{i \rightarrow \infty}f_0(x_i)$, where $x = \lim_{i \rightarrow \infty}x_i$ for $x_i \in V$ converging to $x$ from the left. By the definition of $f_0$, $f_R$ maps Right edge$(\tau_k^{-1}(V_{i,j}^{(k)}))$ bijectively onto Right edge$(\sigma_i(H_{k,j}^{(i)}))$
       \item The \textit{top edge map of $f_0$} is $$f_{T}^{-1}:T \rightarrow T$$ is given by
       $f_{T}^{-1}(x) := \lim_{i \rightarrow \infty}f_0^{-1}(x_i)$, where $x = \lim_{i \rightarrow \infty}x_i$ for $x_i \in H$ converging to $x$ from below. By the definition of $f_0$, $f_T^{-1}$ maps Top edge$(\sigma_k(H_{i,j}^{(k)}))$ bijectively onto Top edge$(\tau_i^{-1}(V_{k,j}^{(i)}))$.
       \item The \textit{bottom edge map of $f_0$} is $$f_{B}^{-1}:B \rightarrow B$$ is given by
       $f_{B}^{-1}(x) := \lim_{i \rightarrow \infty}f_0^{-1}(x_i)$, where $x = \lim_{i \rightarrow \infty}x_i$ for $x_i \in H$ converging to $x$ from above. By the definition of $f_0$, $f_B^{-1}$ maps  Bottom edge$(\sigma_k(H_{i,j}^{(k)}))$ bijectively onto Bottom edge$(\tau_i^{-1}((V_{k,j}^{(i)}))$.
   \end{enumerate}
\end{definition}

Let $X := L \cap R$ and $Y := T \cap B$. These are the unions of the vertical and horizontal edges which do not lie on the boundaries of the $Q_i$. Observe that both $f_L$ and $f_R$ are defined on $X$ and both $f_T^{-1}$ and $f_B^{-1}$ are defined on $Y$. The discrepancy between these maps is precisely what prevents us from continuously extending $f_0$ and $f_0^{-1}$. To alleviate this, we will eventually need to identify certain points on the boundary.

\begin{definition}
    We define four directed graphs, each with vertex set $\{v_1, \dots, v_n\}$. \begin{itemize}
    \item[1.] The \textit{left edge digraph} $D_L$ has directed edges
  $$E^+(D_L) = \{(v_i,v_j) | \ f_L(\text{Left edge}(Q_i)) \subseteq \text{ Left edge}(Q_j)\}.$$
    \item[2.] The \textit{right edge digraph} $D_R$ has directed edges $$E^+(D_R) = \{(v_i,v_j) | \ f_L(\text{Right edge}(Q_i)) \subseteq \text{Right edge}(Q_j)\}$$
    \item[3.] The \textit{top edge digraph} $D_T$ has directed edges
    $$E^+(D_T) = \{(v_i,v_j) | \ f_T^{-1}(\text{Top edge}(Q_i)) \subseteq \text{Top edge}(Q_j)\}$$
    \item[4.] The \textit{bottom edge digraph} $D_B$ has directed edges
    $$E^+(D_B) = \{(v_i,v_j) | \ f_B^{-1}(\text{Bottom edge}(Q_i)) \subseteq \text{Bottom edge}(Q_j)\}.$$
\end{itemize}
\end{definition}

By construction, $D_L$ and $D_R$  (resp. $D_T$ and $D_B$) are subgraphs of the directed graph corresponding to $M$ (resp. $M^T$). 


\begin{figure}[htbp]
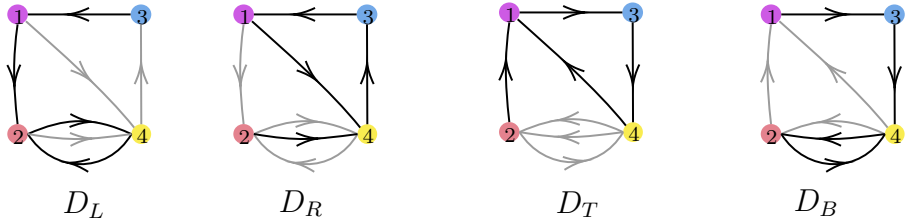

\begin{center}

\tikzset{every picture/.style={line width=0.75pt}} 


\end{center}
\caption{\textbf{Running Example:} The four edge digraphs corresponding to the map $f_0$ shown in \Cref{fig:running ex f0}. The black directed edges are in the edge digraphs, while the light gray edges are the remaining edges in the digraphs corresponding to $M$ for $D_L$ and $D_R$ and corresponding to $M^T$ for $D_T$ and $D_B$. }
    \label{fig:running ex digraphs}
\end{figure}

\begin{lemma}
    The left (resp. right, top, bottom) edge of $Q_i$ contains a single periodic point under $f_L$ (resp. $f_R$, $f_T^{-1}$, $f_B^{-1}$) if and only if $v_i$ is a periodic vertex in $D_L$ (resp. $D_R$, $D_T$, $D_R$). Otherwise, the edge has no periodic points. Thus, there are finitely many points in the boundary of $\Sigma_0$ which are periodic under $f_L$, $f_R$, $f_T^{-1}$, or $f_B^{-1}$. 
\end{lemma}
\begin{proof} Observe that a vertex $v_i$ has period $p$ in the digraph $D_L$ if and only if $$f_L^p(\text{Left edge}(Q_i)) \subseteq  \text{Left edge}(Q_i).$$ Moreover, irreducibility of the matrix $M$ guarantees that the map $$f^p_L : \text{Left edge}(Q_i) \rightarrow \text{Left edge}(Q_i)$$ is a contraction. By the Banach fixed-point theorem, the contraction has exactly one fixed point.

Now suppose $x \in L$ is a period $p$ point of $f_L$, so $f_L^p(x)=x$. Since $f_L(L)$ is contained in the left edges of the $Q_i$, we must have $x \in \text{Left edge}(Q_i)$ for some $i$. Since the image of each left edge is contained in a single other left edge, we must have $f_L^p(\text{Left edge}(Q_i)) \subseteq  \text{Left edge}(Q_i).$ Thus $v_i$ is a periodic vertex in $D_L$ and $$f^p_L : \text{Left edge}(Q_i) \rightarrow \text{Left edge}(Q_i)$$ is a contraction. Thus $x$ must be the unique fixed point.

Hence the left edge of each rectangle $Q_i$ can have at most one periodic point, and thus there are at most $n$ periodic points of $f_L$. In other words, \[\bigcap_{i=0}^{\infty} f^i_L(L)\]
consists of at most $n$ points. An analogous argument holds for $f_R,f^{-1}_T,$ and $f^{-1}_B$.
\end{proof}

Observe that if a periodic point of $f_L$ (resp. $f_{R}$) is a top left or bottom left (resp. top right or bottom right) corner of a rectangle $Q_i$, then it is necessarily also a periodic point of $f^{-1}_T$ or $f^{-1}_B$ (of the same period). Conversely, a periodic point of $f^{-1}_T$ or $f^{-1}_B$ which is a corner point of a rectangle $Q_i$ is also a periodic point of $f_{L}$ or $f_R$. 

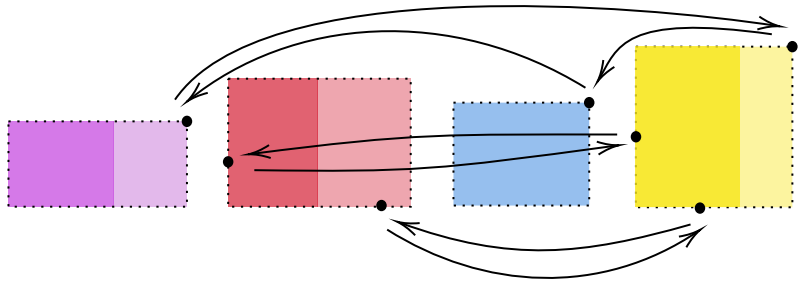
\begin{figure}[htbp]
\begin{center}

\tikzset{every picture/.style={line width=0.75pt}} 

\begin{tikzpicture}[x=0.75pt,y=0.75pt,yscale=-1,xscale=1]

\draw  [draw opacity=0][fill={rgb, 255:red, 207; green, 132; blue, 221 }  ,fill opacity=0.57 ][dash pattern={on 0.84pt off 2.51pt}] (63.93,60.63) -- (101.22,60.63) -- (101.22,103.71) -- (63.93,103.71) -- cycle ;
\draw  [fill={rgb, 255:red, 74; green, 144; blue, 226 }  ,fill opacity=0.58 ][dash pattern={on 0.84pt off 2.51pt}] (235.63,51.18) -- (304.1,51.18) -- (304.1,103.12) -- (235.63,103.12) -- cycle ;
\draw  [fill={rgb, 255:red, 248; green, 231; blue, 28 }  ,fill opacity=0.42 ][dash pattern={on 0.84pt off 2.51pt}] (327.66,22.92) -- (406.56,22.92) -- (406.56,104) -- (327.66,104) -- cycle ;
\draw  [draw opacity=0][fill={rgb, 255:red, 197; green, 69; blue, 223 }  ,fill opacity=0.72 ][dash pattern={on 0.84pt off 2.51pt}] (11.2,60.63) -- (64.45,60.63) -- (64.45,103.71) -- (11.2,103.71) -- cycle ;
\draw  [draw opacity=0][fill={rgb, 255:red, 248; green, 231; blue, 28 }  ,fill opacity=0.81 ][dash pattern={on 0.84pt off 2.51pt}] (327.66,22.92) -- (380.18,22.92) -- (380.18,104) -- (327.66,104) -- cycle ;
\draw  [draw opacity=0][fill={rgb, 255:red, 208; green, 2; blue, 27 }  ,fill opacity=0.35 ][dash pattern={on 0.84pt off 2.51pt}] (167.16,39.09) -- (214.08,39.09) -- (214.08,103.71) -- (167.16,103.71) -- cycle ;
\draw  [draw opacity=0][fill={rgb, 255:red, 208; green, 2; blue, 27 }  ,fill opacity=0.62 ][dash pattern={on 0.84pt off 2.51pt}] (122.01,39.09) -- (167.16,39.09) -- (167.16,103.71) -- (122.01,103.71) -- cycle ;
\draw  [dash pattern={on 0.84pt off 2.51pt}] (11.2,60.63) -- (101.22,60.63) -- (101.22,103.71) -- (11.2,103.71) -- cycle ;
\draw  [dash pattern={on 0.84pt off 2.51pt}] (122.01,39.09) -- (214.08,39.09) -- (214.08,103.71) -- (122.01,103.71) -- cycle ;
\draw  [draw opacity=0][fill={rgb, 255:red, 0; green, 0; blue, 0 }  ,fill opacity=1 ] (98.58,60.63) .. controls (98.58,59.05) and (99.77,57.76) .. (101.22,57.76) .. controls (102.68,57.76) and (103.87,59.05) .. (103.87,60.63) .. controls (103.87,62.22) and (102.68,63.5) .. (101.22,63.5) .. controls (99.77,63.5) and (98.58,62.22) .. (98.58,60.63) -- cycle ;
\draw  [draw opacity=0][fill={rgb, 255:red, 0; green, 0; blue, 0 }  ,fill opacity=1 ] (119.36,81.06) .. controls (119.36,79.47) and (120.54,78.19) .. (122,78.19) .. controls (123.46,78.19) and (124.64,79.47) .. (124.64,81.06) .. controls (124.64,82.64) and (123.46,83.92) .. (122,83.92) .. controls (120.54,83.92) and (119.36,82.64) .. (119.36,81.06) -- cycle ;
\draw  [draw opacity=0][fill={rgb, 255:red, 0; green, 0; blue, 0 }  ,fill opacity=1 ] (196.7,103.1) .. controls (196.7,101.52) and (197.88,100.24) .. (199.34,100.24) .. controls (200.8,100.24) and (201.98,101.52) .. (201.98,103.1) .. controls (201.98,104.69) and (200.8,105.97) .. (199.34,105.97) .. controls (197.88,105.97) and (196.7,104.69) .. (196.7,103.1) -- cycle ;
\draw  [draw opacity=0][fill={rgb, 255:red, 0; green, 0; blue, 0 }  ,fill opacity=1 ] (301.46,51.18) .. controls (301.46,49.59) and (302.64,48.31) .. (304.1,48.31) .. controls (305.56,48.31) and (306.75,49.59) .. (306.75,51.18) .. controls (306.75,52.76) and (305.56,54.04) .. (304.1,54.04) .. controls (302.64,54.04) and (301.46,52.76) .. (301.46,51.18) -- cycle ;
\draw  [draw opacity=0][fill={rgb, 255:red, 0; green, 0; blue, 0 }  ,fill opacity=1 ] (403.91,22.92) .. controls (403.91,21.34) and (405.1,20.05) .. (406.56,20.05) .. controls (408.02,20.05) and (409.2,21.34) .. (409.2,22.92) .. controls (409.2,24.5) and (408.02,25.79) .. (406.56,25.79) .. controls (405.1,25.79) and (403.91,24.5) .. (403.91,22.92) -- cycle ;
\draw  [draw opacity=0][fill={rgb, 255:red, 0; green, 0; blue, 0 }  ,fill opacity=1 ] (357.32,104.4) .. controls (357.32,102.82) and (358.5,101.53) .. (359.96,101.53) .. controls (361.42,101.53) and (362.6,102.82) .. (362.6,104.4) .. controls (362.6,105.99) and (361.42,107.27) .. (359.96,107.27) .. controls (358.5,107.27) and (357.32,105.99) .. (357.32,104.4) -- cycle ;
\draw  [draw opacity=0][fill={rgb, 255:red, 0; green, 0; blue, 0 }  ,fill opacity=1 ] (325.04,68.41) .. controls (325.04,66.83) and (326.22,65.55) .. (327.68,65.55) .. controls (329.14,65.55) and (330.33,66.83) .. (330.33,68.41) .. controls (330.33,70) and (329.14,71.28) .. (327.68,71.28) .. controls (326.22,71.28) and (325.04,70) .. (325.04,68.41) -- cycle ;
\draw    (302.2,43.67) .. controls (234.54,1.88) and (154.01,7.61) .. (101.98,50.03) ;
\draw [shift={(101.2,50.67)}, rotate = 320.41] [color={rgb, 255:red, 0; green, 0; blue, 0 }  ][line width=0.75]    (10.93,-3.29) .. controls (6.95,-1.4) and (3.31,-0.3) .. (0,0) .. controls (3.31,0.3) and (6.95,1.4) .. (10.93,3.29)   ;
\draw    (95.2,49.67) .. controls (138.98,-14.01) and (327.3,1.51) .. (399.13,12.5) ;
\draw [shift={(400.2,12.67)}, rotate = 188.81] [color={rgb, 255:red, 0; green, 0; blue, 0 }  ][line width=0.75]    (10.93,-3.29) .. controls (6.95,-1.4) and (3.31,-0.3) .. (0,0) .. controls (3.31,0.3) and (6.95,1.4) .. (10.93,3.29)   ;
\draw    (396.2,16.67) .. controls (320.93,6.05) and (318.3,23.38) .. (309.22,38.98) ;
\draw [shift={(308.2,40.67)}, rotate = 302.01] [color={rgb, 255:red, 0; green, 0; blue, 0 }  ][line width=0.75]    (10.93,-3.29) .. controls (6.95,-1.4) and (3.31,-0.3) .. (0,0) .. controls (3.31,0.3) and (6.95,1.4) .. (10.93,3.29)   ;
\draw    (355.2,112.67) .. controls (300.75,128.11) and (264.92,132.19) .. (208.9,111.89) ;
\draw [shift={(207.2,111.27)}, rotate = 20.22] [color={rgb, 255:red, 0; green, 0; blue, 0 }  ][line width=0.75]    (10.93,-3.29) .. controls (6.95,-1.4) and (3.31,-0.3) .. (0,0) .. controls (3.31,0.3) and (6.95,1.4) .. (10.93,3.29)   ;
\draw    (318.2,67.27) .. controls (231.63,67.27) and (217.34,66.28) .. (133.47,77.1) ;
\draw [shift={(132.2,77.27)}, rotate = 352.63] [color={rgb, 255:red, 0; green, 0; blue, 0 }  ][line width=0.75]    (10.93,-3.29) .. controls (6.95,-1.4) and (3.31,-0.3) .. (0,0) .. controls (3.31,0.3) and (6.95,1.4) .. (10.93,3.29)   ;
\draw    (135.2,85.27) .. controls (216.79,86.26) and (226.11,85.28) .. (317.81,72.86) ;
\draw [shift={(319.2,72.67)}, rotate = 172.28] [color={rgb, 255:red, 0; green, 0; blue, 0 }  ][line width=0.75]    (10.93,-3.29) .. controls (6.95,-1.4) and (3.31,-0.3) .. (0,0) .. controls (3.31,0.3) and (6.95,1.4) .. (10.93,3.29)   ;
\draw    (202.2,115.27) .. controls (250.71,145.96) and (310,149.21) .. (358.73,116.28) ;
\draw [shift={(360.2,115.27)}, rotate = 145.24] [color={rgb, 255:red, 0; green, 0; blue, 0 }  ][line width=0.75]    (10.93,-3.29) .. controls (6.95,-1.4) and (3.31,-0.3) .. (0,0) .. controls (3.31,0.3) and (6.95,1.4) .. (10.93,3.29)   ;

\end{tikzpicture}
\end{center}
\caption{\textbf{Running Example:} The periodic points of the edge maps corresponding to $f_0$ in \Cref{fig:running ex f0}. Arrows indicate the orbits of the periodic points under the corresponding edge map.}
\label{fig:running ex digraphs}
\end{figure}




\section{Gluing Infinite Strips}\label{infinite strips}

In this section we will tackle the following to-do list:
\begin{enumerate}
    \item Glue infinite strips onto $\Sigma_0$ to form $\Sigma_1$.
    \item Partition $\Sigma_1$ in two ways: $V_1$ and $H_1$.
    \item Build a homeomorphism $f_1:V_1 \rightarrow H_1$ which is equal to $f_0$ far enough away from the periodic points of the edge maps of $f_0$.
    \item Show that the edge maps of $f_1$ have no periodic points. 
    \item Build a 2-complex $\Sigma_2$ as a quotient of $\Sigma_1$ by identifying points which are the image of the same point under different edge maps.
    \item Modify $\Sigma_2$ to build a surface $\Sigma$, called the surface corresponding to $M$.
\end{enumerate}

We begin by introducing some notation to help us keep track of the information needed to define the map $f_1$. 

For each finite orbit $C$ in each of the four edge maps of $f_0$, choose a particular periodic point of $C$, called \textit{the initial point of $C$}. Choose initial points so that if $x$ is a periodic point under two edge maps (i.e. $x$ is a periodic corner of a rectangle $Q_i$) and $x$ is the initial point of an orbit $C$ in one edge map, then $x$ is also the initial point of the orbit $C'$ containing $x$ in the second edge map. 

The map $f_1$ will act by a shift in each of the infinite strips added to $\Sigma_0$ to form $\Sigma_1$, but we will only shift once per cycle at the initial point of each orbit. 

For each periodic point $x$ of $f_L$, let the \textit{left infinite strip corresponding to $x$} be
    $$E_L(x) = [0, 1] \times [0, \infty).$$ 
    Similarly define $E_R(x)$, $E_T(x)$ and $E_B(x)$. Let $(z,w)_{L,x}$ denote the point $(z,w) \in E_L(x)$, and similarly for the other infinite strips. 

    Define the \textit{expanded rectangles corresponding to $f_0$} to be 
        \[\Sigma_1 := \Sigma_0  \ \sqcup_{g} \ \left( \Big{(}\bigcup_x E_L(x)\Big{)} \cup \Big{(}\bigcup_x E_R(x)\Big{)} \cup \Big{(}\bigcup_x E_T(x)\Big{)} \cup \Big{(}\bigcup_x E_B(x)\Big{)}\right)\]
    where each union is taken over the periodic points of the corresponding edge map and the infinite strips are glued onto $\Sigma_0$ via a piecewise linear map $g$, defined as follows: the domain of $g$ is the union of $[0,1] \times \{0\}$ for each infinite strip and its range is the union of certain neighborhoods of the periodic points. 

    Let $x_0,x_1, \dots,x_{p-1}$ be a finite orbit of periodic points under $f_L$, with $x_0 \in$ Left edge$(Q_i)$ an initial periodic point. For each $0 \leq j \leq p-1$, the map $g$ on $([0,1] \times \{0\})_{L,x_j}$, is the linear map with image in $L$ such that \[g(([0,1] \times \{0\})_{L,x_j}) = f_L^{2p+j}(\text{Left edge}(Q_i))\] and \[g((0,0)_{L,x_j}) = f_L^{2p+j}(\text{Bottom left corner}(Q_i)).\]


    Similarly, let $x_0,x_1,\dots,x_{p-1}$ be a finite orbit of periodic points under $f_T^{-1}$ with $x_0 \in$ Top edge$(Q_i)$ an initial periodic point. For each $0 \leq j \leq p-1$, the map $g$ on $([0,1] \times \{0\})_{T,x_j}$, is the linear map with image in $T$ such that \[g(([0,1] \times \{0\})_{T,x_j}) = f_T^{-2p-j}(\text{Top edge}(Q_i))\] and \[g((0,0)_{T,x_j}) = f_T^{-2p-j}(\text{Top left corner}(Q_i)).\]

    We can define $g$ in the same way for periodic points of $f_R$ and $f_B^{-1}$.

\begin{figure}[htbp]
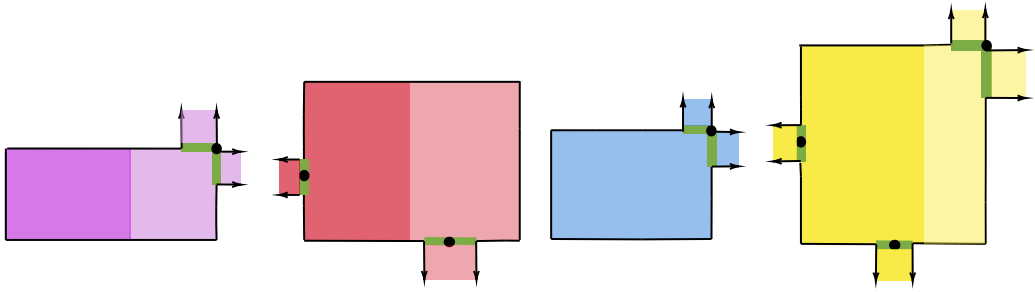

\begin{center}

\tikzset{every picture/.style={line width=0.75pt}} 


\end{center}
\caption{\textbf{Running Example:} The expanded rectangles corresponding to $f_0$ in \Cref{fig:running ex f0}. The green regions indicate where the infinite strips are glued on via the map $g$.}
\label{fig:}
\end{figure}


    For each $V_{i,j}^{(k)}$, let $(V_{i,j}^{(k)})'$ be the subset of $\Sigma_1$ consisting of the interior of the union of $\overline{(V_{i,j}^{(k)})}$ and all infinite strips which have been glued to the boundary of $V_{i,j}^{(k)}$. Define each $(H_{i,j}^{(k)})'$ similarly, and let 
    $$V_1 = \bigcup (V_{i,j}^{(k)})'$$
    and
    $$H_1 = \bigcup (H_{i,j}^{(k)})'.$$

We will construct a homeomorphism $f_1 : V_1 \to H_1$ that agrees with $f_0$ outside a neighborhood of the infinite strips and equals a translation sufficiently far into each infinite strip. We begin by specifying the regions where the behavior of $f_1$ transitions from that of $f_0$ to that of a shift.


\begin{definition}[Switch regions]
Let $x$ be a periodic point of an edge map of $f_0$. 
The \emph{switch region corresponding to $x$}, denoted $P(x)$, 
will be the region of $\Sigma_1$ where the map $f_1$ switches  
from equaling $f_0$ to equaling translation along the infinite strips. We define this region first for the non-corner periodic points of $f_L$ and $f_R$, then for the non-corner periodic points of $f_T^{-1}$ and $f_B^{-1}$. Finally, we define this region for the corner periodic points of the edge maps. 

\medskip

\begin{itemize}
 
\item[(i)] \textbf{Non-corner periodic points of Left and Right edge Maps.}
Let $x_0, \dots, x_{p-1}$ be the orbit of a non-corner periodic point of $f_L$ 
with period $p$, labeled so that $x_0$ is the initial periodic point.
Let $Q_i$ be the rectangle containing $x_0$. Let $W$ be a path joining 
$f_L^p(\mathrm{Top\ left}(Q_i))$ to 
$f_L^p(\mathrm{Bottom\ left}(Q_i))$ 
such that 
\[
f_0^j(\mathrm{int}(W)) \subseteq V_0 \cap H_0 
\quad \text{for } 0 \le j \le p.
\]
Let $I \subset \mathrm{Left\ edge}(Q_i)$ be the subsegment with the same endpoints.
For $0 \le j \le p-1$, define the switch region corresponding to $x_j$, denoted $P(x_j)$, to be the region of $\Sigma_1$ bounded by $f_0^j(W) \cup f_L^j(I)$
(see Figure ~\ref{fig: switch maps left}).

 \begin{figure}[htbp]
  \begin{center}

\tikzset{every picture/.style={line width=0.75pt}} 

\begin{tikzpicture}[x=0.75pt,y=0.75pt,yscale=-1,xscale=1]

\draw  [draw opacity=0][fill={rgb, 255:red, 194; green, 239; blue, 229 }  ,fill opacity=0.48 ] (82.64,49.27) -- (176,49.27) -- (176,134.71) -- (82.64,134.71) -- cycle ;
\draw    (82.64,49.27) -- (82.64,133.02) ;
\draw    (82.64,49.27) -- (176,49.27) ;
\draw    (82.64,134.71) -- (175.63,134.71) ;
\draw  [draw opacity=0][fill={rgb, 255:red, 128; green, 128; blue, 128 }  ,fill opacity=1 ] (79.91,49.27) .. controls (79.91,47.93) and (81.13,46.85) .. (82.64,46.85) .. controls (84.15,46.85) and (85.37,47.93) .. (85.37,49.27) .. controls (85.37,50.61) and (84.15,51.7) .. (82.64,51.7) .. controls (81.13,51.7) and (79.91,50.61) .. (79.91,49.27) -- cycle ;
\draw  [draw opacity=0][fill={rgb, 255:red, 128; green, 128; blue, 128 }  ,fill opacity=1 ] (79.91,133.02) .. controls (79.91,131.68) and (81.13,130.6) .. (82.64,130.6) .. controls (84.15,130.6) and (85.37,131.68) .. (85.37,133.02) .. controls (85.37,134.36) and (84.15,135.45) .. (82.64,135.45) .. controls (81.13,135.45) and (79.91,134.36) .. (79.91,133.02) -- cycle ;
\draw    (8,73.07) -- (83.01,73.07) ;
\draw [shift={(6,73.07)}, rotate = 0] [color={rgb, 255:red, 0; green, 0; blue, 0 }  ][line width=0.75]    (10.93,-3.29) .. controls (6.95,-1.4) and (3.31,-0.3) .. (0,0) .. controls (3.31,0.3) and (6.95,1.4) .. (10.93,3.29)   ;
\draw    (8,99.07) -- (82.64,99.18) ;
\draw [shift={(6,99.07)}, rotate = 0.08] [color={rgb, 255:red, 0; green, 0; blue, 0 }  ][line width=0.75]    (10.93,-3.29) .. controls (6.95,-1.4) and (3.31,-0.3) .. (0,0) .. controls (3.31,0.3) and (6.95,1.4) .. (10.93,3.29)   ;
\draw  [draw opacity=0][fill={rgb, 255:red, 194; green, 239; blue, 229 }  ,fill opacity=0.48 ] (6,72.96) -- (83.01,72.96) -- (83.01,99.07) -- (6,99.07) -- cycle ;
\draw    (82.21,59.2) .. controls (129.72,63.45) and (135.4,111.67) .. (82.27,119.26) ;
\draw  [dash pattern={on 0.84pt off 2.51pt}]  (159.1,50.01) -- (159.1,134.6) ;
\draw  [draw opacity=0][fill={rgb, 255:red, 189; green, 16; blue, 224 }  ,fill opacity=0.2 ] (82.27,119.26) -- (82.21,59.2) -- (100.72,63.54) -- (116.22,73.69) -- (120.36,90.61) -- (115.19,103.3) -- (99.69,115.15) -- cycle ;
\draw  [draw opacity=0][fill={rgb, 255:red, 128; green, 128; blue, 128 }  ,fill opacity=1 ] (80.3,60.16) .. controls (80.3,58.82) and (81.52,57.73) .. (83.03,57.73) .. controls (84.54,57.73) and (85.76,58.82) .. (85.76,60.16) .. controls (85.76,61.5) and (84.54,62.58) .. (83.03,62.58) .. controls (81.52,62.58) and (80.3,61.5) .. (80.3,60.16) -- cycle ;
\draw  [draw opacity=0][fill={rgb, 255:red, 128; green, 128; blue, 128 }  ,fill opacity=1 ] (80.21,119.26) .. controls (80.21,117.92) and (81.43,116.84) .. (82.94,116.84) .. controls (84.45,116.84) and (85.67,117.92) .. (85.67,119.26) .. controls (85.67,120.6) and (84.45,121.69) .. (82.94,121.69) .. controls (81.43,121.69) and (80.21,120.6) .. (80.21,119.26) -- cycle ;
\draw  [draw opacity=0][fill={rgb, 255:red, 0; green, 0; blue, 0 }  ,fill opacity=1 ] (79.51,86.81) .. controls (79.51,85.47) and (80.73,84.38) .. (82.24,84.38) .. controls (83.75,84.38) and (84.97,85.47) .. (84.97,86.81) .. controls (84.97,88.14) and (83.75,89.23) .. (82.24,89.23) .. controls (80.73,89.23) and (79.51,88.14) .. (79.51,86.81) -- cycle ;
\draw  [draw opacity=0][fill={rgb, 255:red, 155; green, 155; blue, 155 }  ,fill opacity=0.23 ] (303.67,49.27) -- (388,49.27) -- (388,134.71) -- (303.67,134.71) -- cycle ;
\draw    (303.67,49.27) -- (303.67,133.02) ;
\draw    (303.67,49.27) -- (388,49.27) ;
\draw    (303.67,134.71) -- (388,134.71) ;
\draw  [draw opacity=0][fill={rgb, 255:red, 128; green, 128; blue, 128 }  ,fill opacity=1 ] (300.94,49.27) .. controls (300.94,47.93) and (302.16,46.85) .. (303.67,46.85) .. controls (305.17,46.85) and (306.4,47.93) .. (306.4,49.27) .. controls (306.4,50.61) and (305.17,51.7) .. (303.67,51.7) .. controls (302.16,51.7) and (300.94,50.61) .. (300.94,49.27) -- cycle ;
\draw  [draw opacity=0][fill={rgb, 255:red, 128; green, 128; blue, 128 }  ,fill opacity=1 ] (300.94,133.02) .. controls (300.94,131.68) and (302.16,130.6) .. (303.67,130.6) .. controls (305.17,130.6) and (306.4,131.68) .. (306.4,133.02) .. controls (306.4,134.36) and (305.17,135.45) .. (303.67,135.45) .. controls (302.16,135.45) and (300.94,134.36) .. (300.94,133.02) -- cycle ;
\draw    (242,73.07) -- (303.3,73.07) ;
\draw [shift={(240,73.07)}, rotate = 0] [color={rgb, 255:red, 0; green, 0; blue, 0 }  ][line width=0.75]    (10.93,-3.29) .. controls (6.95,-1.4) and (3.31,-0.3) .. (0,0) .. controls (3.31,0.3) and (6.95,1.4) .. (10.93,3.29)   ;
\draw    (242,99.18) -- (303.67,99.18) ;
\draw [shift={(240,99.18)}, rotate = 0] [color={rgb, 255:red, 0; green, 0; blue, 0 }  ][line width=0.75]    (10.93,-3.29) .. controls (6.95,-1.4) and (3.31,-0.3) .. (0,0) .. controls (3.31,0.3) and (6.95,1.4) .. (10.93,3.29)   ;
\draw  [draw opacity=0][fill={rgb, 255:red, 155; green, 155; blue, 155 }  ,fill opacity=0.23 ] (240,73.07) -- (304.04,73.07) -- (304.04,99.18) -- (240,99.18) -- cycle ;
\draw    (303.23,59.2) .. controls (342.53,67.54) and (338.91,108.15) .. (303.3,119.26) ;
\draw  [dash pattern={on 0.84pt off 2.51pt}]  (380.12,50.01) -- (380.12,134.6) ;
\draw  [draw opacity=0][fill={rgb, 255:red, 189; green, 16; blue, 224 }  ,fill opacity=0.2 ] (303.3,119.26) -- (303.23,59.2) -- (320.57,66.08) -- (329.87,77.92) -- (330.9,86.38) -- (328.84,99.92) -- (317.47,111.76) -- cycle ;
\draw  [draw opacity=0][fill={rgb, 255:red, 128; green, 128; blue, 128 }  ,fill opacity=1 ] (301.32,60.16) .. controls (301.32,58.82) and (302.55,57.73) .. (304.05,57.73) .. controls (305.56,57.73) and (306.79,58.82) .. (306.79,60.16) .. controls (306.79,61.5) and (305.56,62.58) .. (304.05,62.58) .. controls (302.55,62.58) and (301.32,61.5) .. (301.32,60.16) -- cycle ;
\draw  [draw opacity=0][fill={rgb, 255:red, 128; green, 128; blue, 128 }  ,fill opacity=1 ] (301.23,119.26) .. controls (301.23,117.92) and (302.45,116.84) .. (303.96,116.84) .. controls (305.47,116.84) and (306.69,117.92) .. (306.69,119.26) .. controls (306.69,120.6) and (305.47,121.69) .. (303.96,121.69) .. controls (302.45,121.69) and (301.23,120.6) .. (301.23,119.26) -- cycle ;
\draw  [draw opacity=0][fill={rgb, 255:red, 0; green, 0; blue, 0 }  ,fill opacity=1 ] (300.53,86.81) .. controls (300.53,85.47) and (301.76,84.38) .. (303.27,84.38) .. controls (304.77,84.38) and (306,85.47) .. (306,86.81) .. controls (306,88.14) and (304.77,89.23) .. (303.27,89.23) .. controls (301.76,89.23) and (300.53,88.14) .. (300.53,86.81) -- cycle ;
\draw  [draw opacity=0][fill={rgb, 255:red, 194; green, 239; blue, 229 }  ,fill opacity=0.48 ] (521.64,50.47) -- (615,50.47) -- (615,135.91) -- (521.64,135.91) -- cycle ;
\draw    (521.64,50.47) -- (521.64,134.22) ;
\draw    (521.64,50.47) -- (615,50.47) ;
\draw    (521.64,135.91) -- (614.63,135.91) ;
\draw  [draw opacity=0][fill={rgb, 255:red, 128; green, 128; blue, 128 }  ,fill opacity=1 ] (518.91,50.47) .. controls (518.91,49.13) and (520.13,48.05) .. (521.64,48.05) .. controls (523.15,48.05) and (524.37,49.13) .. (524.37,50.47) .. controls (524.37,51.81) and (523.15,52.9) .. (521.64,52.9) .. controls (520.13,52.9) and (518.91,51.81) .. (518.91,50.47) -- cycle ;
\draw  [draw opacity=0][fill={rgb, 255:red, 128; green, 128; blue, 128 }  ,fill opacity=1 ] (518.91,134.22) .. controls (518.91,132.88) and (520.13,131.8) .. (521.64,131.8) .. controls (523.15,131.8) and (524.37,132.88) .. (524.37,134.22) .. controls (524.37,135.56) and (523.15,136.65) .. (521.64,136.65) .. controls (520.13,136.65) and (518.91,135.56) .. (518.91,134.22) -- cycle ;
\draw    (431.5,74.27) -- (521.27,74.27) ;
\draw [shift={(429.5,74.27)}, rotate = 0] [color={rgb, 255:red, 0; green, 0; blue, 0 }  ][line width=0.75]    (10.93,-3.29) .. controls (6.95,-1.4) and (3.31,-0.3) .. (0,0) .. controls (3.31,0.3) and (6.95,1.4) .. (10.93,3.29)   ;
\draw    (431.5,100.38) -- (521.35,100.27) ;
\draw [shift={(429.5,100.38)}, rotate = 359.93] [color={rgb, 255:red, 0; green, 0; blue, 0 }  ][line width=0.75]    (10.93,-3.29) .. controls (6.95,-1.4) and (3.31,-0.3) .. (0,0) .. controls (3.31,0.3) and (6.95,1.4) .. (10.93,3.29)   ;
\draw  [draw opacity=0][fill={rgb, 255:red, 194; green, 239; blue, 229 }  ,fill opacity=0.48 ] (429.5,74.27) -- (522.01,74.27) -- (522.01,100.38) -- (429.5,100.38) -- cycle ;
\draw  [dash pattern={on 0.84pt off 2.51pt}]  (598.1,51.21) -- (598.1,135.8) ;
\draw  [draw opacity=0][fill={rgb, 255:red, 128; green, 128; blue, 128 }  ,fill opacity=1 ] (519.3,61.36) .. controls (519.3,60.02) and (520.52,58.93) .. (522.03,58.93) .. controls (523.54,58.93) and (524.76,60.02) .. (524.76,61.36) .. controls (524.76,62.7) and (523.54,63.78) .. (522.03,63.78) .. controls (520.52,63.78) and (519.3,62.7) .. (519.3,61.36) -- cycle ;
\draw  [draw opacity=0][fill={rgb, 255:red, 128; green, 128; blue, 128 }  ,fill opacity=1 ] (519.21,120.46) .. controls (519.21,119.12) and (520.43,118.04) .. (521.94,118.04) .. controls (523.45,118.04) and (524.67,119.12) .. (524.67,120.46) .. controls (524.67,121.8) and (523.45,122.89) .. (521.94,122.89) .. controls (520.43,122.89) and (519.21,121.8) .. (519.21,120.46) -- cycle ;
\draw  [draw opacity=0][fill={rgb, 255:red, 0; green, 0; blue, 0 }  ,fill opacity=1 ] (518.51,88.01) .. controls (518.51,86.67) and (519.73,85.58) .. (521.24,85.58) .. controls (522.75,85.58) and (523.97,86.67) .. (523.97,88.01) .. controls (523.97,89.34) and (522.75,90.43) .. (521.24,90.43) .. controls (519.73,90.43) and (518.51,89.34) .. (518.51,88.01) -- cycle ;
\draw  [draw opacity=0][fill={rgb, 255:red, 189; green, 16; blue, 224 }  ,fill opacity=0.29 ] (489.5,73.54) -- (521.27,73.54) -- (521.27,99.65) -- (489.5,99.65) -- cycle ;
\draw    (489.5,74.27) -- (489.5,99.65) ;
\draw  [draw opacity=0][fill={rgb, 255:red, 128; green, 128; blue, 128 }  ,fill opacity=1 ] (79.25,99.07) .. controls (79.25,97.73) and (80.47,96.65) .. (81.98,96.65) .. controls (83.49,96.65) and (84.71,97.73) .. (84.71,99.07) .. controls (84.71,100.41) and (83.49,101.5) .. (81.98,101.5) .. controls (80.47,101.5) and (79.25,100.41) .. (79.25,99.07) -- cycle ;
\draw  [draw opacity=0][fill={rgb, 255:red, 128; green, 128; blue, 128 }  ,fill opacity=1 ] (80.28,72.96) .. controls (80.28,71.62) and (81.5,70.53) .. (83.01,70.53) .. controls (84.52,70.53) and (85.74,71.62) .. (85.74,72.96) .. controls (85.74,74.3) and (84.52,75.38) .. (83.01,75.38) .. controls (81.5,75.38) and (80.28,74.3) .. (80.28,72.96) -- cycle ;
\draw  [draw opacity=0][fill={rgb, 255:red, 128; green, 128; blue, 128 }  ,fill opacity=1 ] (300.65,99.41) .. controls (300.65,98.07) and (301.87,96.98) .. (303.38,96.98) .. controls (304.88,96.98) and (306.11,98.07) .. (306.11,99.41) .. controls (306.11,100.75) and (304.88,101.83) .. (303.38,101.83) .. controls (301.87,101.83) and (300.65,100.75) .. (300.65,99.41) -- cycle ;
\draw  [draw opacity=0][fill={rgb, 255:red, 128; green, 128; blue, 128 }  ,fill opacity=1 ] (301.68,73.3) .. controls (301.68,71.96) and (302.9,70.87) .. (304.41,70.87) .. controls (305.92,70.87) and (307.14,71.96) .. (307.14,73.3) .. controls (307.14,74.64) and (305.92,75.72) .. (304.41,75.72) .. controls (302.9,75.72) and (301.68,74.64) .. (301.68,73.3) -- cycle ;
\draw  [draw opacity=0][fill={rgb, 255:red, 128; green, 128; blue, 128 }  ,fill opacity=1 ] (518.62,100.27) .. controls (518.62,98.93) and (519.84,97.85) .. (521.35,97.85) .. controls (522.85,97.85) and (524.08,98.93) .. (524.08,100.27) .. controls (524.08,101.61) and (522.85,102.7) .. (521.35,102.7) .. controls (519.84,102.7) and (518.62,101.61) .. (518.62,100.27) -- cycle ;
\draw  [draw opacity=0][fill={rgb, 255:red, 128; green, 128; blue, 128 }  ,fill opacity=1 ] (519.65,74.16) .. controls (519.65,72.82) and (520.87,71.73) .. (522.38,71.73) .. controls (523.89,71.73) and (525.11,72.82) .. (525.11,74.16) .. controls (525.11,75.5) and (523.89,76.58) .. (522.38,76.58) .. controls (520.87,76.58) and (519.65,75.5) .. (519.65,74.16) -- cycle ;
\draw    (160,40) .. controls (174.86,33.51) and (181.71,32.17) .. (198.17,39.2) ;
\draw [shift={(200,40)}, rotate = 203.99] [color={rgb, 255:red, 0; green, 0; blue, 0 }  ][line width=0.75]    (10.93,-3.29) .. controls (6.95,-1.4) and (3.31,-0.3) .. (0,0) .. controls (3.31,0.3) and (6.95,1.4) .. (10.93,3.29)   ;
\draw    (230,40) .. controls (244.86,33.51) and (251.71,32.17) .. (268.17,39.2) ;
\draw [shift={(270,40)}, rotate = 203.99] [color={rgb, 255:red, 0; green, 0; blue, 0 }  ][line width=0.75]    (10.93,-3.29) .. controls (6.95,-1.4) and (3.31,-0.3) .. (0,0) .. controls (3.31,0.3) and (6.95,1.4) .. (10.93,3.29)   ;
\draw    (400,40) .. controls (414.86,33.51) and (421.71,32.17) .. (438.17,39.2) ;
\draw [shift={(440,40)}, rotate = 203.99] [color={rgb, 255:red, 0; green, 0; blue, 0 }  ][line width=0.75]    (10.93,-3.29) .. controls (6.95,-1.4) and (3.31,-0.3) .. (0,0) .. controls (3.31,0.3) and (6.95,1.4) .. (10.93,3.29)   ;
\draw    (525.11,74.16) .. controls (547.5,74.48) and (547,98.48) .. (524.08,100.27) ;
\draw  [draw opacity=0][fill={rgb, 255:red, 189; green, 16; blue, 224 }  ,fill opacity=0.2 ] (521.35,100.85) -- (522.38,74.16) -- (535,76) -- (542,85) -- (537,97) -- cycle ;

\draw (127.4,137) node [anchor=north west][inner sep=0.75pt]    {$Q_{i}$};
\draw (3,102.47) node [anchor=north west][inner sep=0.75pt]  [font=\footnotesize]  {$E_{L}( x_{0})$};
\draw (125.3,75.51) node [anchor=north west][inner sep=0.75pt]  [font=\footnotesize]  {$W$};
\draw (115.22,109.54) node [anchor=north west][inner sep=0.75pt]    {$V_{\text{min}}^{( i)}$};
\draw (64,77.4) node [anchor=north west][inner sep=0.75pt]  [font=\footnotesize]  {$x_{0}$};
\draw (171,9.4) node [anchor=north west][inner sep=0.75pt]    {$h_{x_{0}}$};
\draw (230,103.4) node [anchor=north west][inner sep=0.75pt]  [font=\footnotesize]  {$E_{L}( x_{p-1})$};
\draw (238,10.4) node [anchor=north west][inner sep=0.75pt]    {$h_{x_{p-2}}$};
\draw (336,89.39) node [anchor=north west][inner sep=0.75pt]  [font=\footnotesize]  {$f_{0}^{p-1}( W)$};
\draw (566.4,138.2) node [anchor=north west][inner sep=0.75pt]    {$Q_{i}$};
\draw (421,104.4) node [anchor=north west][inner sep=0.75pt]  [font=\footnotesize]  {$E_{L}( x_{0})$};
\draw (542,81.4) node [anchor=north west][inner sep=0.75pt]  [font=\footnotesize]  {$f_{0}^{p}( W)$};
\draw (401,8.4) node [anchor=north west][inner sep=0.75pt]    {$h_{x_{p-1}}$};
\draw (465,53.4) node [anchor=north west][inner sep=0.75pt]  [font=\footnotesize]  {$( 1,1)_{L,x_{0}}$};
\draw (467.42,103.73) node [anchor=north west][inner sep=0.75pt]  [font=\footnotesize]  {$( 0,1)_{L,x_{0}}$};
\draw (504,80.4) node [anchor=north west][inner sep=0.75pt]  [font=\footnotesize]  {$x_{0}$};
\draw (273.65,77.47) node [anchor=north west][inner sep=0.75pt]  [font=\footnotesize]  {$x_{p-1}$};
\draw (204,44.4) node [anchor=north west][inner sep=0.75pt]    {$\dotsc $};

\end{tikzpicture}
  \end{center}
  \caption{The switch regions $P(x_i)$ for an orbit $x_0, x_1, \dots, x_{p-1}$ of non-corner periodic points of $f_L$ for $x_0$ an initial periodic point are indicated by the first two purple regions on the left. The final purple region on the right indicates the image of the final switch map $h_{x_{p-1}}$. This final region is a neighborhood of the initial periodic point $x_0$, but it is different than the switch region $P(x_0)$, because we want to include a shift. }
  \label{fig: switch maps left}
  \end{figure}
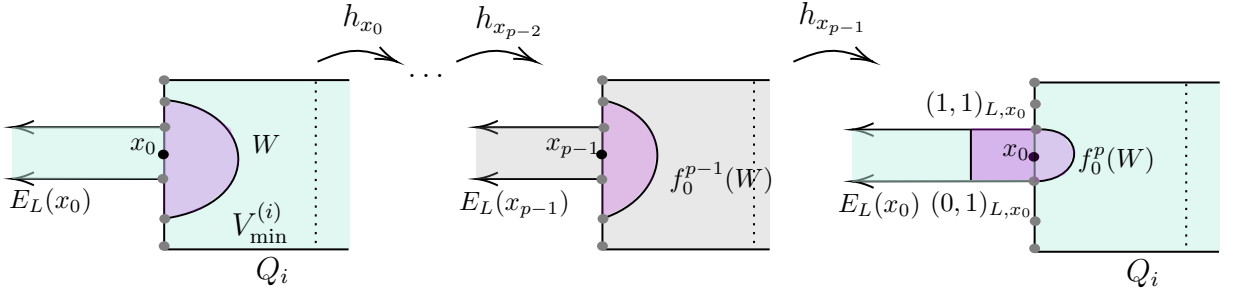

Switch regions for non-corner periodic points of $f_R$ are defined similarly.

\bigskip

\item[(ii)] \textbf{Non-corner periodic points of Top and Bottom Edge Maps.}
Let $x_0, \dots, x_{p-1}$ be the orbit of a non-corner periodic point of $f_T^{-1}$,
labeled so that $f_T(x_j)=x_{j+1}$ and $x_0$ is initial.
Let $W$ be a path joining 
$f_T^{-2p}(\mathrm{Top\ left}(Q_i))$ to 
$f_T^{-2p}(\mathrm{Top\ right}(Q_i))$
such that 
\[
f_0^j(\mathrm{int}(W)) \subseteq V_0 \cap H_0 
\quad \text{for } 0 \le j \le p.
\]

For $0 \le j \le p-1$, define $P(x_j)$ to be the region in $\Sigma_1$ bounded by $f_0^j(\mathrm{int}(W)$ and the left, top and right sides of the square $[0,1]\times[0,1] \subseteq E_T(x_j)$ 
(see \Cref{fig: top switch maps}).

Switch regions for non-corner periodic points of $f_B^{-1}$
are defined similarly.

 \begin{figure}[htbp]
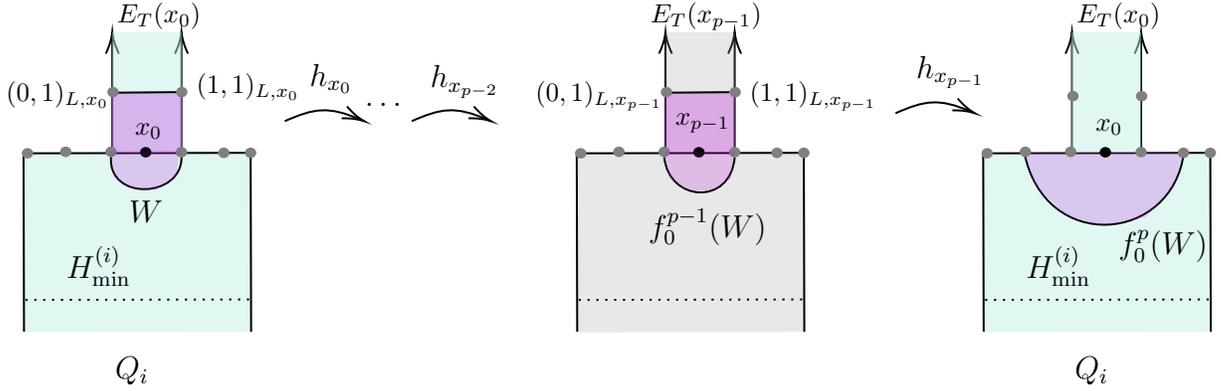

  \begin{center}

\tikzset{every picture/.style={line width=0.75pt}} 


  \end{center}
    \caption{The switch regions $P(x_i)$ for an orbit $x_0, x_1, \dots, x_{p-1}$ of non-corner periodic points of $f_T$ for $x_0$ an initial periodic point are indicated by the first two purple regions on the left. The final purple region on the right indicates the image of the final switch map $h_{x_{p-1}}$. This final region is a neighborhood of the initial periodic point $x_0$, but it is different than the switch region $P(x_0)$, because we want to include a shift.}
    \label{fig: top switch maps}
  \end{figure}


\item[(iii)] \textbf{Corner periodic points.}
Let $x_0, \dots, x_{p-1}$ be an orbit of corner periodic points under $f_L$
with $x_0 = \mathrm{Top\ left}(Q_i)$.
Let $W$ be a path joining 
$f_L^p(\mathrm{Bottom\ left}(Q_i))$ to 
$f_T^{-2p}(\mathrm{Top\ right}(Q_i))$
such that 
\[
f_0^j(\mathrm{int}(W)) \subseteq V_0 \cap H_0 
\quad \text{for } 0 \le j \le p.
\]

Let $I \subset \mathrm{Left\ edge}(Q_i)$ be the subsegment with the endpoints $f_L^{p}(\mathrm{Bottom\ left}(Q_i))$ and $x_0$. 

For $0 \le j \le p-1$, define $P(x_j)$ to be the region in $\Sigma_1$ bounded by $f_0^j(\mathrm{int}(W))$, the left, top, and right sides of the square $[0,1]\times[0,1] \subseteq E_T(x_j)$, as well $f_L^{p+j}(I)$ (see \Cref{fig:switch map corner}).

\begin{figure}[htbp]
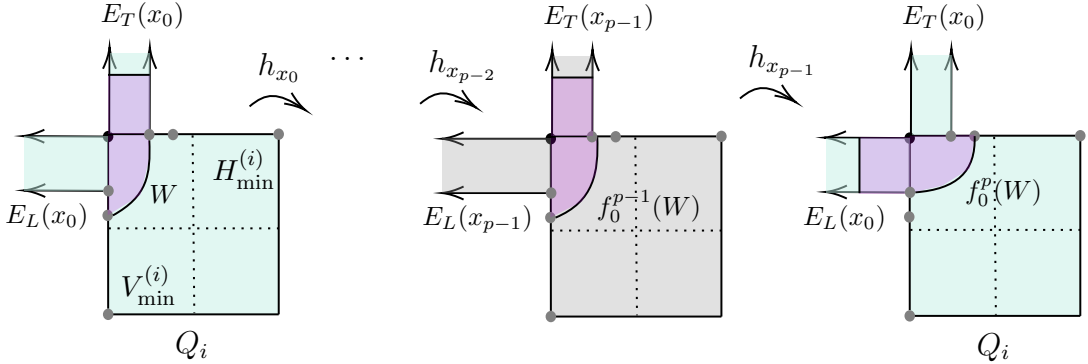

    \begin{center}

\tikzset{every picture/.style={line width=0.75pt}} 


    \end{center}
\caption{The switch region for $x_0$ is indicated in by the purple region in the left most expanded rectangle. The purple region in the rightmost expanded rectangle is the image of the switch map $h_{x_{p-1}}$. }
\label{fig:switch map corner}
    \end{figure}
     
Switch regions for the remaining corner periodic points 
are defined in the same manner.

\end{itemize}
\end{definition}

\begin{definition}
    Let $f_1:V_1 \rightarrow H_1$ be the \textit{extended piece map corresponding to $M$} given by
    \[f_1(y) = \begin{cases} f_0(y) & \text{if } y \in V_0 \text{ and } y \notin P(x) \text{ for any periodic }x \\
    h_x(y) & \text{if } y \in P(x) \\
    (z,w)_{L,f_L(x)} & \text{if } y = (z,w)_{L,x} - P(x) \text{ and }f_L(x) \text{ not initial}\\
    (z,w)_{R,f_R(x)} & \text{if } y = (z,w)_{R,x} - P(x) \text{ and }f_R(x) \text{ not initial}\\
    (z, w)_{T,f_T(x)} & \text{if } y = (z,w)_{T,x} - P(x) \text{ and }f_T(x) \text{ not initial}\\
    (z, w)_{B,f_B(x)} & \text{if } y = (z,w)_{B,x} - P(x) \text{ and }f_B(x) \text{ not initial}\\
    (z,w+1)_{L,f_L(x)} & \text{if } y = (z,w)_{L,x} - P(x) \text{ and }f_L(x) \text{ initial}\\
    (z,w+1)_{R,f_R(x)} & \text{if } y = (z,w)_{R,x} - P(x) \text{ and }f_R(x) \text{ initial}\\
    (z, w-1)_{T,f_T(x)} & \text{if } y = (z,w)_{T,x} - P(x) \text{ and }f_T(x) \text{ initial}\\
    (z, w-1)_{B,f_B(x)} & \text{if } y = (z,w)_{B,x} - P(x) \text{ and }f_B(x) \text{ initial}
    \end{cases}\]
    
    For each periodic point $x$ of the edge maps,  the \textit{switch map} $h_x$ is a continuous interpolation defined as follow. Let $x'=f_{*}(x)$ for $* = L, R, B,$ or $T$. If $x'$ is not an initial periodic point, then $$h_x:P(x) \rightarrow P(x').$$
    If $x'$ is an initial periodic point with period $p$, then 
    \[h_x:P(x) \rightarrow  ([0,1] \times [0,1])_{*,x'} \cup \overline{(f_0^p(\textrm{int}(P(x')))})\]
    for $* = L,R,B$ or $T$ according to which edge $x$ lies on. 
    
    Note that, in general, $f_T^{-1}$ (resp. $f_B^{-1}$) is defined,  but not $f^T$ (resp. $f^B$). However, in the definition of $f_1$ above, we only apply $f_T$ (resp. $f_B$) to points $x$ which are periodic under $f_T^{-1}$ (resp. $f_B^{-1}$), since in this case the inverse is well-defined.

\end{definition}


\begin{figure}[htbp]
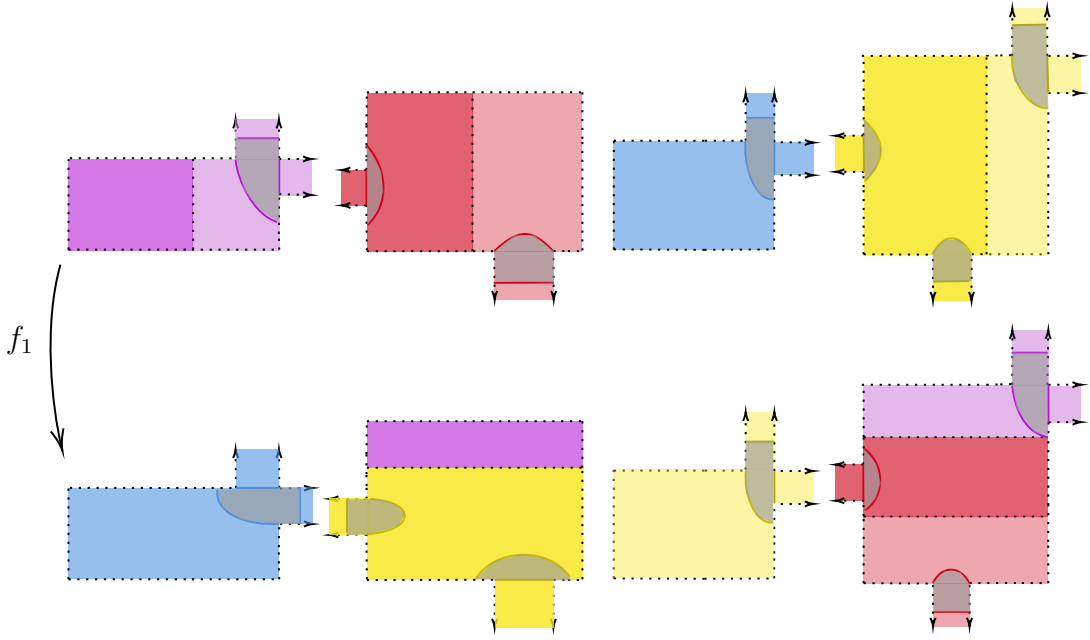

    \begin{center}

\tikzset{every picture/.style={line width=0.75pt}} 


    \end{center}
    \caption{\textbf{Running Example:} The extended piece map $f_1:V_1 \rightarrow H_1$ corresponding to $M$. The three periodic points in $Q_1$ and $Q_2$ are chosen to be the initial periodic points. The switch regions and their images are indicated in gray.}
    \label{fig:ex-piecemap}
\end{figure}

Next, we define four new edge maps coming from the different ways to extend $f_1$ and $f_1^{-1}$ to the boundary of the partitions $V_1$ and $H_1$. 

First, let $L_1$ be the union of the ``left edges" of each $(V_{i,j}^{(k)})'$, where we include the boundary of any infinite strips strips glued onto the left edge of a strip, but exclude the region of the left edge where the gluing takes place.  If an infinite strip is glued at the top of a left edge, only include the bottom boundary of the infinite strip in $L_1$. If an infinite strip is glued at the bottom of a left edge, only include the top boundary of the infinite strip in $L_1$. Similarly define $R_1$, $T_1$ and $B_1$.

\begin{notation} \label[notation]{from the right}
Every point $x \in L_1$ is the limit of a sequence of points $x_i \in V_1$. If $x$ is not in the boundary of an infinite strip, then the points $x_i$ can be taken to \textit{converge to $x$ from the right}. If $x$ is in the boundary of an infinite strip, then take any sequence in $V_1$ converging to $x$, and still refer to this sequence as \textit{converging to $x$ from the right.} We will use similar language of converging from the left for points in $R_1$. For points in $x \in T_1$ (resp. $B_1$), we will choose sequences \textit{converging to $x$ from below (resp. above).}
\end{notation}

\begin{definition} (Edge maps of $f_1$)
   \begin{enumerate}
       \item The \textit{left edge map of $f_1$} is $$f_{1,L}:L_1 \rightarrow L_1$$ is given by
       $f_{1,L}(x) := \lim_{i \rightarrow \infty}f_1(x_i)$, where the limit is taken in $\Sigma_1$ and $x = \lim_{i \rightarrow \infty}x_i$ for $x_i \in V_0$ converging to $x$ from the right.
       \item The \textit{right edge map of $f_1$} is $$f_{1,R}:R_1 \rightarrow R_1$$ is given by
       $f_{1,R}(x) := \lim_{i \rightarrow \infty}f_1(x_i)$, where the limit is taken in $\Sigma_1$ and $x = \lim_{i \rightarrow \infty}x_i$ for $x_i \in V_1$ converging to $x$ from the left.
       \item The \textit{top edge map of $f_1$} is $$f_{1,T}^{-1}:T_1 \rightarrow T_1$$ is given by
       $f_{1,T}^{-1}(x) := \lim_{i \rightarrow \infty}f_1^{-1}(x_i)$, where the limit is taken in $\Sigma_1$ and $x = \lim_{i \rightarrow \infty}x_i$ for $x_i \in H_1$ converging to $x$ from below.
       \item The \textit{bottom edge map of $f_1$} is $$f_{1,B}^{-1}:B_1 \rightarrow B_1$$ is given by
       $f_{1,B}^{-1}(x) := \lim_{i \rightarrow \infty}f_1^{-1}(x_i)$, where the limit is taken in $\Sigma_1$ and $x = \lim_{i \rightarrow \infty}x_i$ for $x_i \in H_1$ converging to $x$ from above.
   \end{enumerate}
\end{definition}

Because we have restricted the direction from which a sequence approaches a point in $L_1,R_1,T_1$, or $B_1$, these four continuous extensions are well-defined. \\

\begin{lemma}\label[lemma]{lemma: finite non-injective}
If $x \in L_1$ and $x$ is not a corner of a vertical strip in $V_1$, then $f_{1,L}^N$ is injective at $x$ for every power $N \in \mathbb{Z}_{\geq 1}$. Moreover, if $x \in L_1$ is a corner of a vertical strip, then for each $N \in \mathbb{Z}_{\geq 1}$ there is at most one other point $w \in L_1$ such that $f_{1,L}^N(w)=f^N_{1,L}(x)$. 

Analogous statements hold for $f_{1,R}$, $f_{1,T}^{-1}$ and $f_{1,B}^{-1}$. 
\end{lemma}

\begin{proof}
    First, suppose $f_{1,L}(x)=f_{1,L}(y)$ for distinct $x,y \in L_1$. Let $\{x_i\}$ (resp. $\{y_i\}$) be sequences in $V_1$ converging to $x$ (resp. $y$) from the right. Some tail of $\{f_1(x_i)\}$ is contained in a single horizontal strip in $H_1$, and similarly for $\{f_1(y_i)\}$. If both tails were contained in the same horizontal strip, then, since $f_1:V_1\rightarrow H_1$ is a homeomorphism, we would have $x=y$. Hence the tails are contained in different horizontal strips. Thus $f_{1,L}(x)$ is equal to the left end point of the top edge of some horizontal strip as well as the left end point of the bottom edge of a different horizontal strip. Thus by definition of $f_1$, we must have that $x$ and $y$ are both left corners of different vertical strips, one a bottom left corner and one a top left corner.  

    Now suppose $f_{1,L}^N(x)=f_{1,L}^N(y)$ for distinct points $x,y \in L_1$. By the previous paragraph, either $f_{1,L}^{N-1}(x)=f_{1,L}^{N-1}(y)$, or $f_{1,L}^{N-1}(x)$ and $f_{1,L}^{N-1}(y)$ are corners of different vertical strips. Let $n \in \mathbb{Z}_{\geq 1}$ be the smallest positive integer such that $f_{1,L}^n(x)=f_{1,L}^n(y)$. Thus $f_{1,L}^{n-1}(x)$ and $f_{1,L}^{n-1}(y)$ are corners of different vertical strips. If $n=1$, we are done, so suppose $n\geq 2$. Then $f_{1,L}^{n-1}(x)$ and $f_{1,L}^{n-1}(y)$ are contained in the left edges of the rectangles, and are hence each a corner of a rectangle, and thus also a corner of a horizontal strip. The only points which get mapped to corners of horizontal strips under powers of $f_{1,L}$ are corners of vertical strips. Hence $x$ and $y$ are corners. 

    The final statement in the lemma follows from observing that at most two corners of vertical strips get mapped to the same point under $f_{1,L}$, and the pre-image of corners of vertical strips under $f_{1,L}$ is contained in the set of corners of vertical strips. 
    
    The proof for the other edge maps is similar. Note that at most two corners of horizontal strips get mapped to the same point under $f_{1,T}^{-1}$ or $f_{1,B}^{-1}$.
\end{proof}

We have constructed $f_1$ so that for sufficiently large $N$, the $N^{\text{th}}$ iterate of an edge map sends every edge point to a point lying in the boundary of one of the infinite strips attached to $\Sigma_0$. Once an iterate enters an infinite strip, $f_1$ acts as a shift, ensuring that all subsequent images form an infinite discrete set. We formalize this in the following lemma.

\begin{lemma}\label[lemma]{lemma: infinite orbit}
    Every orbit of the edge maps of $f_1$ is infinite and discrete. 
\end{lemma}

\begin{proof}
Let $x_0 \in L_1$ and define 
$$x_j := f_{L,1}^j(x_0)$$
for each $j \in \mathbb{Z}_+$. Either $x_0 \in \partial (E_L(z) - P(z))$ or $x_0 \in \partial P(z)$ for some $z \in L_0$ which has period $p$ under $f_L$, or $x_0 \in $ Left edge$(Q_i)$ for some $i$, or $x_0 \in \textrm{int}(\Sigma_0)$. 

\begin{itemize}
    \item[(i)] First, suppose $x_0 = (\alpha, \beta)_{L,z} \in \partial ( E_L(z) - P(z))$. Then
    $$x_0 = \lim_{k \rightarrow \infty} (\alpha_k,\beta_k)_{L,z}$$ for some sequence $(\alpha_k,\beta_k)_{L,z}$ in $\mathrm{int}(E_L(z)- P(z))$. In this case, 
$$f_1^{pm+s}(\alpha_k,\beta_k) = (\alpha_k,\beta_k+m)_{L,f_L^s(z)}$$
 for each $m \in \mathbb{Z}_+$ and each $0 \leq s \leq p-1$. Hence
$$f_{1,L}^{pm+s}(\alpha,\beta) = (\alpha, \beta + m)_{L,f_L^s(z)}.$$
Thus the orbit of $x_0$ is an infinite and discrete set in $\Sigma_1$. 

\item[(ii)] Now, suppose $x_0 \in \partial P(z)$ and
    $$x_0 = \lim_{k \rightarrow \infty}a_k$$ for some sequence $a_k$ in $\mathrm{int}(P(z))$. By definition of the switch maps, 
$$f_1^{2p}(a_k) := (f_1^p \circ h_{z_{p-1}} \circ h_{z_{p-2}} \circ \dots  \circ h_{z_0})(a_k) \in \mathrm{int}(E_L(z) - P(z)).$$

Hence $x_{2p}$ is the limit of points in $\mathrm{int}(E_L(z)-P(z))$. Therefore the orbit of $x_0$ is infinite and discrete by part (i). 

\item[(iii)] Suppose $x \in \mathrm{Left \ edge}(Q_i)$. As long as $x_j$ is on a left edge without any periodic point of $f_L$, then $x_{j+1}=f_{1,L}(x_j)=f_L(x_j)$. Since every long enough directed path in the digraph $D_L$ eventually contains a periodic vertex, there is a power $N$ such that $x_{N}=f^N_L(x_0)$ is on a left edge of $Q_k$ containing a periodic point of $f_L$. Choose $N$ to be the first time this happens. In this case, 
$$x_{N+1} = f_{1,L}(x_N) \in \partial P(z).$$ 
Therefore the orbit of $x_0$ is infinite and discrete by part (ii).

\item[(iv)] Finally, suppose $x_0 \in \textrm{int}(\Sigma_1)$. Since $x_0 \in \textrm{int}(\Sigma_0) \cap L_1$, we know $x_0$ is the limit of points $y_i \in V_0$ converging to $x_0$ from the right such that $f_1(y_i) = f_0(y_i)$. Hence $x_1=f_L(x_0)$. Thus $x_1$ is on the left edge of some extended rectangle $Q_k$. Hence the orbit of $x_0$ is infinite and discrete by part (iii). \qedhere
\end{itemize}

\end{proof}

\begin{definition} \label[definition]{def:equiv}
    The \textit{infinite 2-complex corresponding to $M$} is 
    $$\Sigma_2 : = \Sigma_1 / \sim$$
    where $\sim$ is an equivalence relation generated by
    $z \sim w$ if
    \begin{itemize}
        \item[(i)] there exists $x \in L_1 \cap R_1$ and $N \in \mathbb{Z}_{\geq 1}$ such that $z = f_{1,L}^N(x)$ and $w = f_{1,R}^N(x)$, or
        \item[(ii)] there exists $y \in T_1 \cap B_1$ and $N \in \mathbb{Z}_{\geq 1}$ such that $z = f_{1,T}^{-N}(y)$ and $w = f_{1,B}^{-N}(y)$.
    \end{itemize}
\end{definition}

The 2-complex $\Sigma_2$ is defined so that we will be able to extend $f_1$ to a homeomorphism on $\Sigma_2$. However, $\Sigma_2$ is not, in general, a surface. To address this, we will make a few minor modifications to $\Sigma_2$ to form a surface $\Sigma$, and then show that $f_1$ extends to a homeomorphism on $\Sigma$. In subsequent sections, we will show that $\Sigma$ can be further modified to ensure that it is of infinite type (\Cref{infinite-type}) and that the homeomorphism we have built is indeed end-periodic (\Cref{sec:end-periodic map}). First, we focus on proving several technical lemmas which will aid in building $\Sigma$ and extending $f_1$. 

We will use the following characterization of the infinite equivalence classes for the equivalence relation introduced in \Cref{def:equiv} both in the proof of \Cref{lemma:links} and when we show that the map we are building is end-periodic in \Cref{sec:end-periodic map}.

\begin{lemma}\label[lemma]{lemma:infinite equivalence}
    The equivalence class $[f_{1,L}(z)]_{\sim}$ is infinite if and only if $[z]_{\sim}$ is. Similarly for $f_{1,R}$, $f_{1,T}^{-1}$, and $f_{1,B}^{-1}$. 
\end{lemma}

\begin{proof}
    
     If $z \sim w$ via part $(i)$ of the definition of $\sim$ for $z \in L_1$ and $w \in R_1$, then $f_{1,L}(z) \sim f_{1,R}(w)$. Hence if $[z]_\sim$ is infinite, so is $[f_{1,L}(z)]_\sim$. 

     On the other hand, if $[f_{1,L}(z)]_\sim$ is infinite, there exists a distinct sequence  
     $$w_0 \sim w_1 \sim w_2 \sim \dots$$
     where $w_0 = f_{1,L}(z)$, each $w_i$ is distinct, and each relation above follows from part (i). Thus for each $k \in \mathbb{Z}_{\geq 0}$ there exists $N_k, J_k \in \mathbb{Z}_{\geq 1}$, $x_k \in L_1$, and $y_k \in R_1$ such that 
     \begin{itemize}
         \item[] $w_{2k}=f_{1,L}^{N_k}(x_k)$,
         \item[] $w_{2k+1}=f_{1,R}^{N_k}(x_k)$,
         \item[] $w_{2k+1}= f_{1,R}^{J_k}(y_k)$, and
         \item[] $w_{2k+2} = f_{1,L}^{J_k}(y_k)$.
     \end{itemize}

     If only finitely many of the $N_k$ or $J_k$ are greater than or equal to $2$, then $f_{1,R}(x_k)=f_{1,R}(y_k)$ for all large enough $k$. By \Cref{lemma: finite non-injective} there are only finitely many points in $R_1$  on which $f_{1,R}$ is non-injective. Thus $x_k = y_k$ for large enough $k$, however this contradicts that the $w_i$ are distinct.

     Therefore, infinitely many of the $N_k$ or $J_k$ are greater than or equal to $2$. Thus we can find a similar sequence of relations between distinct elements in the equivalence class $[z]_\sim$. \end{proof}

Next we further describe the equivalence classes of our equivalence relation $\sim$ on $\Sigma_1$. After this we will be ready to show that, after mild modifications, we obtain an actual surface from the $2$-complex, $\Sigma_2 = \Sigma_1/\sim$. 

\begin{lemma} \label[lemma]{lemma:links}
    Every equivalence class of $\sim$ is discrete. There are finitely many infinite equivalence classes, 
    and the links of these points in $\Sigma_2$ are homeomorphic to either $\mathbb R$ or a countably infinite disjoint union of circles. The finite equivalence classes are either singletons or contain exactly two points, and the links of these points in $\Sigma_2$ are homeomorphic to $\mathbb{S}^1$ or $[0,1]$, respectively. 
\end{lemma}

\begin{proof}
    If $z \in \textrm{int}(\Sigma_1)$, then $z$ is not in the image of $f_{1,L}$, $f_{1,R}$, $f_{1,B}^{-1}$, nor $f_{1,T}^{-1}$. Hence $[z]$ is a singleton and has link homeomorphic to $\mathbb{S}^1$. 

    Suppose $z \in \partial \Sigma_1$ and for all $N \in \mathbb{Z}_{\geq0}$, $[z]$ does not contain the $N^{th}$ image of a corner of a vertical or horizontal strip under any edge map. Suppose $z \in L_1$ and 
    $$z = z_0 \sim z_1 \sim z_2$$
    such that each relation above comes from (i) or (ii) in the definition of $\sim$. 
    Since $z_0 \in L_1$ and is not a corner, we have $z_0 \notin T_1 \cup B_1$. Thus $z_0 \sim z_1$ via part (i) of the definition of $\sim$. Hence there is some $x \in L_1 \cap R_1$ and some $N \in \mathbb{Z}_{\geq 1}$ such that $z_0=f_{1,L}^N(x)$ and $z_1=f_{1,R}^N(x)$. 
    Since $z_1$ is also not a corner, the relation $z_1 \sim z_2$ also follows from part (i) of the definition of $\sim$. Hence there is some $x' \in L_1 \cap R_1$ and some $N' \in \mathbb{Z}_{\geq 1}$ such that $z_1=f_{1,R}^{N'}(x')$ and $z_2=f_{1,L}^{N'}(x')$. 
    Without loss of generality, suppose $N' \geq N$. Then
    \begin{align*}
        z_1 & = f_{1,R}^N(x)  = f_{1,R}^{N'}(x') \\
        & = f_{1,R}^N(f_{1,R}^{(N'-N)}(x'))
    \end{align*}
    Since $x$ is not a corner point, $f^N_{1,R}$ is injective at $x$ by \Cref{lemma: finite non-injective}. Hence $x=f_{1,R}^{N'-N}(x')$. Since $x \in L_1 \cap R_1$ and the image of $f_{1,R}$ lies in $R_1-(L_1 - $ corner points$)$, we must have $N'=N$, and hence $x'=x$. Thus $z_2=z_0$. Therefore, $[z]=\{z,z_1\}$ or $[z]=\{z\}$ and, by construction, has link $\mathbb S^1$ or $[0,1]$ , respectively. 

    Since there are only finitely many corners of the horizontal and vertical strips and by \Cref{lemma: infinite orbit} the orbits of the edge maps are discrete, there is a discrete set of points in $\Sigma_1$ whose equivalence class is infinite and each infinite equivalence class is a discrete set.
    
    Observe that the left edges in $\Sigma_1$ minus the boundaries of the infinite strips, form a bounded fundamental domain for $f_{1,L}$. By discreteness, this domain contains only finitely many points with infinite equivalence class. It follows from \Cref{lemma:infinite equivalence} that there are therefore only finitely many infinite equivalence classes in total. Moreover, discreteness implies that the link in $\Sigma_2$ of each infinite equivalence class is homeomorphic either to $\mathbb{R}$ or to a countable disjoint union of circles. \end{proof}

    When $M$ is a primitive matrix, to build the \textit{surface corresponding to $M$}, denoted $\Sigma$: replace each point $x \in \Sigma_2$ whose link is homeomorphic to $\mathbb{R}$ or a countably infinite disjoint union of circles with link$(x)$ to form $\Sigma_2'$. Next, double $\Sigma_2'$ along its boundary to form $\Sigma$.  When $M$ is irreducible but not primitive, we replace each point whose link is $\mathbb{R}$  or a countably infinite disjoint union of circles with its link, and then perform a slightly modified doubling operation (which ensures connectedness) to build $\Sigma$. See Section \ref{weak perron} for a discussion of the irreducible but not primitive case. 

\begin{proposition}
    $\Sigma$ is a surface. 
\end{proposition}

\begin{proof} By \Cref{lemma:links} there are only finitely many points at which $\Sigma_2$ fails to be a surface with boundary because they have links homeomorphic to either $\mathbb R$ or a countably infinite disjoint union of circles. The remaining points of $\Sigma_2$ are either interior points with neighborhoods homeomorphic to a disk (if their link is $\mathbb S^1$) or they are boundary points with neighborhoods homeomorphic to a half disk (if their link is $[0,1]$). When building $\Sigma$ from $\Sigma_2$, the finitely many points of $\Sigma_2$ with links homeomorphic to $\mathbb R$ are first replaced by their link. This produces a surface with boundary since all points which previously had link $\mathbb R$ are now boundary points with link $[0,1]$. We then double along the boundary to obtain the closed surface $\Sigma$. \end{proof}

\begin{definition} \label[definition]{def:final-foliations}
Note that the vertical and horizontal foliations $\mathcal{V}_0, \mathcal H_0$ on $\Sigma_0$ admit natural extensions to a pair of singular, transverse, measured foliations, $\mathcal{V}_2, \mathcal{H}_2$, on $\Sigma_2$, where the extensions of $\mathcal V_0$ and $\mathcal H_0$ to the infinite strips are induced by the identifications of $\sim$. See \Cref{fig:integer foliations} for an illustration. Note that the set of leaves of $\mathcal{V}_2$ and $\mathcal{H}_2$ that intersect points with infinite equivalence class has measure zero. We therefore define the pair of singular, transverse, measured foliations $\mathcal{V}$ and $\mathcal{H}$ on $\Sigma$ to be the foliations induced by $\mathcal{V}_2$ and $\mathcal{H}_2$, after removing this measure zero set of leaves.\end{definition}

\begin{figure}[htbp]
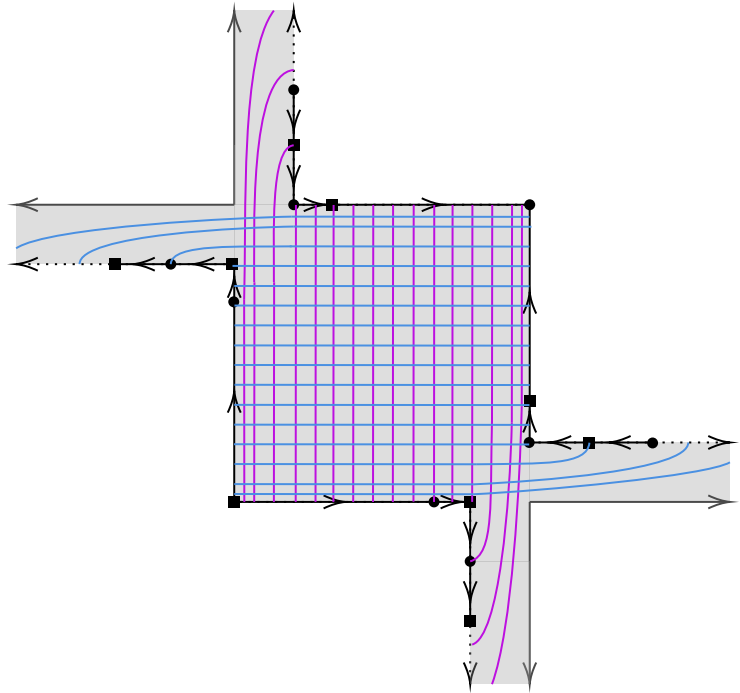

\begin{center}

\tikzset{every picture/.style={line width=0.75pt}} 


\end{center}
\caption{The horizontal and vertical foliations $\mathcal{V}_2$ and $\mathcal{H}_2$ on $\Sigma_2$ from \Cref{integers}.}
    \label{fig:integer foliations}
\end{figure}



\section{Ensuring $\Sigma$ is Infinite Type}\label{infinite-type}

It is possible that the surface $\Sigma$ constructed in the previous sections has finite-type and $f:\Sigma \rightarrow \Sigma$ is a pseudo-Anosov homeomorphism with stretch factor equal to $\lambda$. It would be interesting to understand precisely when this happens. However, we should note that it is only possible in very restrictive cases. For example, if the eigenvalue $\lambda$ is not a bi-Perron algebraic unit, then $\Sigma$ cannot be finite type, since every pseudo-Anosov stretch factor is bi-Perron. 

In order to ensure that every eigenvalue $\lambda$ can arise as an end-periodic stretch factor, we show that a mild alteration of our construction always yields an infinite-type surface. To do this, we first show that there is at least one corner which is periodic under the edge maps of $f_0$. Then, we insert handles along the invariant set of external lines appearing from the unidentified sides of the two infinite strips glued along the two neighborhoods of the corner periodic point.

\begin{proposition}\label[proposition]{corner periodic point}
    For any irreducible matrix $M$, there exist bijections $\{\sigma_i\}_{i=1}^n$ and $\{\tau_i\}_{i=1}^n$ such that the top left corner of $Q_1$ is a periodic point of $f_L$. 
\end{proposition}

\begin{proof}
    Since $M$ is irreducible, there exists an embedded cycle $C$ in the digraph associated to $M$ which contains the vertex $v_1$. Denote the (distinct) indices of this cycle in order by 
    $$1=\Sigma_1, \Sigma_2, s_3, \dots, s_k.$$
    By definition of the associated digraph, we must have $m_{s_{j+1},s_{j}} \neq 0$ for each $j$ (taking addition modulo $k$ in the indices of the $s_i$.). Thus there is a vertical strip $V^{(s_j)}_{s_{j+1},1}$ in $V$ and a horizontal strip $H^{(s_{j+1})}_{s_j,1}$ in $H$ for every $j$.

    Now, choose a set of bijections $\{\tau_i\}_{i=1}^n$ such that
    $$\tau_{s_j}(V^{(s_j)}_{\text{min}}) = V^{(s_j)}_{s_{j+1},1},$$
    for every $1 \leq j \leq k$, again taking addition modulo $k$.
    
    Further, choose a set of bijections $\{\sigma_i\}_{i=1}^n$ such that
    $$\sigma_{s_{j+1}}(H^{(s_{j+1})}_{s_{j},1}) = H^{(s_{j+1})}_{\text{min}},$$
    for every $0 \leq j \leq k-1$, again taking addition modulo $k$ in the indices of the $s_i$. Since the $s_j$ are distinct, there exist sets $\{\tau_i\}_{i=1}^n$ and $\{\sigma_i\}_{i=1}^n$ of bijections for which this is true. \\

    Let $x_j$ be the top left corner of $Q_{s_j}$ for each $1 \leq j \leq k$. Now $f_L(x_j)$ is the top endpoint of $f_L(\text{Left edge}(V^{(s_j)}_\text{min}))$. Since $\tau_{s_j}(V^{(s_j)}_{\text{min}}) = V^{(s_j)}_{s_{j+1},1}$, we have
    \begin{align*}
        f_L(\text{Left edge}(V^{(s_j)}_\text{min})) & = \text{Left edge}(\sigma_{s_{j+1}}(H^{(s_{j+1})}_{s_j,1}))\\
        & = \text{Left edge}(H^{(s_{j+1})}_{\text{min}})).
    \end{align*}
    Thus $f_L(x_j)$ is the top left corner of $Q_{s_{j+1}}$, which is $x_{j+1}$. Hence $f^{k}_L(x_1)=x_{1}$.  
\end{proof}

\begin{example}[Running Example]\label{example: run ex new order} Consider the cycle $v_1 \rightarrow v_2 \rightarrow v_4 \rightarrow v_4 \rightarrow v_1$ in the digraph corresponding to $M$. The proof of the above proposition implies if $\{\tau_i\}_{i=1}^4$ and $\{\sigma_i\}_{i=1}^4$ satisfy:
\begin{align*}
    \tau_1(V_{2,1}^{(1)}) & = V_{2,1}^{(1)}, & \tau_2(V_{4,1}^{(2)}) & = V_{4,1}^{(2)}, & \tau_3(V_{1,1}^{(3)}) & = V_{1,1}^{(3)}, & \tau_4(V_{2,1}^{(4)}) & = V_{3,1}^{(4)}
\end{align*}
and 
\begin{align*}
    \sigma_1(H_{3,1}^{(1)})& = H_{3,1}^{(1)}, & \sigma_2(H_{1,1}^{(2)})& = H_{1,1}^{(2)}, & \sigma_3(H_{4,1}^{(3)})& = H_{4,1}^{(3)}, & \sigma_4(H_{2,1}^{(4)})& = H_{1,1}^{(4)}
\end{align*}
then $x_1=$ the top left corner of $Q_1$ is periodic under $f_L$. \\

\begin{figure}[htbp]
\begin{center}

\tikzset{every picture/.style={line width=0.75pt}} 

\begin{tikzpicture}[x=0.75pt,y=0.75pt,yscale=-1,xscale=1]

\draw  [fill={rgb, 255:red, 207; green, 132; blue, 221 }  ,fill opacity=0.58 ][dash pattern={on 0.84pt off 2.51pt}] (88.29,64) -- (168.95,64) -- (168.95,107.39) -- (88.29,107.39) -- cycle ;
\draw  [fill={rgb, 255:red, 216; green, 79; blue, 96 }  ,fill opacity=0.57 ][dash pattern={on 0.84pt off 2.51pt}] (187.58,42.31) -- (270.07,42.31) -- (270.07,107.39) -- (187.58,107.39) -- cycle ;
\draw  [fill={rgb, 255:red, 74; green, 144; blue, 226 }  ,fill opacity=0.58 ][dash pattern={on 0.84pt off 2.51pt}] (289.38,52.52) -- (350.73,52.52) -- (350.73,104.84) -- (289.38,104.84) -- cycle ;
\draw  [fill={rgb, 255:red, 248; green, 231; blue, 28 }  ,fill opacity=0.39 ][dash pattern={on 0.84pt off 2.51pt}] (371.84,27) -- (442.53,27) -- (442.53,108.66) -- (371.84,108.66) -- cycle ;
\draw  [draw opacity=0][fill={rgb, 255:red, 197; green, 69; blue, 223 }  ,fill opacity=0.65 ][dash pattern={on 0.84pt off 2.51pt}] (88.29,64) -- (136.01,64) -- (136.01,107.39) -- (88.29,107.39) -- cycle ;
\draw  [draw opacity=0][fill={rgb, 255:red, 230; green, 60; blue, 80 }  ,fill opacity=0.67 ][dash pattern={on 0.84pt off 2.51pt}] (187.58,42.31) -- (228.03,42.31) -- (228.03,107.39) -- (187.58,107.39) -- cycle ;
\draw  [fill={rgb, 255:red, 74; green, 144; blue, 226 }  ,fill opacity=0.57 ][dash pattern={on 0.84pt off 2.51pt}] (88.06,294.96) -- (168.95,294.96) -- (168.95,338.34) -- (88.06,338.34) -- cycle ;
\draw  [draw opacity=0][fill={rgb, 255:red, 248; green, 231; blue, 28 }  ,fill opacity=0.36 ][dash pattern={on 0.84pt off 2.51pt}] (187.58,298.79) -- (271.2,298.79) -- (271.2,338.34) -- (187.58,338.34) -- cycle ;
\draw  [draw opacity=0][fill={rgb, 255:red, 189; green, 16; blue, 224 }  ,fill opacity=0.62 ][dash pattern={on 0.84pt off 2.51pt}] (187.58,271.99) -- (270.07,271.99) -- (270.07,298.79) -- (187.58,298.79) -- cycle ;
\draw  [fill={rgb, 255:red, 248; green, 231; blue, 28 }  ,fill opacity=0.79 ][dash pattern={on 0.84pt off 2.51pt}] (292.79,282.2) -- (351.87,282.2) -- (351.87,339.62) -- (292.79,339.62) -- cycle ;
\draw    (56.48,97.44) .. controls (13.52,183.77) and (13.31,326.95) .. (56.95,423.91) ;
\draw [shift={(57.61,425.37)}, rotate = 245.44] [color={rgb, 255:red, 0; green, 0; blue, 0 }  ][line width=0.75]    (10.93,-3.29) .. controls (6.95,-1.4) and (3.31,-0.3) .. (0,0) .. controls (3.31,0.3) and (6.95,1.4) .. (10.93,3.29)   ;
\draw  [draw opacity=0][fill={rgb, 255:red, 248; green, 231; blue, 28 }  ,fill opacity=0.81 ][dash pattern={on 0.84pt off 2.51pt}] (371.84,27) -- (397.31,27) -- (397.31,108.66) -- (371.84,108.66) -- cycle ;
\draw  [dash pattern={on 0.84pt off 2.51pt}]  (397.31,27) -- (397.31,108.66) ;
\draw  [dash pattern={on 0.84pt off 2.51pt}] (187.58,271.99) -- (270.07,271.99) -- (270.07,338.34) -- (187.58,338.34) -- cycle ;
\draw  [dash pattern={on 0.84pt off 2.51pt}]  (270.07,298.79) -- (187.58,298.79) ;
\draw  [fill={rgb, 255:red, 207; green, 132; blue, 221 }  ,fill opacity=0.58 ][dash pattern={on 0.84pt off 2.51pt}] (88.29,176.29) -- (168.95,176.29) -- (168.95,219.68) -- (88.29,219.68) -- cycle ;
\draw  [fill={rgb, 255:red, 216; green, 79; blue, 96 }  ,fill opacity=0.57 ][dash pattern={on 0.84pt off 2.51pt}] (187.58,154.6) -- (270.07,154.6) -- (270.07,219.68) -- (187.58,219.68) -- cycle ;
\draw  [fill={rgb, 255:red, 74; green, 144; blue, 226 }  ,fill opacity=0.58 ][dash pattern={on 0.84pt off 2.51pt}] (289.38,164.81) -- (350.73,164.81) -- (350.73,217.12) -- (289.38,217.12) -- cycle ;
\draw  [fill={rgb, 255:red, 248; green, 231; blue, 28 }  ,fill opacity=0.39 ][dash pattern={on 0.84pt off 2.51pt}] (371.84,139.29) -- (442.53,139.29) -- (442.53,220.95) -- (371.84,220.95) -- cycle ;
\draw  [draw opacity=0][fill={rgb, 255:red, 197; green, 69; blue, 223 }  ,fill opacity=0.65 ][dash pattern={on 0.84pt off 2.51pt}] (88.29,176.29) -- (136.01,176.29) -- (136.01,219.68) -- (88.29,219.68) -- cycle ;
\draw  [draw opacity=0][fill={rgb, 255:red, 230; green, 60; blue, 80 }  ,fill opacity=0.67 ][dash pattern={on 0.84pt off 2.51pt}] (187.58,154.6) -- (228.03,154.6) -- (228.03,219.68) -- (187.58,219.68) -- cycle ;
\draw  [draw opacity=0][fill={rgb, 255:red, 248; green, 231; blue, 28 }  ,fill opacity=0.81 ][dash pattern={on 0.84pt off 2.51pt}] (397.31,139.29) -- (442.43,139.29) -- (442.43,220.95) -- (397.31,220.95) -- cycle ;
\draw  [dash pattern={on 0.84pt off 2.51pt}]  (397.31,138.01) -- (397.31,219.68) ;
\draw  [fill={rgb, 255:red, 74; green, 144; blue, 226 }  ,fill opacity=0.57 ][dash pattern={on 0.84pt off 2.51pt}] (88.06,416.18) -- (168.95,416.18) -- (168.95,459.56) -- (88.06,459.56) -- cycle ;
\draw  [draw opacity=0][fill={rgb, 255:red, 248; green, 231; blue, 28 }  ,fill opacity=0.4 ][dash pattern={on 0.84pt off 2.51pt}] (187.58,420.01) -- (271.2,420.01) -- (271.2,459.56) -- (187.58,459.56) -- cycle ;
\draw  [draw opacity=0][fill={rgb, 255:red, 189; green, 16; blue, 224 }  ,fill opacity=0.6 ][dash pattern={on 0.84pt off 2.51pt}] (187.58,393.21) -- (270.07,393.21) -- (270.07,420.01) -- (187.58,420.01) -- cycle ;
\draw  [fill={rgb, 255:red, 248; green, 231; blue, 28 }  ,fill opacity=0.8 ][dash pattern={on 0.84pt off 2.51pt}] (292.79,403.42) -- (351.87,403.42) -- (351.87,460.84) -- (292.79,460.84) -- cycle ;
\draw  [dash pattern={on 0.84pt off 2.51pt}] (187.58,393.21) -- (270.07,393.21) -- (270.07,459.56) -- (187.58,459.56) -- cycle ;
\draw  [dash pattern={on 0.84pt off 2.51pt}]  (270.07,420.01) -- (187.58,420.01) ;
\draw  [draw opacity=0][fill={rgb, 255:red, 189; green, 16; blue, 224 }  ,fill opacity=0.38 ][dash pattern={on 0.84pt off 2.51pt}] (371.84,259.23) -- (442.53,259.23) -- (442.53,286.79) -- (371.84,286.79) -- cycle ;
\draw  [draw opacity=0][fill={rgb, 255:red, 208; green, 2; blue, 27 }  ,fill opacity=0.7 ][dash pattern={on 0.84pt off 2.51pt}] (371.84,286.79) -- (442.53,286.79) -- (442.53,308.49) -- (371.84,308.49) -- cycle ;
\draw  [draw opacity=0][fill={rgb, 255:red, 208; green, 2; blue, 27 }  ,fill opacity=0.42 ][dash pattern={on 0.84pt off 2.51pt}] (371.84,308.49) -- (442.53,308.49) -- (442.53,340.32) -- (371.84,340.32) -- cycle ;
\draw  [dash pattern={on 0.84pt off 2.51pt}] (371.84,259.23) -- (442.53,259.23) -- (442.53,340.32) -- (371.84,340.32) -- cycle ;
\draw  [dash pattern={on 0.84pt off 2.51pt}]  (442.53,286.79) -- (371.84,286.79) ;
\draw  [dash pattern={on 0.84pt off 2.51pt}]  (442.53,308.49) -- (371.84,308.49) ;
\draw  [draw opacity=0][fill={rgb, 255:red, 189; green, 16; blue, 224 }  ,fill opacity=0.38 ][dash pattern={on 0.84pt off 2.51pt}] (372.97,407.22) -- (443.67,407.22) -- (443.67,429.71) -- (372.97,429.71) -- cycle ;
\draw  [draw opacity=0][fill={rgb, 255:red, 208; green, 2; blue, 27 }  ,fill opacity=0.7 ][dash pattern={on 0.84pt off 2.51pt}] (372.97,380.45) -- (443.67,380.45) -- (443.67,408.01) -- (372.97,408.01) -- cycle ;
\draw  [draw opacity=0][fill={rgb, 255:red, 208; green, 2; blue, 27 }  ,fill opacity=0.42 ][dash pattern={on 0.84pt off 2.51pt}] (372.97,429.71) -- (443.67,429.71) -- (443.67,461.54) -- (372.97,461.54) -- cycle ;
\draw  [dash pattern={on 0.84pt off 2.51pt}] (372.97,380.45) -- (443.67,380.45) -- (443.67,461.54) -- (372.97,461.54) -- cycle ;
\draw  [dash pattern={on 0.84pt off 2.51pt}]  (443.67,408.01) -- (372.97,408.01) ;
\draw  [dash pattern={on 0.84pt off 2.51pt}]  (443.67,429.71) -- (372.97,429.71) ;
\draw    (78.29,124.92) .. controls (70.58,138.9) and (72.17,144.79) .. (77.59,157.93) ;
\draw [shift={(78.29,159.61)}, rotate = 247.4] [color={rgb, 255:red, 0; green, 0; blue, 0 }  ][line width=0.75]    (10.93,-3.29) .. controls (6.95,-1.4) and (3.31,-0.3) .. (0,0) .. controls (3.31,0.3) and (6.95,1.4) .. (10.93,3.29)   ;
\draw    (78.18,243.86) .. controls (70.48,257.85) and (72.06,263.74) .. (77.48,276.87) ;
\draw [shift={(78.18,278.56)}, rotate = 247.4] [color={rgb, 255:red, 0; green, 0; blue, 0 }  ][line width=0.75]    (10.93,-3.29) .. controls (6.95,-1.4) and (3.31,-0.3) .. (0,0) .. controls (3.31,0.3) and (6.95,1.4) .. (10.93,3.29)   ;
\draw    (80.13,358.98) .. controls (72.43,372.97) and (74.01,378.86) .. (79.43,391.99) ;
\draw [shift={(80.13,393.67)}, rotate = 247.4] [color={rgb, 255:red, 0; green, 0; blue, 0 }  ][line width=0.75]    (10.93,-3.29) .. controls (6.95,-1.4) and (3.31,-0.3) .. (0,0) .. controls (3.31,0.3) and (6.95,1.4) .. (10.93,3.29)   ;
\draw  [dash pattern={on 0.84pt off 2.51pt}]  (228.03,154.6) -- (228.03,219.68) ;
\draw  [dash pattern={on 0.84pt off 2.51pt}]  (228.03,42.31) -- (228.03,107.39) ;
\draw  [dash pattern={on 0.84pt off 2.51pt}]  (136.01,64) -- (136.01,107.39) ;
\draw  [dash pattern={on 0.84pt off 2.51pt}]  (136.01,176.29) -- (136.01,219.68) ;

\draw (3.2,245.74) node [anchor=north west][inner sep=0.75pt]    {$f_{0}$};
\draw (100.92,76.54) node [anchor=north west][inner sep=0.75pt]  [font=\scriptsize]  {$V_{2,1}^{( 1)}$};
\draw (140.68,76.54) node [anchor=north west][inner sep=0.75pt]  [font=\scriptsize]  {$V_{4,1}^{( 1)}$};
\draw (373.67,59.96) node [anchor=north west][inner sep=0.75pt]  [font=\scriptsize]  {$V_{2,1}^{( 4)}$};
\draw (218.01,313.88) node [anchor=north west][inner sep=0.75pt]  [font=\scriptsize]  {$H_{4,1}^{( 2)}$};
\draw (196.35,62.51) node [anchor=north west][inner sep=0.75pt]  [font=\scriptsize]  {$V_{4,1}^{( 2)}$};
\draw (237.25,62.51) node [anchor=north west][inner sep=0.75pt]  [font=\scriptsize]  {$V_{4,2}^{( 2)}$};
\draw (311.1,68.89) node [anchor=north west][inner sep=0.75pt]  [font=\scriptsize]  {$V_{1,1}^{( 3)}$};
\draw (219.14,276.88) node [anchor=north west][inner sep=0.75pt]  [font=\scriptsize]  {$H_{1,1}^{( 2)}$};
\draw (119.17,307.5) node [anchor=north west][inner sep=0.75pt]  [font=\scriptsize]  {$H_{3,1}^{( 1)}$};
\draw (311.17,299.84) node [anchor=north west][inner sep=0.75pt]  [font=\scriptsize]  {$H_{4,1}^{( 3)}$};
\draw (412.22,61.23) node [anchor=north west][inner sep=0.75pt]  [font=\scriptsize]  {$V_{3,1}^{( 4)}$};
\draw (100.92,188.83) node [anchor=north west][inner sep=0.75pt]  [font=\scriptsize]  {$V_{2,1}^{( 1)}$};
\draw (140.68,188.83) node [anchor=north west][inner sep=0.75pt]  [font=\scriptsize]  {$V_{4,1}^{( 1)}$};
\draw (373.67,172.24) node [anchor=north west][inner sep=0.75pt]  [font=\scriptsize]  {$V_{2,1}^{( 4)}$};
\draw (196.35,174.8) node [anchor=north west][inner sep=0.75pt]  [font=\scriptsize]  {$V_{4,1}^{( 2)}$};
\draw (237.25,174.8) node [anchor=north west][inner sep=0.75pt]  [font=\scriptsize]  {$V_{4,2}^{( 2)}$};
\draw (311.1,181.18) node [anchor=north west][inner sep=0.75pt]  [font=\scriptsize]  {$V_{1,1}^{( 3)}$};
\draw (409.94,173.52) node [anchor=north west][inner sep=0.75pt]  [font=\scriptsize]  {$V_{3,1}^{( 4)}$};
\draw (218.01,435.1) node [anchor=north west][inner sep=0.75pt]  [font=\scriptsize]  {$H_{4,1}^{( 2)}$};
\draw (219.14,398.1) node [anchor=north west][inner sep=0.75pt]  [font=\scriptsize]  {$H_{1,1}^{( 2)}$};
\draw (119.17,428.72) node [anchor=north west][inner sep=0.75pt]  [font=\scriptsize]  {$H_{3,1}^{( 1)}$};
\draw (311.17,421.06) node [anchor=north west][inner sep=0.75pt]  [font=\scriptsize]  {$H_{4,1}^{( 3)}$};
\draw (397.52,264.12) node [anchor=north west][inner sep=0.75pt]  [font=\scriptsize]  {$H_{1,1}^{( 4)}$};
\draw (397.52,289.64) node [anchor=north west][inner sep=0.75pt]  [font=\scriptsize]  {$H_{2,1}^{( 4)}$};
\draw (398.16,317.07) node [anchor=north west][inner sep=0.75pt]  [font=\scriptsize]  {$H_{2,2}^{( 4)}$};
\draw (398.65,386.61) node [anchor=north west][inner sep=0.75pt]  [font=\scriptsize]  {$H_{1,1}^{( 4)}$};
\draw (398.65,410.86) node [anchor=north west][inner sep=0.75pt]  [font=\scriptsize]  {$H_{2,1}^{( 4)}$};
\draw (400.36,442.12) node [anchor=north west][inner sep=0.75pt]  [font=\scriptsize]  {$H_{2,2}^{( 4)}$};
\draw (59.34,128.88) node [anchor=north west][inner sep=0.75pt]    {$\tau $};
\draw (55.85,247.83) node [anchor=north west][inner sep=0.75pt]    {$w$};
\draw (62.32,362.95) node [anchor=north west][inner sep=0.75pt]    {$\sigma $};

\end{tikzpicture}
\end{center}
\caption{\textbf{Running Example:} The piece map $f_0:V \rightarrow H$ corresponding to $M$ from \Cref{example: run ex} with maps $\sigma$ and $\tau$ chosen as in \Cref{corner periodic point} to ensure the top left corner of $Q_1$, denoted $x_1$, is a periodic point of the edge maps. Indeed, $f_L^4(x_1) = x_1$. }
\label{fig:run ex new order}
\end{figure}
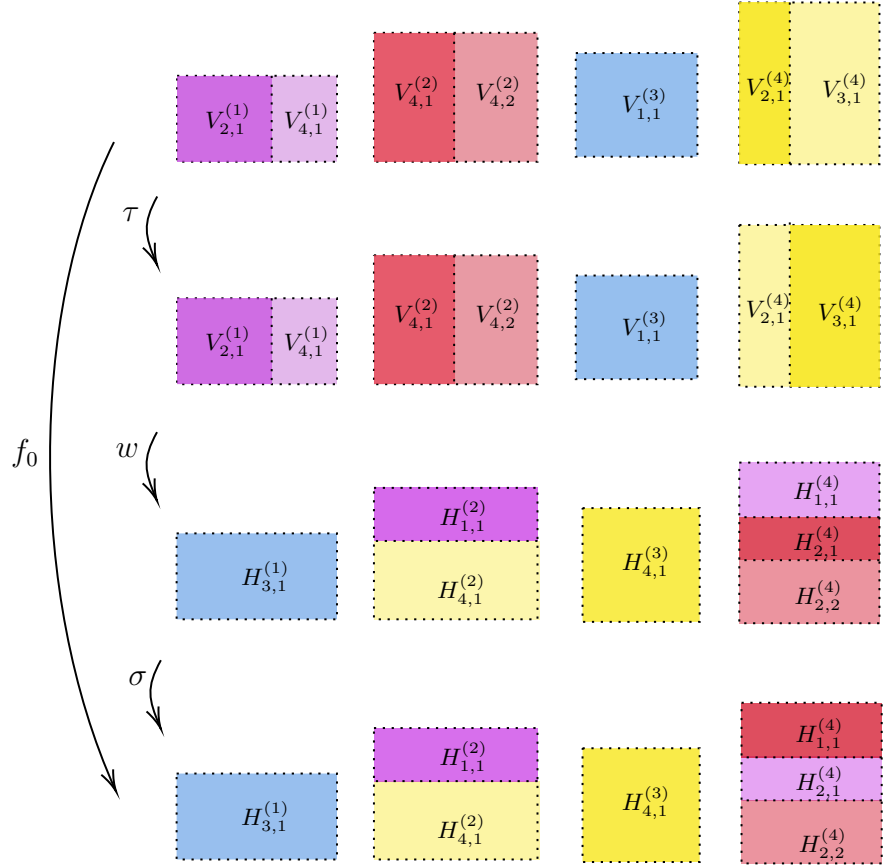

Observe that after changing $\tau$ and $\sigma$, we need to adjust the widths and heights of the strips so that $f_0$ still stretches horizontally by $\lambda$ and contracts vertically by $1/\lambda$. The heights and widths of the whole rectangles remain the same. The four digraphs change as well, as seen in Figure \ref{fig: digraph after re-ordering}.

\begin{figure}[htbp]
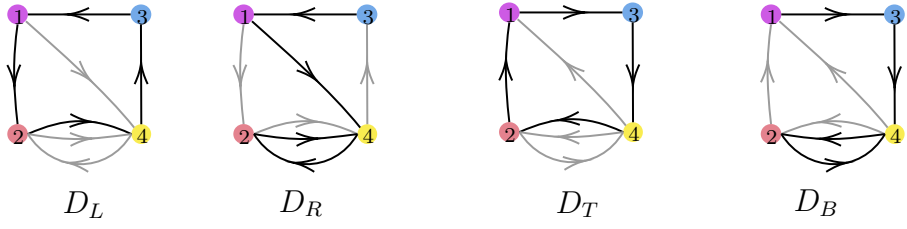

\begin{center}

\tikzset{every picture/.style={line width=0.75pt}} 


\end{center}
\caption{\textbf{Running Example:} The four edge digraphs corresponding to $f_0:V \rightarrow H$ in Example \ref{example: run ex new order}. The black directed edges are in the directed graphs and the light gray edges are the remaining edges in the digraphs corresponding to $M$ for $D_L$ and $D_R$ and corresponding to $M^T$ for $D_T$ and $D_B$.}
\label{fig: digraph after re-ordering}
\end{figure}

\end{example}

Now we can guarantee there is at least one corner periodic point, and hence at least one periodic boundary component of $\Sigma_2$ which results from the unidentified sides of the two adjacent infinite strips. Choose $x$ on this line and let $G = \{f^n(x)\}_{n \in \mathbb{Z}}$. Replace each point in $G$ with its link. As before, replace infinite equivalence classes with their link, double and identify along boundary components to form $\Sigma$. Note that none of the infinite equivalence classes occur on the specified periodic boundary components, so the links of points in $G$ do not interfere with this step. The links of points in $G$ will result in at least two ends accumulated by genus. 

\section{The End-Periodic Map}\label{sec:end-periodic map}

The following two lemmas will help us ensure we can find nesting neighborhoods of the ends of $\Sigma$ which are attracting or repelling.

\begin{lemma}\label[lemma]{lemma: max N}
    There exists $N$ such that $f_{1,L}^N(L_1)$ is contained in the boundary of the infinite strips. Analogous statements hold for $f_{1,R},f_{1,T}^{-1}$ and $f_{1,B}^{-1}$. 
\end{lemma}

\begin{proof}
    On any finite directed graph, there exists $k$ such that every directed path of length $k$ ends at a periodic vertex. Let $k$ be such a number for the directed graph $D_L$. Then for any $x \in L \subseteq \Sigma_0$, the point $f_L^k(x)$ is on a left edge which contains a periodic point $z$. Hence for any $x \in L_1$, we have $f_{1,L}^k(x)$ is on a left edge which contains an infinite strip. Let $p$ be the period of $z$. By definition of the switch region, $f_{1,L}^{k+p}(x)$ is either contained in the boundary of the switch region of $z$, or in the boundary of the infinite strip. By definition of the switch maps, $f_{1,L}^{k+2p}(x)$ must be in the boundary of the infinite strip. 
    
    By definition of $f_1$, if $y$ is in the boundary of an infinite strip, $f_{1,L}^i(y)$ is in the boundary of a (possibly different) infinite strip for every $i \in \mathbb{Z}_{\geq 0}$. Hence, letting $$N = k + 2 \max \{ p \ | \ p \text{ is the period of a cycle in }f_L\}$$ proves the lemma.
\end{proof}

\begin{lemma}\label[lemma]{lemma: alpha}
    There exists $\alpha \in \mathbb{R}_{>0}$ such that for every infinite strip $E_{*}(x)$, if $\beta > \alpha$ then $(0,\beta)_{*,x}$ and $(1,\beta)_{*,x} \in E_{*}(x)$ have $\sim$ equivalence class either infinite or contained in boundary of the set of infinite strips. Here $*=L,R,T,$ or $B$. 
\end{lemma}

\begin{proof}
    Let $N$ be the maximal $N$ from \Cref{lemma: max N} among the four edge maps. Choose $\alpha \in \mathbb{R}_{>0}$ large enough so that for every infinite strip $E_{*,x}$ and every $\beta > \alpha$, the points $(0,\beta)_{*,x}$ $(1,\beta)_{*,x}$ are contained in the $N^{th}$ image of whichever edge map of $f_1$ is defined on that point. 
    
    If $x$ is a corner periodic point, then there are some points in the boundary of the corresponding infinite strip on which there are no edge maps of $f_1$ defined. Their equivalence class is a singleton.

    Now suppose $z \in f_{1,L}^N(L_1)$ and has finite equivalence class. Then by \Cref{lemma:links}, either $[z]_{\sim}$ is a singleton or contains exactly two points. If $[z]_\sim$ contains a second point $w$, then there is some $x \in L_1 \cap R_1$ and some $K \in \mathbb{Z}_{\geq 1}$ such that $z=f_{1,L}^{K}(x)$ and $w=f_{1,R}^K(x)$. Since $x \in L_1 \cap R_1$ and the image of $f_{1,L}$ does not intersect $L_1 \cap R_1$, we must have $K \geq N$. Thus $w=f_{1,R}^K(x) \subseteq f_{1,R}^N(R_1)$. By choice of $N$, $w$ is in the boundary of the infinite strips. 
    
    A similar argument holds for elements in $f^N_{1,R}(R_1),f_{1,T}^{-N}(T_1)$, and $f_{1,B}^{-N}(B_1)$. 
\end{proof}

\begin{proposition} \label[proposition]{prop:end-periodic}
    The map $f_1:V_1 \rightarrow H_1$ extends to an end-periodic homeomorphism $f:\Sigma \rightarrow \Sigma$. 
\end{proposition}

\begin{proof} 
 Every point $x \in \Sigma$ is the limit of a sequence of points $\{x_i\}_{i=1}^\infty \subseteq V_1$. First, suppose $x$ is not a point in $\Sigma$ resulting from blowing up an infinite equivalence class in $\Sigma_1$. Then $ [\lim_{i \rightarrow \infty} f_1(x_i)]_\sim$ is also a finite equivalence class by \Cref{lemma:infinite equivalence}. Since the equivalence relation $\sim$ identifies the different images of points under the left and right edge maps, setting $f(x) := [\lim_{i \rightarrow \infty} f_1(x_i)]_\sim$ is well-defined and continuous.

 Now suppose $x \in \Sigma$ is a point in $\Sigma$ resulting from blowing up an infinite equivalence class in $\Sigma_2$. Let $[z]_\sim$ be the infinite equivalence class. Since $x \in \text{link}([z]_\sim)$, there is a corresponding sequence $\{x_i\}_{i=1}^\infty \subseteq V_1$ converging to a point $z \in [z]_\sim$. There is a point $y \in \text{link}([\lim_{i \rightarrow \infty} f_1(x_i)]_\sim)$ which corresponds to the sequence $\{f_1(x_i)\}$. By \Cref{lemma:infinite equivalence}, $[\lim_{i \rightarrow \infty} f_1(x_i)]_\sim$ is an infinite equivalence class, and hence $y \in \Sigma$. Let $f(x)=y$. Continuity of $f_1$ ensures that $f$ extended to these links as described is continuous. 
 
 Similarly, $f^{-1}$ is well-defined and continuous. Hence $f:\Sigma \rightarrow \Sigma$ is a homeomorphism. 

In order to show that $f: \Sigma \rightarrow \Sigma$ is end-periodic, it suffices to show that there exists $m > 0$ such that for each end $E$ of $\Sigma$ there is a nesting neighborhood $U_E$ of $E$ so that either $f^m(U_E) \subsetneq U_E$ and the sets $\{f^{nm}(U_E)\}_{n > 0}$ form a neighborhood basis of $E$ or $f^{-m}(U_E) \subsetneq U_E$ and the sets $\{f^{-nm}(U_E)\}_{n > 0}$ form a neighborhood basis of $E$.

Fix an end $E$ of $\Sigma$. There is some infinite strip $E_*(x) = [0,1] \times [0, \infty)$ in $\Sigma_1$ and some $\gamma >0$ such that the region in $\Sigma$ induced by $([0,1] \times (\gamma ,\infty))_{*,x}$ is contained in a neighborhood of $E$. Here, $*=L,R,T$ or $B$. 

Let $\alpha$ be as in \Cref{lemma: alpha} and assume $\alpha > \gamma$. By choice of $\alpha$, there exists a neighborhood $U_E$ of the region induced by $([0,1] \times (\alpha ,\infty))_{*,x}$ which is contained entirely in the infinite strips. If needed, increase $\alpha$ to remove any intersection with the switch regions.  If $*=L$ or $R$, then $U_E$ is contained in the left and right infinite strips. If $*=T$ or $B$, then $U_E$ is contained in the top and bottom infinite strips.

Let $m$ be the product of the periods of every cycle in $f_L$, $f_R$, $f_T^{-1}$ and $f_B^{-1}$. Now  the nesting behavior is by construction. More specifically, if $*=L$ or $R$, then $\{f^{nm}(U_E)\}_{n > 0}$ form a neighborhood basis for $E$. If $*=T$ or $B$, then $\{f^{-nm}(U_E)\}_{n > 0}$ form a neighborhood basis for $E$. \end{proof}

    We now call $f:\Sigma \rightarrow \Sigma$ the \textit{end-periodic homeomorphism corresponding to $M$}. Note that the choices we made when constructing $f$ and $\Sigma$ are: the bijections $\sigma$ and $\tau$ of the vertical and horizontal strips, the initial periodic points for each finite orbit of the edge maps of $f_0$, the exact paths $W$ for each initial periodic point, and the switch maps (each of which can be any homeomorphism that behaves as required on the boundary of the switch regions).



\begin{proposition} \label[proposition]{prop:stretch-factor}
    The end-periodic homeomorphism $f:\Sigma \rightarrow \Sigma$ corresponding to $M$ has Handel-Miller stretch factor $\lambda(f) = \lambda$, the spectral radius of $M$. 
\end{proposition}

\begin{proof}

Our goal is to show that the incidence matrix of the Markov decomposition determined by 
\(\Lambda^+ \cap \Lambda^-\) agrees with the incidence matrix of the Markov decomposition induced by the foliations \(\mathcal V\) and \(\mathcal H\), which is
\[
\begin{bmatrix}
M & 0 \\
0 & M
\end{bmatrix},
\]
where \(M\) is the irreducible matrix with which we began.

It is convenient to replace the foliations \(\mathcal V\) and \(\mathcal H\) with laminations. Abusing notation, we denote by \(\mathcal V\) and \(\mathcal H\) the laminations obtained by blowing up the singular leaves of the corresponding foliations. Since the foliations are transverse and have no prongs of order \(1\) or \(2\), these laminations determine the same Markov decomposition and incidence matrix as the original foliations.

Let \(J\) be an essential multiloop bounding a finite-type subsurface in which \(\mathcal V\) and \(\mathcal H\) intersect, and write \(J=j_+\sqcup j_-\), where \(j_+\) (resp. \(j_-\)) bounds a neighborhood of the regions arising from the infinite strips attached to the left and right (resp. top and bottom) sides of the rectangles. By \Cref{prop:end-periodic}, \(j_+\) and \(j_-\) are positive and negative junctures for \(f\). Let
\[
J^+=\bigcup_{k\in\mathbb Z} f^k(j_+),
\qquad
J^-=\bigcup_{k\in\mathbb Z} f^k(j_-).
\]

Fix a hyperbolic metric on \(\Sigma\) and let \(\mathcal J^\pm\) denote the union of the geodesic representatives of curves in \(J^\pm\). The Handel--Miller laminations are
\[
\Lambda^+=\overline{\mathcal J^-} - \mathcal J^-,
\qquad
\Lambda^-=\overline{\mathcal J^+} - \mathcal J^+ .
\]

After straightening, \(j_+\) is disjoint from \(\mathcal V\) and intersects the leaves of \(\mathcal H\) essentially and with consistent orientation. Consequently, each iterate \(f^k(j_+)\) lies between leaves of \(\mathcal V\) while intersecting leaves of \(\mathcal H\). Since the dynamics of $f$ on $\mathcal{V} \cap \mathcal{H}$ are determined by the irreducible matrix \(M\), the iterates \(f^k(j_+)\) accumulate on leaves of \(\mathcal V\). Hence, $\Lambda^- =\mathcal V$ as projective measured laminations. A symmetric argument applied to \(j_-\) shows that $\Lambda^+ =  \mathcal H $ as projective measured laminations. \end{proof}

\begin{remark}
    Note that one implication of the proof of \Cref{prop:stretch-factor} is that all the end-periodic homeomorphisms we construct have the additional property that their Handel-Miller laminations admit transverse measures of full support that are contracted/expanded by the stretch factor. We appreciate Chi Cheuk Tsang for suggesting we highlight this fact.
\end{remark}

 \section{Ensuring $\Sigma$ is Connected}\label{weak perron}

 In this section, we show that for any weak Perron number, we can build $\Sigma$ in a way that guarantees it is connected without impacting the conclusions of \Cref{prop:end-periodic} and \Cref{prop:stretch-factor}. 

\begin{proposition}
    If $\lambda$ is a Perron number, the surface $\Sigma$, as in \Cref{prop:end-periodic}, can be chosen to be connected.
\end{proposition}
\begin{proof}
    Let $M$ be a primitive non-negative integer $n \times n$ matrix with spectral radius equal to $\lambda$. Let $f_0:V \rightarrow H$ be the piece map corresponding to $M$ with $\{\sigma_i\}_{i=1}^n$ and $\{\tau_i\}_{i=1}^n$ coming from Proposition \ref{corner periodic point}. Let $\Sigma$ be the corresponding infinite-type surface and $f:S\rightarrow S$ the corresponding end-periodic homeomorphism. 
    
    Since $M$ is primitive, there is a power $k$ such that the first column of $M^k$ is positive. Hence $f_0^k(Q_1) \cap Q_i \neq \emptyset$ for each $1 \leq i \leq n$.  For each $i$, let $U_i$ be the induced subset of $\Sigma$ coming from $Q_i$ in $\Sigma_0$. We must have $f^k(U_1) \cap U_i \neq \emptyset$ for each $1 \leq i \leq n$. Since $U_1$ is connected, $f^k(U_1)$ is connected. Hence, each $U_i$ is in the same connected component, and hence $\Sigma$ is connected. 
\end{proof}

\begin{proposition}
    If $\mu$ is a weak Perron number, there is a choice of primitive matrix $M$ with spectral radius $\mu$ so that the surface $\Sigma$, as in \Cref{prop:end-periodic}, can be chosen to be connected.
\end{proposition}

\begin{proof}
 Let $\mu$ be a weak Perron number. Then $\mu^k=\lambda$ is a Perron number for some power $k$. Let $M$ be a primitive non-negative integer $m \times m$ matrix with spectral radius equal to $\lambda$. Let $n=km$ and define an $n \times n$ block permutation matrix:
 $$M' = \begin{bmatrix}
     0 & 0 & \hdots & 0 & M \\
     I_m & 0 & \hdots & 0 & 0 \\
     0 & I_m & \ddots & \vdots & \vdots \\
     \vdots & \ddots & \ddots & 0 & \vdots \\
     0 & \hdots & 0 & I_m & 0
 \end{bmatrix}$$
where $I_m$ is the $m \times m$ identity matrix. Observe that the spectral radius of $M'$ is exactly $\lambda^{1/k} = \mu$.

Let $f_0:V \rightarrow H$ be the piece map corresponding to $M'$ with $\{\sigma_i\}_{i=1}^n$ and $\{\tau_i\}_{i=1}^n$ coming from Proposition \ref{corner periodic point}. Proceed with the construction as described until we have $\Sigma_3$, which is equal to $\Sigma_2$ with each infinite equivalence class replaced by its link. Now double $\Sigma_3$ yielding $\Sigma_3$ and it's double $\Sigma_3'$. 

Note that if we glue $\Sigma_3$ and $\Sigma_3'$ along their boundaries by the identity map, we will build a surface with $k$ connected components, $C_1, \dots C_k$, and the resulting homeomorphism will map $C_i \mapsto C_{i+1}$ for each $i$, taking subscripts modulo $k$. 

Let $U_i$ be the induced subset of $S_3$ coming from $Q_i$ in $\Sigma_0$, and let $U_i'$ be its double. Since we chose our bijections $\{\sigma_i\}_{i=1}^n$ and $\{\tau_i\}_{i=1}^n$ to ensure the top left corner of $Q_1$ is periodic under $f_L$ (and hence also under $f_T^{-1}$, there is a boundary component of $U_1$ coming from the unidentified sides of the two infinite strips attached along a neighborhood of this corner periodic point. Let $A_1$ be the left infinite ray of this boundary component and $B_1$ be the top infinite ray of this boundary component, so $A_1$ and $B_1$ meet at a point $x_1$. Note that by the construction of $M'$, the orbit of $A_1$ is a set of infinite rays $A_2, \dots, A_k$ which are the equivalent boundary component for $U_{m+1},U_{2m+1}, \dots U_{n-k+1}$, and similarly for $B_1$. 

Now, glue each $A_i$ to $A_i'$ by the identity map and glue each $B_i$ to $B_{i+1}'$ by the identity map. Glue the remaining boundary components of $\Sigma_3$ and $\Sigma_3'$ by the identity map. If needed, insert genus along the $A_i$ and $B_i$ boundary components as described in \Cref{infinite-type}, and let $\Sigma$ be the resulting infinite-type surface. 

Observe that all of the $x_i$ are identified, and hence $U_1,U_{m+1},U_{2m+1}, \dots U_{n-k+1}$ and their doubles are all in the same connected component of $\Sigma$. 

Since $M$ is primitive, there is a power $\Sigma$ such that the first column of $M^s$ is positive. Hence there is a power $t$ such that the first through $m^{th}$ entries of the first column of $(M')^t$ are positive. Thus $f^t(U_1) \cap U_i \neq \emptyset$ for $1 \leq i \leq m$. Thus $U_1,\dots,U_m$ are all in the same connected component of $\Sigma$. Moreover, $f$ maps $U_j \cup \dots \cup U_{j+m}$ to $U_{j+m+1} \cup \dots \cup U_{j+2m}$, for each $j \in \{1, m+1, 2m+1, \dots n-k+1\}$. Hence $U_j, U_{j+1}, \dots, U_{j+m}$ are all in the same connected component for these values of $j$. Similarly for their doubles. As we have established that $U_1,U_{m+1},U_{2m+1}, \dots U_{n-k+1}$ and their doubles are all in the same connected component, $\Sigma$ is connected. \end{proof}


\section{Discussion and Further Questions}\label{discussion}

Although our work, combined with work of Cantwell--Conlon--Fenley, gives a complete characterization of end-periodic stretch factors, many interesting questions remain. We record a few of them here for the interested reader. In the finite-type setting, it remains an almost entirely open question to determine the minimal pseudo-Anosov stretch factor that can be realized on a surface with fixed complexity. In fact, it is not known for any punctured surface with genus $\geq 2$ or any closed surface with genus $\geq 3$ \cite{ChoHam}. Although the braid case has recently been resolved in all but 6 cases by work of \cite{TsangZeng}; see also \cite{SongKoLos, HamSong, Lanneau2011}. There is a natural adaption of this question to the infinite-type setting by replacing fixed complexity with fixed \emph{core characteristic}. Following \cite{Endperiodic2}, define the \emph{core characteristic} of $f$ to be \[\chi(f) = \max_{Y \subset \Sigma} \chi(Y),\] where the maximum is taken over all cores $Y \subset \Sigma$ for $f$.

\begin{question} For a fixed core characteristic $m$, what is the minimal stretch factor realized by any end-periodic homeomorphism $f$ with $\chi(f) = m$?
\end{question}

You could further restrict this question by also fixing a surface $L$. It is unclear to us whether the minimal stretch factor for a fixed core characteristic ``should" appear on a surface with more or less ends. Another restriction of this question would be to compute the minimal end-periodic stretch factor arising from the construction we used to prove \Cref{thm:main} for a fixed core characteristic. This naturally leads us to ask if we can compute bounds on the topological complexity of the end-periodic maps that we construct. The topological complexity of an end-periodic map $f$ is captured by its \emph{capacity} which is a pair consisting of the core characteristic $\chi(f)$ and the end-complexity $\xi(f)$. Morally, these quantities describe the complexity of the subsurface where $f$ behaves as a pseudo-Anosov and the amount that $f$ shifts by on the ends, respectively. See \cite[Section~2.3]{Endperiodic2} for detailed definitions and discussion.

We believe that the homeomorphisms obtained from this construction are atoroidal, but this is only clear in the examples built in \Cref{integers}. What is clear, in general, is that all of the examples we build are neither strongly irreducible nor irreducible. Indeed, the collection of links of infinite equivalence classes is invariant under $f$. These links give rise to a system of periodic arcs and reducing curves. 

\begin{question} Can this construction be modified to produce irreducible end-periodic homeomorphisms? More generally, can we characterize which weak Perron numbers can arise as the stretch factor of an (strongly) irreducible end-periodic homeomorphism? 
\end{question}

In a slightly different direction, we pose the following questions.

\begin{question} Can the infinite-type surface $\Sigma$ produced by our construction be realized as a limit of finite-type translation surfaces? Can the corresponding end-periodic maps we produced be realized as a limit of pseudo-Anosov maps on these finite-type translation surface?
\end{question}

Note that the example constructed in \Cref{integers} for $d=2$ can be approximated by a sequence of pseudo-Anosov maps on finite-type translation surfaces. Two consecutive elements in this sequence are illustrated in \Cref{fig:p5}. See \cite{bowman2013complete} for another example of such an approximation. 

In addition, it was pointed out to us by Chi Cheuk Tsang that for the examples in \Cref{integers} we are able to apply the ``despinning" construction of Landry--Minsky--Taylor to obtain a sequence of psuedo-Anosovs approximating the given end-periodic map. In particular, all the examples in \Cref{integers} have a compactified mapping torus with exactly two boundary components which are both genus $2$ surfaces with junctures consisting of a single nonseparating curves. Thus, there is a homeomorphism from the positive boundary to the negative boundary of the compactified mapping torus taking the positive juncture class to the negative juncture class and we are able to apply \cite[Proposition~3.9]{LandryMinskyTaylor2023}. It would be interesting to determine whether our general construction could be modified to guarantee that the end-periodic homeomorphism we build is realized on a ladder surface. 

\begin{figure}[htbp]
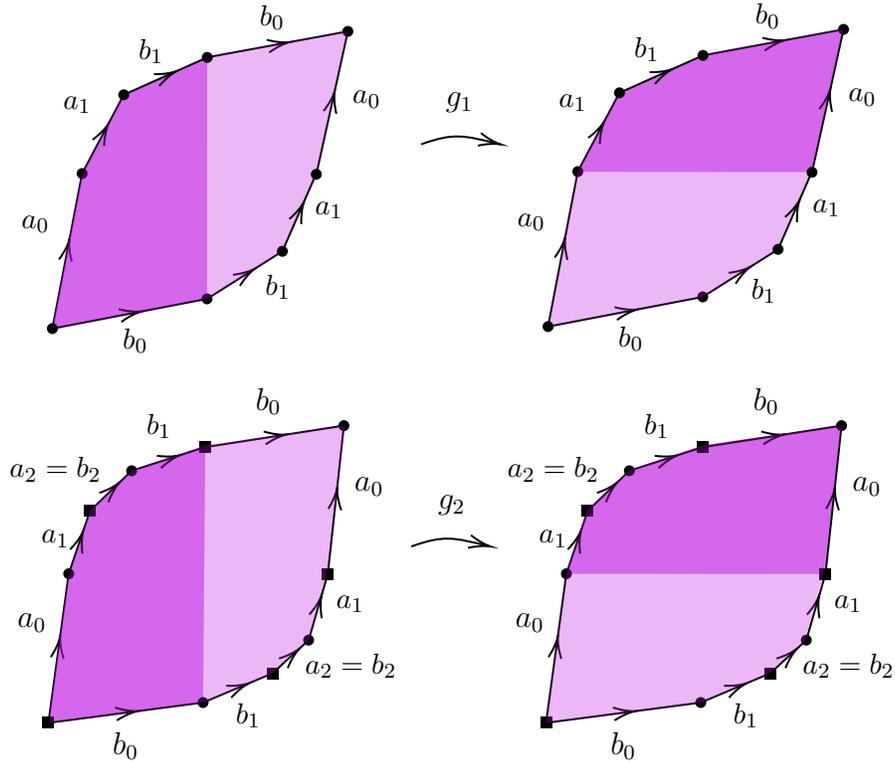

\begin{center}

\tikzset{every picture/.style={line width=0.75pt}} 


\end{center}
\caption{Two finite type translation surfaces along with psuedo-Anosov maps. Continuing this sequence approximates the end-periodic homeomorphism on the ladder surface constructed in \Cref{integers} for $d=2$. }
\label{fig:p5}
\end{figure}



We end with one final question. Recall that in order to guarantee our construction produced a surface of infinite-type we had to do a little more work in \Cref{infinite-type}. However, we do not know when this extra adjustment is necessary and when our surface was already of infinite type.  

\begin{question}
    Is it possible to completely characterize when our construction (or some variant of it) yields a surface of finite type? If not, are there sufficient conditions that guarantee the resulting surface has finite type?
\end{question} 

This seems to be a very hard question. For example, Baik--Rafiqi--Wu give sufficient conditions for a construction similar to ours to produce a finite-type surface equipped with a pseudo-Anosov homeomorphism \cite{BaikRafiqiWu2016}. However, their hypotheses are quite restrictive.


\bibliographystyle{alpha}
\bibliography{references}

\newcommand{\etalchar}[1]{$^{#1}$}
\begin{thebibliography}{DDH{\etalchar{+}}24}

\bibitem[Bow13]{bowman2013complete}
Joshua~P Bowman.
\newblock The complete family of {A}rnoux--{Y}occoz surfaces.
\newblock {\em Geometriae Dedicata}, 164(1):113--130, 2013.

\bibitem[BRW16]{BaikRafiqiWu2016}
Hyungryul Baik, Ahmad Rafiqi, and Chenxi Wu.
\newblock Constructing pseudo-{A}nosov maps with given dilatations.
\newblock {\em Geom. Dedicata}, 180:39--48, 2016.

\bibitem[Buc25]{Buckminster}
Ellis Buckminster.
\newblock Periodic points of endperiodic maps.
\newblock {\em Groups, Geometry, and Dynamics}, 2025.

\bibitem[CCF21]{CC-book}
John Cantwell, Lawrence Conlon, and Sergio~R. Fenley.
\newblock Endperiodic automorphisms of surfaces and foliations.
\newblock {\em Ergodic Theory Dynam. Systems}, 41(1):66--212, 2021.

\bibitem[CH08]{ChoHam}
Jin-Hwan Cho and Ji-Young Ham.
\newblock The minimal dilatation of a genus-two surface.
\newblock {\em Experimental Mathematics}, 17(3):257 -- 267, 2008.

\bibitem[Cha04]{Chamanara}
R.~Chamanara.
\newblock Affine automorphism groups of surfaces of infinite type.
\newblock In {\em In the tradition of {A}hlfors and {B}ers, {III}}, volume 355
  of {\em Contemp. Math.}, pages 123--145. Amer. Math. Soc., Providence, RI,
  2004.

\bibitem[DDH{\etalchar{+}}24]{DDHKOP}
Ryan Dickmann, George Domat, Thomas Hill, Sanghoon Kwak, Carlos Ospina, Priyam
  Patel, and Rebecca Rechkin.
\newblock Thurston's theorem: entropy in dimension one.
\newblock {\em Math. Res. Lett.}, 31(1):127--174, 2024.

\bibitem[FKLL25]{Endperiodic2}
Elizabeth Field, Autumn Kent, Christopher Leininger, and Marissa Loving.
\newblock A lower bound on volumes of end-periodic mapping tori.
\newblock {\em Journal of Topology}, 18(3):e70037, 2025.

\bibitem[Fri85]{Fried1985}
David Fried.
\newblock Growth rate of surface homeomorphisms and flow equivalence.
\newblock {\em Ergodic Theory and Dynamical Systems}, 5(4):539–563, 1985.

\bibitem[HS07]{HamSong}
Ji-Young Ham and Won~Taek Song.
\newblock The minimum dilatation of pseudo-{A}nosov 5-braids.
\newblock {\em Experimental Mathematics}, 16(2):167 -- 180, 2007.

\bibitem[Lin84]{Lind1984}
D.~A. Lind.
\newblock The entropies of topological markov shifts and a related class of
  algebraic integers.
\newblock {\em Ergodic Theory and Dynamical Systems}, 4(2):283–300, 1984.

\bibitem[LMT]{LandryMinskyTaylor2023}
Michael Landry, Yair Minsky, and Samuel~J. Taylor.
\newblock Endperiodic maps via pseudo-{A}nosov flows.
\newblock ar{X}iv:2304.10620.

\bibitem[LMT23]{LandryMinskyTaylor2022}
Michael Landry, Yair Minsky, and Samuel~J. Taylor.
\newblock Flows, growth rates, and the veering polynomial.
\newblock {\em Ergodic Theory and Dynamical Systems}, 43(9):3026–3107, 2023.

\bibitem[LT11]{Lanneau2011}
Erwan Lanneau and Jean-Luc Thiffeault.
\newblock On the minimum dilatation of pseudo-{A}nosov homeromorphisms on
  surfaces of small genus.
\newblock {\em Annales de l’institut Fourier}, 61(1):105--144, 2011.

\bibitem[LW25]{LovingWu2025}
Marissa Loving and Chenxi Wu.
\newblock A lower bound on end-periodic stretch factors.
\newblock {\em Proc. Amer. Math. Soc.}, 153(9):4071--4077, 2025.

\bibitem[SKL02]{SongKoLos}
Won~Taek Song, Ki~Hyoung Ko, and J\'{e}r\^{o}me~E. Los.
\newblock Entropies of braids.
\newblock {\em Journal of Knot Theory and Its Ramifications}, 11(04):647--666,
  2002.

\bibitem[Str17]{Strenner2017}
Bal\'azs Strenner.
\newblock Algebraic degrees of pseudo-{A}nosov stretch factors.
\newblock {\em Geometric and Functional Analysis}, 27:1497 -- 1539, 2017.

\bibitem[Thu88]{Thurston88}
William~P. Thurston.
\newblock On the geometry and dynamics of diffeomorphisms of surfaces.
\newblock {\em Bulletin (New Series) of the American Mathematical Society},
  19(2):417 -- 431, 1988.

\bibitem[Thu14]{Thurston2014}
William Thurston.
\newblock Entropy in dimension one, 2014.

\bibitem[TZ25]{TsangZeng}
Chi~Cheuk Tsang and Xiangzhuo Zeng.
\newblock Minimum dilatations of pseudo-{A}nosov braids, 2025.

\end{thebibliography}

\end{document}